\documentclass[a4paper, reqno]{article}

\usepackage[utf8]{inputenc}
\usepackage[T1]{fontenc}
\usepackage{unicode-math}
\usepackage{amsmath}
\usepackage{bbm}
\usepackage{amsthm}
\usepackage{amssymb}
\usepackage{color}
\usepackage[margin=1.2in]{geometry}
\usepackage[bookmarks = true]{hyperref}

\theoremstyle{plain}
\newtheorem{thm}{Theorem}[section]
\newtheorem{lem}[thm]{Lemma}
\newtheorem{prop}[thm]{Proposition}
\newtheorem{cor}[thm]{Corollary}

\theoremstyle{definition}
\newtheorem{rmk}[thm]{Remark}

\newtheorem{defn}[thm]{Definition}
\newtheorem{ex}[thm]{Example}
\newtheorem{conj}[thm]{Conjecture}
\newtheorem{construction}[thm]{Construction}

\numberwithin{equation}{section}

\newcommand{\overbar}[1]{\mkern 1.5mu\overline{\mkern-1.5mu#1\mkern-1.5mu}\mkern 1.5mu}

\newcommand{\ob}{\overbar}

\newcommand{\mr}{\mathrm}

\newcommand{\tif}{\ \text{if}\ }
\newcommand{\tand}{\ \text{and}\ }

\newcommand{\totherwise}{\ \text{otherwise}\ }

\newcommand{\Hom}{\operatorname{Hom}}
\newcommand{\End}{\operatorname{End}}

\newcommand{\Lie}{\operatorname{Lie}}
\newcommand{\Spec}{\operatorname{Spec}}

\newcommand{\Spf}{\operatorname{Spf}}

\newcommand{\tensor}{\otimes}
\newcommand{\os}{\overset}

\newcommand{\iso}{\cong}

\newcommand{\mbC}{\mathbb{C}}
\newcommand{\mbD}{\mathbb{D}}

\newcommand{\mbF}{\mathbb{F}}

\newcommand{\mbL}{\mathbb{L}}

\newcommand{\mbN}{\mathbb{N}}

\newcommand{\mbQ}{\mathbb{Q}}
\newcommand{\mbR}{\mathbb{R}}

\newcommand{\mbX}{\mathbb{X}}
\newcommand{\mbY}{\mathbb{Y}}
\newcommand{\mbZ}{\mathbb{Z}}

\newcommand{\mcA}{\mathcal{A}}

\newcommand{\mcC}{\mathcal{C}}
\newcommand{\mcD}{\mathcal{D}}

\newcommand{\mcF}{\mathcal{F}}

\newcommand{\mcL}{\mathcal{L}}

\newcommand{\mcN}{\mathcal{N}}
\newcommand{\mcO}{\mathcal{O}}
\newcommand{\mcP}{\mathcal{P}}

\newcommand{\mcW}{\mathcal{W}}
\newcommand{\mcX}{\mathcal{X}}
\newcommand{\mcY}{\mathcal{Y}}
\newcommand{\mcZ}{\mathcal{Z}}

\newcommand{\mfn}{\mathfrak{n}}

\newcommand{\mfs}{\mathfrak{s}}

\newcommand{\mfu}{\mathfrak{u}}

\begin{document}

\title{Relative unitary RZ-spaces and the Arithmetic Fundamental Lemma}
\author{Andreas Mihatsch}
\date{July 22, 2019}
\maketitle
\tableofcontents
\newpage

\section{Introduction}
In \cite{Zhang}, Wei Zhang introduces his so-called Arithmetic Fundamental Lemma conjecture (AFL).
This is a conjectural identity between certain derivatives of orbital integrals on $p$-adic\footnote{Throughout this work, we assume $p\neq 2$.} symmetric spaces and certain intersection products in unitary Rapoport-Zink spaces (RZ-spaces).
The AFL is proven in the case of dimension $n \leq 3$, see \cite{Zhang}.
In the subsequent work \cite{RTZ}, Rapoport, Terstiege and Zhang verify the AFL for arbitrary $n$ and so-called minuscule group elements $g$.
Their proof was later simplified by Li and Zhu \cite{LZ1, LZ2} and He, Li and Zhu \cite{HLZ}. 

In the present paper, we verify more cases of the AFL for arbitrary $n$ but under restrictive conditions on $g$. These computations rely on a certain recursion formalism which involves comparison isomorphisms between different RZ-spaces. More precisely, we will compare two PEL moduli problems, one for $p$-divisible groups and one for strict formal $\mcO$-modules. This comparison relies on the theory of display as developed by Zink \cite{Zink}, Lau \cite{Lauequiv} and Ahsendorf \cite{Ahs, ACZ}.

There is some resemblance of our comparison isomorphism with the one from Rapoport and Zink in the Drinfeld case, see \cite{RZDri}.
However, our moduli problems involve a polarization which adds an additional twist. The reason is that a polarization of a strict formal $\mcO$-module is not the same as a polarization of the underlying $p$-divisible group. We treat this problem in the appendix.

Let us briefly mention the following papers around the AFL. First, the AFL is related to an arithmetic Gan-Gross-Prasad conjecture which can be seen as a higher-dimensional generalization of the Gross-Zagier formula, see \cite{GGP}. We refer to the survey \cite{ZhangICM} and also to \cite{RSZdiag, Zhang} for these global aspects.
Second, the AFL from \cite{Zhang} is formulated for an unramified quadratic extension. See \cite{RSZ} and \cite{RSZatc} for variants in the ramified situation.

\subsection*{Part I: Relative unitary RZ-spaces}

We now describe our main results on unitary RZ-spaces. Let $E/E_0$ be an unramified quadratic extension of $p$-adic local fields with rings of integers $\mcO_{E_0}\subset \mcO_E$ and Galois conjugation $σ$. We denote by $\breve E$ the completion of a maximal unramified extension of $E$ with ring of integers $\mcO_{\breve E}$ and residue field $\mbF$.

\begin{defn}
Let $S$ be a scheme over $\Spf \mcO_{\breve E}$.\footnote{That is, an $\mcO_{\breve E}$-scheme such that $p$ is locally nilpotent in $\mcO_S$.} A \emph{hermitian $\mcO_E$-module} over $S$ is a triple $(X,ι,λ)$ where $X/S$ is a supersingular strict $\mcO_{E_0}$-module, $ι:\mcO_E→\End(X)$ an action and $λ:X\os{\sim}{→}X^\vee$ a compatible principal polarization, see Definition \ref{defHermmodule}.\\ 
The hermitian $\mcO_E$-module $(X,ι,λ)$ is \emph{of signature $(r,s)$} if, for all $a\in \mcO_E$,
$$\mr{charpol}(ι(a)\mid \Lie(X))(T) = (T-a)^r(T-σ(a))^s \in \mcO_S[T].$$
\end{defn}
It follows from Dieudonné theory that, up to quasi-isogeny, there is a unique hermitian $\mcO_E$-module $(\mbX_{E_0,(r,s)},ι_{\mbX},λ_{\mbX})$ of signature $(r,s)$ over $\mbF$.
\begin{defn}
For an $\mcO_{\breve E}$-scheme $S$, we denote by $\ob{S} := S\tensor_{\mcO_{\breve E}} \mbF$ its special fiber. Let $\mcN_{E_0,(r,s)}$ be the following set-valued functor on the category of schemes over $\Spf \mcO_{\breve E}$. To any $S$, it associates the set of isomorphism classes of quadruples $(X,ι,λ,ρ)$, where $(X,ι,λ)$ is a hermitian $\mcO_E$-module of signature $(r,s)$ and where
$$ρ:X\times_S\ob{S} → \mbX_{E_0,(r,s)}\times_{\Spec\mbF} \ob{S}$$
is an $E$-linear quasi-isogeny such that $ρ^*λ_{\mbX} = λ$.
\end{defn}

\begin{prop}[{\cite[Theorem 2.16]{RZ}}, Proposition \ref{propSmoothness}]\label{propIntro}
The functor $\mcN_{E_0,(r,s)}$ is representable by a formal scheme which is locally formally of finite type and formally smooth of dimension $rs$ over $\Spf \mcO_{\breve E}$.
\end{prop}
For arithmetic intersection theory and for the AFL, the case $(r,s)=(1,n-1)$ is of particular interest. The formal scheme $\mcN_{\mbQ_p,(1,n-1)}$ has been studied in detail by Vollaard-Wedhorn in \cite{VW}. Cho \cite{Cho} extended their results to the general case.
We remark that, if $E_0 = \mbQ_p$, then the moduli problem is of PEL-type in the sense of Rapoport and Zink, see \cite[Section 3.17]{RZ}. By contrast, if $E_0\neq \mbQ_p$, then this moduli problem is not covered by their book. This is due to the polarization $λ$, which is a polarization as a strict $\mcO_{E_0}$-module, see Definition \ref{defPol}. We call $\mcN_{E_0,(r,s)}$ a \emph{relative RZ-space} since the underlying moduli problem is formulated in strict $\mcO_{E_0}$-modules as opposed to $p$-divisible groups.

The last observation motivates our main result on unitary RZ-spaces, which we now state in a rather informal way. See Theorem \ref{mainThm1} for the precise statement.
\begin{thm}\label{thmMain1}
There exists an RZ-space $\mcN_{E_0/\mbQ_p,(r,s)}$ of PEL-type in the sense of \cite{RZ} together with an isomorphism
$$\mcN_{E_0/\mbQ_p,(r,s)} \iso \mcN_{E_0,(r,s)}$$
which is equivariant with respect to the unitary groups acting on both sides.
\end{thm}
In particular, the RZ-space $\mcN_{E_0/\mbQ_p,(r,s)}$ is smooth over $\Spf \mcO_{\breve E}$. This is remarkable since we do not impose any conditions on the ramification behavior of $E_0/\mbQ_p$. Instead, we impose a very specific Kottwitz condition for the moduli problem $\mcN_{E_0/\mbQ_p,(r,s)}$. Namely, the Kottwitz condition has to be induced from the maximal unramified intermediate field $\mbQ_p\subset E_0^u \subset E_0$ at all but possibly one place $ψ_0:E_0^u\hookrightarrow \breve E$, see Definition \ref{defSignature}. Our definition bears some similarity with the situation in \cite[Equation (2.1)]{RZDri}. But note that the unramified intermediate field does not play a role in loc.\! cit. Instead, the authors impose the \emph{Eisenstein condition} to get a regular moduli problem. A similar definition is made in \cite{KRZ}.

Finally, let us mention the following application of Theorem \ref{thmMain1}. The formal scheme $\mcN_{\mbQ_p, (1,n-1)}$ is well-known to uniformize the supersingular locus in certain Shimura varieties for unitary groups over $\mbQ$, see \cite[Section 5]{VW}. Essentially, this follows directly from the moduli description of the Shimura variety in terms of abelian varieties. For unitary groups over general totally real fields, it is the moduli description of $\mcN_{E_0/\mbQ_p,(1,n-1)}$ that naturally occurs in the uniformization. It follows from the Theorem that $\mcN_{E_0,(1,n-1)}$ can be used equivalently which gives the link between the AFL conjecture for $\mcN_{E_0,(1,n-1)}$ and the GGP conjecture for general totally real fields.


\subsection*{Part II: Application to the Arithmetic Fundamental Lemma}
We now describe the application of Theorem \ref{thmMain1} to the AFL. For this, we briefly recall the AFL conjecture in the \emph{inhomogeneous group} formulation from \cite{Zhang}. In the main text, we will also consider the AFL in the \emph{Lie algebra} formulation. We refer the reader to \cite{RSZ} for the \emph{homogeneous group} formulation.

Let us fix an integer $n\geq 2$ and let $W_0$ be an $(n-1)$-dimensional $E_0$-vector space. Set $W := E\tensor_{E_0} W_0$ and $V := W\oplus Eu$. We embed $GL(W)$ into $GL(V)$ as $h\mapsto \mathrm{diag}(h,1)$. In this way, $GL(W)$ acts by conjugation on $\End(V)$. An element $γ\in \End(V)$ is said to be \emph{regular semi-simple}, if its stabilizer for this action is trivial and if its orbit is Zariski-closed.

Let $S(E_0)$ denote the symmetric space
$$S(E_0):=\{γ\in \End(V)\mid γ\overbar{γ}=1\}.$$
It is stable under the action of $GL(W_0)$. We denote its regular semi-simple elements by $S(E_0)_{\rm{rs}}$ and form the set-theoretic quotient $[S(E_0)_{\mathrm{rs}}]:=GL(W_0)\backslash S(E_0)_{\mathrm{rs}}.$

For a regular semi-simple element $γ\in S(E_0)_{\mathrm{rs}}$, for a test function $f\in C^\infty_c(S(E_0))$ and for a complex parameter $s\in \mbC$, we define the \emph{orbital integral}
$$O_γ(f,s):=\int_{GL(W_0)} f(h^{-1}γh)η(\det h)|\det h|^s dh,$$
where $η:E_0^\times\rightarrow \{\pm 1\}$ is the quadratic character associated to $E/E_0$ by local class field theory and where $|\cdot | := q_{E_0}^{-v(\cdot)}$ is the normalized absolute value of $E_0$. We consider the special value $O_γ(f):= O_γ(f,0)$ and the \emph{derived orbital integral}
$$\partial O_γ(f):=\left.\frac{d}{ds}\right|_{s=0} O_γ(f,s).$$

Note that $O_γ(f)$ transforms with $η \circ \det$ under the action of $GL(W_0)$ on $γ$. The so-called \emph{transfer factor} $Ω(γ)\in \{\pm 1\}$, see Definition \ref{deftransfer}, is $η$-invariant as well, making the product $Ω(γ)O_γ(f)$ descend to the quotient $[S(E_0)_{\mathrm{rs}}]$.

Now let $J^\flat_0$ (resp.\! $J^\flat_1$) be a hermitian form with discriminant of even (resp.\! odd) valuation on $W$. For $i = 0,1$, we extend $J_i^\flat$ to a form $J_i$ on $V$ by defining $J_i(u,u) = 1$ and $u\perp W$. The unitary group $U(J^\flat_i)$ acts by conjugation on the regular semi-simple elements $U(J_i)_{\mathrm{rs}}$ and we define the quotient
$$[U(J_i)_{\mathrm{rs}}] := U(J^\flat_i)\backslash U(J_i)_{\mr{rs}}.$$

\begin{defn}
Two elements $δ\in U(J_i)_{\mathrm{rs}}$ and $γ\in S(E_0)_{\mathrm{rs}}$ are said to \emph{match}, if they are conjugate under $GL(W)$ within $\End(V)$.
\end{defn}

\begin{lem}[\protect{\cite[Lemma 2.3]{Zhang}}]\label{orbitcomparison}
The matching relation induces a bijection
$$α:[S(E_0)_{\mathrm{rs}}]\iso [U(J_0)_{\mathrm{rs}}]\sqcup [U(J_1)_{\mathrm{rs}}].$$
\end{lem}

Let $\mbY_{E_0}$ (resp.\! $\mbX_{E_0,(1,n-2)}$) be a hermitian $\mcO_E$-module over $\mbF$ of signature $(0,1)$ (resp.\! of signature $(1,n-2)$). Define $\mbX_{E_0,(1,n-1)} := \mbX_{E_0,(1,n-2)}\times \mbY_{E_0}$, which has signature $(1,n-1)$, and consider the associated RZ-spaces $\mcN_{E_0,(1,n-2)}$ and $\mcN_{E_0,(1,n-1)}$. Note that there is a unique deformation $\mcY_{E_0}$ of $\mbY_{E_0}$ to $\Spf \mcO_{\breve E}$ by Proposition \ref{propIntro} which defines a closed immersion
$$\begin{aligned}
δ:\mcN_{E_0,(1,n-2)}& → \mcN_{E_0,(1,n-1)}\\
X&\longmapsto X\times \mcY_{E_0}.\end{aligned}$$
Its image can be identified with the Kudla-Rapoport divisor $\mcZ(u)$ associated to the homomorphism $u := (0, \mr{id}):\mbY_{E_0}→\mbX_{E_0,(1,n-2)}\times \mbY_{E_0}$, see \cite{KRlocal}.

We can identify the group $\mr{Aut}(\mbX_{E_0,(1,n-2)})$ with $U(J^\flat_1)$. Then $U(J^\flat_1)$ acts on $\mcN_{E_0,(1,n-2)}$ by composition in the framing, $g.(X,λ,ι,ρ) = (X,λ,ι,gρ)$. Similarly, we may identify $\mr{Aut}(\mbX_{E_0,(1,n-1)})$ with $U(J_1)$ and we can even choose the identification in such a way that $δ$ becomes equivariant with respect to the embedding $U(J^\flat_1)\subset U(J_1)$.

\begin{defn}\label{defcycleintro}
(1) For an element $g\in U(J_1)$, we denote by $\mcZ(g)\subset \mcN_{E_0,(1,n-1)}$ the closed formal subscheme of $(X,ρ)$ with $ρ^{-1}gρ\in \End(X)$, see \cite[Proposition 2.9]{RZ}. It only depends on the $\mcO_E$-algebra spanned by $g$, $\mcO_E[g]\subset \End^0(\mbX_{E_0,(1,n-1)})$.\\
(2) An element $g\in U(J_1)$ is called \emph{artinian} if the intersection $\mr{Im}(δ)\cap \mcZ(g)$ is an artinian scheme.\\
(3) For artinian $g$, we define the \emph{intersection number}
$$\mr{Int}(g) := \mr{len}_{\mcO_{\breve E}} \mcO_{\mr{Im}(δ)\cap \mcZ(g)}.$$
\end{defn}

Zhang defines an intersection product more generally for all regular semi-simple elements $g\in U(J_1)_{\mr{rs}}$. Then the schematic intersection $\mr{Im}(δ)\cap \mcZ(g)$ may be higher-dimensional and higher Tor-terms appear, see Definition \ref{defHigherTor}. But note that the results of this paper only apply to the artinian case.

We are now ready to state the AFL conjecture for artinian elements. Let $Λ_0\subset W_0$ be some lattice and set $Λ := (Λ_0\tensor_{\mcO_{E_0}} \mcO_E)\oplus \mcO_Eu$. Define $S(\mcO_{E_0}):=S(E_0)\cap \End(Λ)$ and denote its characteristic function by $1_{S(\mcO_{E_0})}$.

\begin{conj}[AFL, \protect{\cite[Conjecture 2.9]{Zhang}}]
For every element $γ\in S(E_0)_{\mr{rs}}$ that matches an artinian element $g\in U(J_1)_{\mr{rs}}$, there is an equality
\begin{equation}\label{eqAFLIntro}
Ω(γ)\partial O_γ(1_{S(\mcO_{E_0})}) = -\mr{Int}(g)\log(q). \tag{$\mr{AFL}_{E/E_0,(V,J_1),u,g}$}
\end{equation}
\end{conj}
Here, the indexing quadruple $(E/E_0,(V,J_1),u,g)$ is chosen in such a way that it allows an unambiguous reconstruction of the terms involved in the identity $(\mr{AFL}_{E/E_0,(V,J_1),u,g})$.
Our main result on the AFL in the group version is a simplification of the AFL identity for elements $g$ of so-called inductive type, cf. Theorem \ref{thmMain3} below.
\begin{defn}
An element $g\in U(J_1)_{\mr{rs}}$ is \emph{of inductive type} if there exists a non-trivial étale algebra $A_0/E_0$ and, setting $A:=A_0\tensor_{E_0}E$, an inclusion
$$\mcO_A \subset \mcO_E[g]$$
that is equivariant for the Galois conjugation of $A/A_0$ and the Rosati involution on $\mcO_E[g]$.
\end{defn}
Let us fix such a $g$ and let us for simplicity also assume that $A$ is a field. Then $V$ becomes an $A$-vector space and we set $n' := \dim_A(V)$. Let $ϑ_A$ be a generator of the inverse different of $A_0/E_0$ and let $J^A_1$ be the unique $A$-valued hermitian form such that
$$J_1 = \mr{tr}_{A/E}(ϑ_AJ^A_1).$$
We assume for simplicity that $J_1^A(u,u)=1$, see Section \ref{subsect:mainfield} for variants.
\begin{thm}\label{thmMain3}
For $g$ of inductive type as above, there is an equivalence
$$(\mr{AFL}_{E/E_0,(V,J_1),u,g})\ \ \Leftrightarrow\ \ (\mr{AFL}_{A/A_0,(V,J^A_1), u, g}).$$
\end{thm}
Since the AFL has been proven for $n\leq 3$, we get the following corollary.
\begin{cor}
Let $g$ be of inductive type as above with $n'\leq 3$. Then the AFL identity for $g$,\\
$(\mr{AFL}_{E/E_0,(V,J_1),u,g}),$
holds.
\end{cor}
The proof of the Theorem relies on a comparison of the $g$-fixed points on $\mcN_{E_0,(1,n-1)}$ with the $g$-fixed points on $\mcN_{A_0,(1,n'-1)}$.
For better distinguishability, we write $g^A$ for the endomorphism $g$, viewed as $A$-linear endomorphism. We denote by $\mcZ(\mcO_A)\subset \mcN_{E_0,(1,n-1)}$ the closed formal subscheme of $(X,ρ)$ such that $ρ^{-1}\mcO_Aρ\subset \End(X)$. Then $\mcZ(g)\subset \mcZ(\mcO_A)$. We also let $f$ denote the inertia degree of $A/E$.
\begin{thm}\label{thmMain2}
There is an isomorphism of formal schemes,
$$\mcZ(\mcO_A) \iso \coprod_{i = 1}^f \mcN_{A_0,(1,n'-1)},$$
which is compatible with the formation of $\mcZ(g)$ in the following sense. It induces an identification
$$\mcZ(g) \iso \coprod_{i = 1}^f \mcZ(g^A).$$
\end{thm}
This result follows from the comparison isomorphism Theorem \ref{thmMain1}. Namely, the moduli description of the cycle $\mcZ(\mcO_A)$ corresponds precisely to an RZ-space $\mcN_{A_0/E_0,(1,n'-1)}$ which (the relative variant of) Theorem \ref{thmMain1} identifies with $\mcN_{A_0,(1,n'-1)}$.

In the main text, we also prove variants of Theorem \ref{thmMain3} in the Lie algebra version of the AFL which is more flexible than the group version. In this formalism, there are also situations where Theorem \ref{thmMain2} is applied to cycles in the smaller space $\mcN_{E_0,(1,n-2)}$, see Corollary \ref{corLieFlat}.
It is worth pointing out that all considered cases are corollaries of the Theorems \ref{thmAFLgenmain} and \ref{thmMainEtale}. These theorems are formulated uniformly for the AFL in the group and the Lie algebra version for artinian elements. We introduce this uniform treatment in Section \ref{sect:GeneralAFL}.

\section*{Acknowledgements}

I heartily thank my advisor M. Rapoport for suggesting to think about the AFL, for many helpful discussions and for the many remarks on earlier drafts of this paper. I also thank T. Zink for his notes on $\mcO$-displays and for his remarks on an earlier version of the appendix. Also, I thank D. Kirch, M. Morrow and P. Scholze for various helpful discussions.
I also thank the referee for several comments that helped improve the article.
This paper is the author's PhD thesis at the university of Bonn. I thank the Bonn International Graduate School for its financial support.

\part{Relative Unitary RZ-spaces}
\section{The moduli spaces $\mcN_{E_0/K,(r,s)}$}

In this chapter, we formulate a moduli problem of PEL-type where the PEL-datum is given over a finite extension $K$ of $\mbQ_p$. It generalizes the moduli problem of Vollaard and Wedhorn \cite{VW} in that it is also associated to the basic Frobenius-conjugacy class in a unitary group for an unramified quadratic extension. If $K=\mbQ_p$, then the moduli problem is a special case of the PEL-formalism of Rapoport and Zink \cite{RZ}. 

\subsection{Skew-hermitian $E$-$K$-modules}
\label{subsect:Set-up}
Let $p>2$ be a prime and fix finite extensions $\mbQ_p\subset K\subset E_0\subset E$ where $E/E_0$ is unramified quadratic. Let $d:=[E_0:K]$ with $d=ef$ where $e$ denotes the ramification index and $f$ the inertia degree. We denote the Galois conjugation of $E/E_0$ by $σ$ and the rings of integers by $\mcO_K\subset \mcO_{E_0}\subset \mcO_E$. We also fix a uniformizer $π_K\in \mcO_K$.

\begin{defn}\label{defshEMod}
A \emph{skew-hermitian $E$-$K$-module} $(V,\langle\ ,\ \rangle)$ is an $E$-vector space together with a perfect alternating $K$-bilinear pairing $\langle\ ,\ \rangle:V\times V→K$ such that $\langle a\ ,\ \rangle = \langle\ ,a^σ\ \rangle$ for all $a\in E$.
An isomorphism of skew-hermitian $E$-$K$-modules $(V,\langle\ ,\ \rangle)$ and $(V',\langle\ ,\ \rangle)$ is an $E$-linear isometry $V\iso V'$. We denote by $U(V)$ the group of automorphisms of $(V,\langle\ ,\ \rangle)$.
\end{defn}

For every $n$, there exist two isomorphism classes of skew-hermitian $E$-$K$-modules $(V,\langle\ ,\ \rangle)$ of dimension $n$. We say that $V$ is \emph{even} if there exists a self-dual $\mcO_E$-lattice in $V$. Otherwise we call $V$ \emph{odd}. This distinguishes the two isomorphism classes. Note that $(V,\langle\ ,\ \rangle)$ is even (resp.\! odd) if and only if the index $[M^\vee:M]$ is even (resp.\! odd) for every $\mcO_K$-lattice $M\subset V$. Here, $M^\vee := \{v\in V\mid \langle M,v\rangle \in \mcO_K\}$ is the lattice dual to $M$.

The category of skew-hermitian $E$-$K$-modules is endowed with the \emph{adjoint involution} $*$. If $V_1$ and $V_2$ are two such modules, then this is the isomorphism
$$\begin{aligned}
*:\Hom_E(V_1,V_2)&\overset{\iso}{→} \Hom_E(V_2,V_1)\\
f &\longmapsto f^* : V_2 \iso V_2^\vee \overset{f^\vee}{→} V_1^\vee \iso V_1\end{aligned}$$
where the identifications $V_1 \iso V_1^\vee$ and $V_2 \iso V_2^\vee$ are induced by the alternating pairings.

\subsection{Hermitian $\mcO_E$-$\mcO_K$-modules}
As usual, $\breve K$ denotes the completion of a maximal unramified extension of $K$. We denote by $E^u\subset E$ the maximal subfield which is unramified over $K$ and define $Ψ:=\Hom_{K}(E^u,\breve K)$. We choose a decomposition $Ψ=Ψ_0\sqcup Ψ_1$ such that $σ(Ψ_0)=Ψ_1$ and we fix an element $ψ_0\in Ψ_0$. Finally, we define $\breve E:= E\tensor_{E^u,ψ_0}\breve K$ which is the completion of a maximal unramified extension of $E$.

We denote the ring of integers in $\breve K$ (resp.\! in $\breve E$) by $\mcO_{\breve K}$ (resp.\! $\mcO_{\breve E}$). Let $\mbF$ be their residue field and let $x\mapsto {}^Fx$ denote the Frobenius on $\breve K$. There is a natural identification
\begin{equation}\label{eqGrading}
\mcO_E\tensor_{\mcO_K} \mcO_{\breve K} = \prod_{ψ\in Ψ} \mcO_{\breve E}
\end{equation}
such that the Frobenius $1\tensor {}^F$ is homogeneous and acts simply transitive on the indexing set.

In the following, the notions of height, slope, Dieudonné crystal, polarization (with respect to the fixed uniformizer $π_K$) etc. are always used in the relative sense for strict formal $\mcO_K$-modules. We recall some of these definitions in Section \ref{sect:strmod}.
\begin{defn}\label{defHermmodule}
Let $S$ be a scheme over $\Spf \mcO_{\breve K}$. A \emph{hermitian $\mcO_E$-$\mcO_K$-module over $S$} is a triple $(X,ι,λ)$ where $X/S$ is a supersingular\footnote{Meaning that the slopes of $X$ are $1/2$.} strict formal $\mcO_K$-module together with an action $ι:\mcO_E→\End(X)$ and a principal polarization $λ:X\overset{\sim}{→} X^\vee$ such that
$$λ^{-1}ι(a)^\vee λ = ι(a^σ).$$
An \emph{isomorphism} (resp.\! \emph{quasi-isogeny}) of two hermitian $\mcO_E$-$\mcO_K$-modules $(X,ι,λ)$ and $(X',ι',λ')$ is an $\mcO_E$-linear isomorphism (resp.\! quasi-isogeny) $µ:X→X'$ of the underlying strict formal $\mcO_K$-modules such that $µ^*λ' = λ$.
The hermitian $\mcO_E$-$\mcO_K$-module $(X,ι,λ)$ is \emph{of rank $n$} if the height of $X$ is $2nd$. This implies $\dim X = nd$.
By \emph{hermitian $\mcO_E$-module}, we mean a hermitian $\mcO_E$-$\mcO_{E_0}$-module.
\end{defn}

\begin{defn}\label{defRosati}
The category of hermitian $\mcO_E$-$\mcO_K$-modules over a scheme $S$ is endowed with the \emph{Rosati involution} $*$. If $(X_1,ι_1,λ_1)$ and $(X_2,ι_2,λ_2)$ are two such modules, then this is defined as the isomorphism
$$\begin{aligned}
*:\Hom_{\mcO_E}(X_1,X_2) &\overset{\iso}{→}\Hom_{\mcO_E}(X_2,X_1)\\
f & \longmapsto f^* := λ_1^{-1}\circ f^\vee \circ λ_2.\end{aligned}$$
\end{defn}

\begin{defn}
A \emph{skew-hermitian $E$-$K$-isocrystal} is a tuple $(N,\langle\ ,\ \rangle,F,ι)$ where $N$ is a finite $\breve K$-vector space, $\langle\ ,\ \rangle:N\times N→\breve K$ is an alternating perfect pairing, $F:N→N$ is an ${}^F$-linear isomorphism with all slopes $1/2$ such that $\langle F\ ,F\ \rangle = π_K{}^F\langle\ ,\ \rangle$ and $ι:E→\End(N,F)$ is an action of $E$ such that $\langle a\ ,\ \rangle = \langle \ ,a^σ\ \rangle$ for all $a\in E$. For two skew-hermitian $E$-isocrystals $N_1,N_2$, the adjoint involution $*:\Hom_E(N_1,N_2)\iso \Hom_E(N_2,N_1)$ is defined as in the case of skew-hermitian $E$-$K$-modules.
\end{defn}
By Dieudonné theory, the category of hermitian $\mcO_E$-$\mcO_K$-modules up to quasi-isogeny over $\Spec \mbF$ is equivalent to the category of skew-hermitian $E$-$K$-isocrystals.

\begin{prop}\label{propIsocrys}
There is an equivalence of categories
$$
\begin{aligned}
&\{\text{skew-hermitian }E\text{-}K\text{-modules }(V,\langle\ ,\ \rangle)\}\\
\iso\ &\{\text{skew-hermitian }E\text{-}K\text{-isocrystals }(N,\langle\ ,\ \rangle,F,ι)\}.
\end{aligned}
$$
that is compatible with the adjoint involutions on both sides. In particular for a given rank $n$, there are precisely two hermitian $\mcO_E$-$\mcO_K$-modules over $\mbF$ up to quasi-isogeny.
\end{prop}

\begin{defn}
A skew-hermitian $E$-$K$-isocrystal is called \emph{even} (resp.\! \emph{odd}) if it corresponds to an even (resp.\! odd) skew-hermitian $E$-$K$-module under the above equivalence of categories.
\end{defn}

\begin{proof}
Given a skew-hermitian $E$-$K$-module $(V,\langle\ ,\ \rangle)$, the associated skew-hermitian $E$-$K$-isocrystal is defined as follows. Let $N:=V\tensor_K\breve K$ be the scalar extension and extend both the pairing $\langle\ ,\ \rangle$ and the $E$-action in the $\breve K$-(bi)linear way to $N$. Note that $N$ is a module over $E\tensor_K \breve K$ and hence graded according to \eqref{eqGrading},
$$N = \prod_{ψ\in Ψ} N_ψ.$$
The pairing satisfies
\begin{equation}
\langle\ ,\ \rangle\vert_{N_ψ\times N_{ψ'}} \equiv 0\ \ \ \tif\ ψ'\neq ψσ \tag{$*$}
\end{equation}
and the ${}^F$-linear operator $α:=\mr{id}_V\tensor {}^F$ is homogeneous in the sense that $α(N_ψ)=N_{{}^Fψ}$. Let $α_ψ:N_ψ→N_{{}^Fψ}$ denote the $ψ$-component of $α$ and
define a supersingular Frobenius $F:=\prod F_ψ$ on $N$ as
\begin{equation}\label{eqFrob}
[F_ψ:N_ψ→N_{{}^Fψ}]:=\begin{cases} π_Kα_ψ & \tif ψ\in Ψ_0\\
                                  α_ψ & \tif ψ\in Ψ_1.\end{cases}
\end{equation}
Then $(N,F)$ is supersingular since $F^{2f} = π_K^fα^{2f}$. In particular, there exists a $\mcO_{\breve K}$-lattice $M\subset N$ such that $F^{2f}M = π_K^fM$. 
Furthermore, the previously defined pairing $\langle\ ,\ \rangle$ is a polarization of $(N,F)$ since, for $(x,y)\in N_ψ\times N_{ψσ}$,
$$\langle Fx,Fy \rangle = π_K\langle αx,αy\rangle = π_K{}^F(\langle x,y\rangle).$$
(For the first equality, we used that precisely one out of $\{ψ,ψσ\}$ lies in $Ψ_0$.)

The $E$-action on $N$ is compatible with $\langle\ ,\ \rangle$ and commutes with $F$, so $(N,\langle\ ,\ \rangle, ι, λ)$ defines a skew-hermitian $E$-$K$-isocrystal. Since the scalar extension from $K$ to $\breve K$ is functorial and commutes with taking duals, this defines a functor as asserted in the proposition.
To prove that it is an equivalence, we give the inverse construction.

For a skew-hermitian $E$-$K$-isocrystal $(N,\langle\ ,\ \rangle,F,ι)$, define the ${}^F$-linear operator $α=\prod α_ψ$ as
\begin{equation}\label{eqalpha}
[α_ψ:N_ψ→N_{{}^Fψ}]:=\begin{cases} π_K^{-1}F_ψ & \tif ψ\in Ψ_0\\
                                  F_ψ & \tif ψ\in Ψ_1.\end{cases}
\end{equation}
Set $V:=N^{α=1}$ and restrict the form $\langle\ ,\ \rangle$ to $V$. Note that $α$ is isoclinic of slope $0$ since $N$ is supersingular. Relation $(*)$ holds for any skew-hermitian $E$-$K$-isocrystal and thus
$$\langle α\ ,α\ \rangle = π_K^{-1}\langle F\ ,F\ \rangle = {}^F\langle\ ,\ \rangle.$$
So the form $\langle\ ,\ \rangle\vert_{V}$ takes values in $K =\breve K^{{}^F=1}$. Finally, the $E$-action commutes with $α$ and hence $E$ acts on $V$. Then $(V,\langle\ ,\ \rangle)$ defines a skew-hermitian $E$-$K$-module with $V\tensor_K \breve K \iso N$.
\end{proof}

\subsection{Moduli of hermitian $\mcO_E$-$\mcO_K$-modules}
Let $r,s\in \mbZ_{\geq 0}$ and set $n:=r+s$.

\begin{defn}
For $a\in E$, we define the following polynomials.
$$
\begin{aligned}
P_{(0,1)}(a;t) &:=\prod_{ψ\in Ψ_1} ψ(\mr{charpol}_{E/E^u}(a;t))\ \ &\in \breve K[t].\\
P_{(1,0)}(a;t) &:=P_{(0,1)}(a;t)(t-a)(t-a^σ)^{-1}\ \ &\in \breve E[t].\\
P_{(r,s)}(a;t) &:= P_{(1,0)}(a;t)^rP_{(0,1)}(a;t)^s\ \ &\in \breve E[t].\\
\end{aligned}$$
\end{defn}

If $X$ is a hermitian $\mcO_E$-$\mcO_K$-module over a $\Spf \mcO_{\breve K}$-scheme $S$, then its Lie algebra is $Ψ$-graded,
\begin{equation}\label{eqGradedLie}
\Lie(X) = \bigoplus_{ψ\in Ψ} \Lie_ψ(X),
\end{equation}
where $\Lie_ψ(X)$ is the direct summand on which $\mcO_{E^u}$ acts via the embedding $ψ:\mcO_{E^u}→\mcO_{\breve K}$. By definition
$$\mcO_{\breve E} = \mcO_E\tensor_{\mcO_{E^u},ψ_0} \mcO_{\breve K}$$
and we consider any $\Spf \mcO_{\breve E}$-scheme as an $\mcO_E$-scheme via the first and as a $\mcO_{\breve K}$-scheme via the second projection.

\begin{defn}\label{defSignature}
Let $S$ be a scheme over $\Spf \mcO_{\breve E}$. A hermitian $\mcO_E$-$\mcO_K$-module $(X,ι,λ)$ of rank $n$ over $S$ is \emph{of signature $(r,s)$} if the following two conditions hold.\vspace{-2mm}
\begin{itemize}
\item[(i)] $\mr{charpol}(ι(a)\mid \Lie(X);t) = P_{(r,s)}(a;t)\ \ \ \forall\ a\in \mcO_E.$
\item[(ii)] $(ι(a) - a)) \mid_{\Lie_{ψ_0}(X)} = 0\ \ \ \forall\ a \in \mcO_E.$
\end{itemize}
Here in (i), we view $P_{(r,s)}(a;t)$ as element of $\mcO_S[t]$ via the structure morphism. Condition (ii) means that $\mcO_E$ acts on $\Lie_{ψ_0}(X)$ via the structure morphism.
\end{defn}
\begin{rmk}
In the case $\mbQ_p = K = E_0$, our definition of signature agrees with the one from Vollaard-Wedhorn \cite{VW}. Moreover in the case of an unramified extension $E_0/K$, condition (ii) is automatically satisfied.
\end{rmk}
\begin{lem}\label{lemSign}
A hermitian $\mcO_E$-$\mcO_K$-module $(X,ι,λ)/S$ is of signature $(r,s)$ if and only if it satisfies (ii) from Definition \ref{defSignature} and the following rank condition. \vspace{-2mm}
\begin{itemize}
\item[(i')] The ranks of the summands in Equation \eqref{eqGradedLie} are as follows:
$$
\mr{rk}_{\mcO_S} \Lie_{ψ}(X) = \begin{cases} 0 & \tif ψ\in Ψ_0\setminus\{ψ_0\}\\
                                           r & \tif ψ = ψ_0\\
                                           ne & \tif ψ\in Ψ_1\setminus\{ψ_0σ\}\\
                                           ne - r & \tif ψ=ψ_0σ.
                                           \end{cases}
$$
\end{itemize}
\end{lem}
To prove the lemma, we first introduce the so-called local model. Let $D$ be the $\mcO_{\breve E}$-module $D:=(\mcO_E\tensor_{\mcO_K}\mcO_{\breve E})^n$ and let $\langle\ ,\ \rangle:D\times D→\mcO_{\breve E}$ be a perfect alternating $\mcO_{\breve E}$-bilinear pairing such that $\langle a\tensor 1\ ,\ \rangle = \langle\ ,a^σ\tensor 1\ \rangle$ for all $a\in \mcO_E$. The pair $(D,\langle\ ,\ \rangle)$ is unique up to isomorphism since the form is automatically split in the sense that $D$ has a grading $D=\prod_{ψ\in Ψ} D_ψ$ such that $\langle D_ψ,D_{ψ'}\rangle = 0$ if $ψ'\neq ψσ$. Let $Gr_{dn}(D)→\Spec \mcO_{\breve E}$ be the Grassmannian of $dn$-dimensional subspaces of $D$, viewed as $\mcO_{\breve E}$-module. Denote by $G\subset Gr_{dn}(D)$ the closed subscheme of $\mcO_E$-stable isotropic subspaces,
$$G(S):=\left\{\mcF \in Gr_{dn}(D)(S)\ \mid\ \mcF^\perp = \mcF,\ \mcO_E\cdot \mcF=\mcF \right\}.$$
\begin{defn}
The \emph{local model} is the closed subfunctor $M_{E_0/K,(r,s)}^\mr{loc}\subset G$ of those $\mcF$ such that the $\mcO_E$-action on the quotient $Q(\mcF) := D\tensor_{\mcO_{\breve E}}\mcO_S/\mcF$ satisfies the conditions (i) and (ii) (where $\Lie(X)$ is replaced by $Q(\mcF)$).
\end{defn}
\begin{prop}\label{propCompLMs}
The local model $M_{E_0/K,(r,s)}^\mr{loc}$ is smooth over $\Spec \mcO_{\breve E}$ of relative dimension $rs$. It can be equivalently described as the subfunctor $M'\subset G$ of those $\mcF$ such that the $\mcO_E$-action on $Q(\mcF)$ satisfies the conditions (i') and (ii) (again with $\Lie(X)$ replaced by $Q(\mcF)$).
\end{prop}
\begin{proof}
We claim that $M'\subset G$ is a closed subscheme, isomorphic to
$$Gr_r(D_{ψ_0}\tensor_{\mcO_E\tensor_{\mcO_{E^u},ψ_0}\mcO_{\breve E}} \mcO_{\breve E})→\Spec \mcO_{\breve E}.$$
Indeed, a subbundle $\mcF = \prod_{ψ\in Ψ}\mcF_ψ\in G(S)$ lies in $M'(S)$ if and only if
$$\begin{cases} \begin{aligned}& \mcF_ψ = D_ψ\tensor_{\mcO_{\breve E}}\mcO_S & \text{ if }&ψ\in Ψ_0\setminus\{ψ_0\}\\
				& \mcF_ψ = 0 & \text{ if }&ψ\in Ψ_1\setminus\{ψ_0σ\}\end{aligned}\\
				\mcF_{ψ_0} \text{ rank }r\text{ quotient of } D_{ψ_0}\tensor_{\mcO_E\tensor_{\mcO_{E^u},ψ_0}\mcO_S}\mcO_S\\
				\mcF_{ψ_0σ} = \mcF_{ψ_0}^\perp. \end{cases}$$
In particular, $M'$ is integral and smooth over $\mcO_{\breve E}$. Evaluating (i) on elements of $\mcO_{E^u}$ shows that $M_{E_0/K,(r,s)}^\mr{loc}\subset M'$. Moreover, their generic fibers are equal which implies equality.
\end{proof}
\emph{Proof of Lemma \ref{lemSign}.}
Let $\mbD_X$ be the covariant Dieudonné crystal of $X$ on the $\mcO_K$-crystalline site of $S$ (see \cite[Chapter 3]{ACZ}). It comes with an action of $\mcO_E$ and a skew-hermitian perfect pairing $\mbD_X\times \mbD_X→\mcO_{S_{crys}}$.
\begin{lem}\label{lemLocFree}
For any $\mcO_K$-pd-thickening $S\hookrightarrow S'$, $\mbD_X(S')$ is locally on $S'$ isomorphic to $D\tensor_{\mcO_{\breve E}} \mcO_{S'}$. 
\end{lem}
\begin{proof}
The crucial point is to show that $\mbD_X(S')$ is locally free over $\mcO_E\tensor_{\mcO_K}\mcO_{S'}$. If this is done, then one can construct such an isomorphism e.g. by choosing $\mcO_E\tensor_{\mcO_{E^u},ψ}\mcO_{S'}$-bases of the $\mbD_X(S')_ψ$ for $ψ\in Ψ_0$ and taking the dual bases of the $\mbD_X(S')_ψ$ for $ψ\in Ψ_1$, at least locally on $S'$.

The sheaf $\mbD_X(S')$ is locally free of rank $2nd$ as $\mcO_{S'}$-module. For $s\in S$ with residue field $κ(s)$, there is a canonical identification
$$\mbD_X(S')\tensor_{\mcO_{S'}}κ(s)\iso \mbD_{X\times_{S}\Spec κ(s)}(\Spec κ(s)).$$
It follows from considering the Dieudonné module for the base change to the perfection $X\times_S\Spec κ(s)^{\mr{perf}}$ that this fiber is free of rank $n$ over $\mcO_E\tensor_{\mcO_K}κ(s)$. Lifting generators in a neighborhood $U'\subset S'$ of $s$ yields a map
$$(\mcO_E\tensor_{\mcO_K}\mcO_{U'})^n→\mbD_X(U')$$
which is surjective in the fiber over $s$. Since both modules are also locally free of the same rank over $\mcO_{U'}$, the map is an isomorphism in a neighborhood of $s$.
\end{proof}
We now consider the Hodge filtration $\mcF_X\subset \mbD_X(S)$. Working locally on $S$, we choose an isomorphism as in Lemma \ref{lemLocFree}. Then $\mcF_X$ defines an element in $G(S)$. Applying Proposition \ref{propCompLMs} finishes the proof.\qed

\begin{lem}\label{lemExistframing}
Consider the quasi-isogeny class of a hermitian $\mcO_E$-$\mcO_K$-module over $\mbF$ corresponding to a skew-hermitian $E$-$K$-module $(V,\langle\ ,\ \rangle)$ of rank $n$.\\
There exists a formal hermitian $\mcO_E$-$\mcO_K$-module $(X,ι,λ)$ of signature $(r,s)$ in this class if and only if the parity of $V$ coincides with the parity of $r$.
\end{lem}
We prepare the proof with the following lemma, where the operator $α$ is as in Equation \eqref{eqalpha}.
\begin{lem}\label{lemExistCan}
Let $V$ be a skew-hermitian $E$-$K$-module and let $N:=V\tensor_{K}\breve K$ be the induced skew-hermitian $E$-$K$-isocrystal. Then there is a bijection
$$
\big\{\text{self-dual }\mcO_E\text{-lattices }Λ\subset V\big\} \iso
\left\{\begin{gathered}\text{self-dual }\mcO_E\text{-stable Dieudonné-lattices }\\
M\text{ in } N\text{ of signature }(0,n)\end{gathered}\right\}.$$
given by $Λ\mapsto Λ\tensor_{\mcO_K}\mcO_{\breve K}$ and $M\mapsto M^{α=\mr{id}}$.
\end{lem}
\begin{proof}
The map $Λ\mapsto M(Λ) := Λ\tensor_{\mcO_K}\mcO_{\breve K}$ is an injective map from self-dual $\mcO_E$-lattices to self-dual $\mcO_E$-stable $\mcO_{\breve K}$-lattices. Each $M(Λ)$ is $F$-stable and of signature $(0,n)$ by definition of $F$, see \eqref{eqFrob}. Conversely, a Dieudonné-lattice $M$ of signature $(0,n)$ is stable under $α$.
\end{proof}
\emph{Proof of Lemma \ref{lemExistframing}.}
We first prove the existence of hermitian $\mcO_E$-$\mcO_K$-modules for all signatures. It is enough to do this in the cases of signature $(0,1)$ and $(1,0)$. Taking direct products then settles the general case. The case of signature $(0,1)$ is taken care of by Lemma \ref{lemExistCan}, so we are left with the case $(1,0)$.

Let $(V,\langle\ ,\ \rangle)$ be any odd skew-hermitian $E$-$K$-module of rank $1$. Let $(N,F,ι,\langle\ ,\ \rangle)$ be the associated isocrystal. Fix a uniformizer $π_E\in E$ and an $\mcO_E$-lattice $Λ\subset V$ such that $Λ^\vee = π_E^{-1}Λ$.
Let $M:=Λ\tensor_{\mcO_K}\mcO_{\breve K}$ be the associated $\mcO_E$-stable $\mcO_{\breve K}$-lattice. It decomposes as
$$M = \bigoplus_{ψ\in Ψ} M_ψ.$$
Define $M':= \bigoplus_{ψ\in Ψ} M'_ψ$ as
$$M'_{ψ} :=\begin{cases} M_ψ & \tif ψ\in \{{}^Fψ_0,{}^{F^2}ψ_0,\ldots,{}^{F^f}ψ_0\}\\
                         π_E^{-1}M_{ψ} & \totherwise\!.\end{cases}$$
This lattice is self-dual, stable under $\mcO_E$, stable under $F$, stable under $π_KF^{-1}$ and of signature $(1,0)$, so existence is proved.

Now let $(X,ι,λ)/\mbF$ of signature $(r,s)$ be given. Let $N=\prod_{ψ\in Ψ} N_ψ$ be the isocrystal of $X$ together with the alternating pairing $\langle\ ,\ \rangle$ induced by $λ$. For any $\mcO_{\breve E}$-lattice $L_{ψ_0}\subset N_{ψ_0}$, we denote by $L_{ψ_0}^\vee\subset N_{ψ_0σ}$ the dual lattice with respect to the form $\langle\ ,\ \rangle$.

Let $X$ correspond to the $\mcO_{\breve E}$-lattice $M\subset N$. It is self-dual and thus $M_{ψ_0}^\vee = M_{ψ_0σ}$. The signature condition for $M$ implies that
$$[α^fM_{ψ_0}:M_{ψ_0}^\vee] = r.$$
It follows that for every $\mcO_{\breve E}$-lattice $L_{ψ_0}\subset N_{ψ_0}$, $[α^fL_{ψ_0}:L_{ψ_0}^\vee] \equiv r$ mod $2$. By Proposition \ref{propIsocrys}, there are precisely two quasi-isogeny classes of formal hermitian $\mcO_E$-$\mcO_K$-modules over $\mbF$. In particular, formal hermitian $\mcO_E$-$\mcO_K$-modules $X$ and $X'$ over $\mbF$ of signatures $(r,s)$ and $(r',s')$ respectively are quasi-isogeneous if and only $r\equiv r'$ mod $2$. By Lemma \ref{lemExistCan}, $r$ and $V$ have the same parity.
\qed

For every signature $(r,s)$, we fix a hermitian $\mcO_E$-$\mcO_K$-module $\mbX_{E_0/K,(r,s)}$ over $\mbF$ of that signature.\footnote{It would be enough to fix any triple $(\mbX,ι,λ)$ quasi-isogeneous to a hermitian $\mcO_E$-$\mcO_K$-module of signature $(r,s)$.} 
\begin{defn}\label{defRZ}
Let $\mcN_{E_0/K,(r,s)}$ denote the following set-valued functor on schemes over $\Spf \mcO_{\breve E}$. It associates to $S$ the set of isomorphism classes of quadruples $(X,ι,λ,ρ)$ where $(X,ι,λ)$ is a hermitian $\mcO_E$-$\mcO_K$-module of signature $(r,s)$ over $S$ and where
$$ρ:X\times_S \overbar{S}→\mbX_{E_0/K,(r,s)}\times_{\Spec \mbF} \overbar{S}$$
is a quasi-isogeny of hermitian $\mcO_E$-$\mcO_K$-modules. Here, $\overbar S$ denotes the special fiber $\overbar S = S \times_{\Spf \mcO_{\breve E}} \Spec \mbF$. The quasi-isogeny $ρ$ is called the \emph{framing} and $\mbX_{E_0/K,(r,s)}$ is called the \emph{framing object}.
As a matter of notation, we set
$$\mcN_{E_0,(r,s)} := \mcN_{E_0/E_0,(r,s)}.$$
\end{defn}

\begin{prop}\label{propSmoothness}
The functor $\mcN_{E_0/K,(r,s)}$ is representable by a formal scheme which is locally formally of finite type over $\Spf \mcO_{\breve E}$ and formally smooth over $\Spf\mcO_{\breve E}$ of relative dimension $rs$.
\end{prop}
\begin{proof}
The representability is proven in \cite{RZ}. We prove the formal smoothness which follows from the Grothendieck-Messing Theorem. Consider a point $X\in \mcN_{E_0/K,(r,s)}(S)$ and an $\mcO_K$-pd-thickening $S\hookrightarrow S'$ (e.g. a square-zero thickening). Working locally on $S'$, we choose an $\mcO_E$-linear isometry of $\mbD_X(S')$ with $D\tensor_{\mcO_{\breve E}}\mcO_{S'}$ as in Lemma \ref{lemLocFree}. Then deforming $X$ becomes equivalent to lifting the Hodge filtration
$$\left[\mcF_X\subset \mbD_X(S)\right]\in M_{E_0/K,(r,s)}^{\mr{loc}}(S)$$
to $S'$. This deformation problem is formally smooth of rank $rs$ over $\mcO_{\breve E}$ by Proposition \ref{propCompLMs}.
\end{proof}

\section{Comparison of moduli spaces}
Let $(V,\langle\ ,\ \rangle)$ be a skew-hermitian $E$-$\mbQ_p$-module as in Definition \ref{defshEMod}. Let $K$ be an intermediate field $\mbQ_p\subset K\subset E_0$ and choose a generator $ϑ_K$ of the inverse different of $K/\mbQ_p$. There exists a unique non-degenerate $K$-bilinear alternating form
$$\langle\ ,\ \rangle_K:V\times V→K$$
such that $\mr{tr}_{K/\mbQ_p}(ϑ_K \langle\ ,\ \rangle_K) = \langle\ ,\ \rangle$. It is still $E$-hermitian in the sense that
$$\langle a\ ,\ \rangle_K = \langle\ ,a^{σ}\ \rangle_K,\ \ \ a\in E.$$
So $(V,\langle\ ,\ \rangle_K)$ is a skew-hermitian $E$-$K$-module and the groups of $E$-linear isometries of $\langle\ ,\ \rangle$ and $\langle\ ,\ \rangle_K$ are then identical. Note that a lattice $Λ\subset V$ is self-dual with respect to the lifted form $\langle\ ,\ \rangle_K$ if and only if it is self-dual for the original form $\langle\ ,\ \rangle$.

For every such intermediate field $K$, we fix a uniformizer $π_K\in \mcO_K$ in order to talk about polarizations of strict formal $\mcO_K$-modules, see Remark \ref{rmkDualModule}. We make $N:=(V,\langle\ ,\ \rangle_K)\tensor_K\breve {K}$ into a polarized $K$-isocrystal as in the proof of Proposition \ref{propIsocrys}. By Lemma \ref{lemExistframing}, if $r$ and $V$ have the same parity, then $V$ gives rise to a whole family of framing objects
$$\{\mbX_{E_0/K,(r,s)}\}_{\mbQ_p\subset K\subset E_0}$$
which all come with an action (by quasi-isogenies) of the unitary group $U(V)$. Our main result in this section is that the corresponding RZ-spaces are all isomorphic.

\begin{thm}\label{mainThm1}
Let $(V,\langle\ ,\ \rangle)$ be a skew-hermitian $E$-module and let $r$ have the same parity as $V$. For any intermediate field $\mbQ_p\subset K\subset E_0$, there is a $U(V)$-equivariant isomorphism
$$c:\mcN_{E_0/K,(r,s)}\iso \mcN_{E_0,(r,s)}.$$
In particular, the formal scheme $\mcN_{E_0/K,(r,s)}$ is independent of the choice of the decomposition $Ψ = Ψ_0\sqcup Ψ_1$.
\end{thm}

The proof relies on the following equivalence of categories. We consider the category $\textsl{Sch}/\Spf \mcO_{\breve E}$ of locally noetherian schemes over $\Spf \mcO_{\breve E}$ together with the Zariski topology.

\begin{defn}
We denote by $\mcO_E\text{-}\mcO_K\text{-}\textsl{Herm}$ the stack of hermitian $\mcO_E$-$\mcO_K$-modules $(X,ι,λ)$ \emph{that have a signature} over $\textsl{Sch}/\Spf \mcO_{\breve E}$. By the condition, we mean that locally for the Zariski topology, the hermitian $\mcO_E$-$\mcO_K$-module is of signature $(r,s)$ for some integers $r,s\in \mbZ_{\geq 0}$. The morphisms in this category are the $\mcO_E$-linear morphisms of $p$-divisible groups. We also write $\mcO_E\text{-}\textsl{Herm}$ for the stack of hermitian $\mcO_E$-modules.
\end{defn}

\begin{thm}\label{thmEquiv}
There is an isomorphism of stacks on $\textsl{Sch}/ \Spf \mcO_{\breve E}$
$$\mcC:\mcO_E\text{-}\mcO_K\text{-}\textsl{Herm}\overset{\iso}{→} \mcO_E\text{-}\textsl{Herm}$$
that satisfies the following properties. It is equivariant for the Rosati involution and it sends objects of signature $(r,s)$ to objects of signature $(r,s)$.
\end{thm}
This section is devoted to the proof of these two theorems. 

\subsection{The unramified case}
\begin{prop}\label{propUram}
Consider an intermediate field $\mbQ_p\subset K\subset E_0$ and let $E_0^u\subset E_0$ be the maximal subfield which is unramified over $K$. Then there is an isomorphism of stacks
$$\mcC:\mcO_E\text{-}\mcO_K\text{-}\textsl{Herm}\overset{\iso}{→} \mcO_E\text{-}\mcO_{E_0^u}\text{-}\textsl{Herm}$$
that is equivariant for the Rosati involutions and sends objects of signature $(r,s)$ to objects of signature $(r,s)$.
\end{prop}

\emph{Proof.}
We will construct the functor $\mcC$ and its quasi-inverse. Let $S = \Spec R$ be an affine scheme over $\Spf \mcO_{\breve E}$ and let $(X,ι,λ)$ be a hermitian $\mcO_E$-$\mcO_K$-module of signature $(r,s)$ over $R$. Let $(P,Q,F,\dot{F})$ be the $\mcO_K$-display of $(X,ι,λ)$. We denote by $\langle\ ,\ \rangle:P\times P→W_{\mcO_K}(R)$ the perfect alternating form induced by the polarization $λ$.

Recall that $Ψ = \Hom_K(E^u,\breve K)$ and note that there exists a natural morphism $\mcO_{E^u}→W_{\mcO_K}(R)$ of $\mcO_K$-algebras that lifts the morphism $\mcO_{E^u}→R$, see \cite[Section 1.2]{FF}. This morphism induces gradings
$$P = \prod_{ψ\in Ψ} P_ψ,\ \ Q = \prod_{ψ\in Ψ} Q_ψ\ \ \mr{with}\ Q_ψ = Q \cap P_ψ.$$

For $ψ\notin\{ψ_0,ψ_0σ\}$, we define the Frobenius-linear isomorphism
$$[\mathbf{F}_{ψ}:P_ψ→P_{{}^Fψ}]:=\begin{cases} \dot{F}_ψ & \tif ψ\in Ψ_0\setminus \{ψ_0\}\\
                                                F_ψ & \tif ψ\in Ψ_1\setminus \{ψ_0σ\}.\end{cases}
                                                $$
Here we used that $(X,ι,λ)$ has a signature, which implies $Q_ψ = P_ψ$ whenever $ψ\in Ψ_0\setminus \{ψ_0\}$. We set $P' := P_{ψ_0}\oplus P_{ψ_0σ}$ with submodule $Q' := Q_{ψ_0}\oplus Q_{ψ_0σ}$ and define the ${}^{F^f}$-linear operators $F',\dot {F}'$ on $P'$ and $Q'$ as
$$F':=\mathbf{F}^{f-1} \circ F,\ \ \ \dot{F}':=\mathbf{F}^{f-1} \circ \dot{F}.$$
Then $\mcP':=(P',Q',F',\dot F')$ defines an $f$-$\mcO_{K}$-display in the sense of Ahsendorf \cite{ACZ}. There is an induced $\mcO_E$-action and a $W_{\mcO_K}(R)$-valued alternating pairing $\langle\ ,\ \rangle':=\langle\ ,\ \rangle\vert_{P'}$ on $P'$.

Note that $f$-$\mcO_K$-displays are the same as windows over the $\mcO_{E_0^u}$-frame
$$\mcA_{\mcO_{E_0^u}/\mcO_K}(R):=(W_{\mcO_K}(R),I_{\mcO_K}(R),{}^{F^f},{}^{F^{f-1}}{}^{V^{-1}}).$$
In our case, the pairing $\langle\ ,\ \rangle'$ takes values in $\mcA_{\mcO_{E_0^u}/\mcO_K}(R)$ in the sense that
$$\langle \dot{F}'\ ,\ \dot{F}'\ \rangle' = {}^{F^{f-1}}{}^{V^{-1}}\langle\ ,\ \rangle'$$
which follows immediately from the identities $\langle F\ ,\dot F\ \rangle = \langle \dot{F}\ ,F\ \rangle = {}^F\langle\ ,\ \rangle$ for the pairing $\langle\ ,\ \rangle$ on $\mcP$. In other words, the pairing defines a principal polarization of the $f$-$\mcO_K$-display $\mcP'$, see Proposition \ref{propDualForm}.

As explained in the appendix, \eqref{eqDualStrict}, base change along the natural strict morphism of $\mcO_{E_0^u}$-frames
$$\mcA_{\mcO_{E_0^u}/\mcO_K}(R)→(W_{\mcO_{E_0^u}}(R),I_{\mcO_{E_0^u}}(R),{}^{F'},{}^{V'^{-1}})$$
defines a principally polarized strict formal $\mcO_{E_0^u}$-module $\mcC(X)$ together with a compatible $\mcO_E$-action. This module is of signature $(r,s)$ and hence an element of $\mcO_E$-$\mcO_{E_0^u}$-$\textsl{Herm}(S)$.

\emph{Construction of a quasi-inverse of $\mcC$:} Let $\mcP':=(P',Q',F',\dot{F}')$ be the $f$-$\mcO_K$-display associated to a hermitian $\mcO_E$-$\mcO_{E_0^u}$-module $(X,ι,λ)$ over $S$. By functoriality, it comes with an $\mcO_E$-action and a compatible principal polarization. To construct the associated hermitian $\mcO_E$-$\mcO_K$-module, we apply a slightly modified version of the construction from the proof of Proposition \ref{propfOwindow}.

The $\mcO_E$-action induces a bigrading, $P' = P'_0\oplus P'_1$, $Q' = Q'_0\oplus Q'_1$. We set
$$\begin{aligned}
P_{ψ_0} := P'_0&\ \ \ \ P_{ψ_0σ} := P'_1,\\
Q_{ψ_0} := Q'_0&\ \ \ \ Q_{ψ_0σ} := Q'_1.
\end{aligned}$$
For $i = 0,\ldots,f-2$, we define
$$
P_{{}^{F^{i+1}}ψ_0} := P^{(F)}_{{}^{F^{i}}ψ_0},\ \ \ P_{{}^{F^{i+1}}ψ_0σ} := P^{(F)}_{{}^{F^{i}}ψ_0σ}.
$$
The signature condition forces us to set, for $ψ\notin \{ψ_0,ψ_0σ\}$,
$$Q_ψ = \begin{cases} P_ψ & \tif ψ\in \Psi_0\setminus \{ψ_0\}\\
I_{\mcO_K}(R)P_ψ & \tif ψ\in \Psi_1\setminus \{ψ_0σ\}.\end{cases}$$
The display structure is defined by giving a normal decomposition. Let $(P' = L'\oplus T',ϕ)$ be a normal decomposition of $\mcP'$. Then we define a normal decomposition $(P = L\oplus T, \Phi)$ as
$$\begin{aligned}
L &= L_{ψ_0}\oplus L_{ψ_0σ}\oplus \bigoplus_{ψ\in Ψ_0\setminus \{ψ_0\}} P_ψ,\\
T &= T_{ψ_0}\oplus T_{ψ_0σ}\oplus \bigoplus_{ψ\in Ψ_1\setminus \{ψ_0σ\}} P_ψ
\end{aligned}$$
and the ${}^F$-linear operator
$$Φ = \left(\begin{matrix} & & & & ϕ\\
1 & & & & \\
& 1 & & & \\
& & \ddots & & \\
& & & 1 & \\
\end{matrix}\right)\circ {}^F$$
with respect to the decomposition $P = (P_{ψ_0}\oplus P_{ψ_0σ})\oplus (P_{ψ_0}\oplus P_{ψ_0σ})^{(F)}\oplus\ldots\oplus (P_{ψ_0}\oplus P_{ψ_0σ})^{(F^{f-1})}$. This already defines an $f$-$\mcO_K$-display $\mcP:=(P,Q,F,\dot F)$ equipped with an $\mcO_E$-action of the correct signature.

We now construct the polarization (as $\mcO_K$-display) on $\mcP$. Recall that we have a perfect pairing
$$\langle\ , \rangle_{ψ_0}:P_{ψ_0}\times P_{ψ_0σ}→W_{\mcO_K}(R).$$
There is at most one way to extend it to a pairing $\langle\ ,\ \rangle$ on all of $P$ such that the relations of a polarization are satisfied. Let us explain this for the direct summand $P_{{}^Fψ_0}\oplus P_{{}^Fψ_0σ}$. If $l+t\in L_{ψ_0}\oplus T_{ψ_0}$ and $l'+t'\in L_{ψ_0σ}\oplus T_{ψ_0σ}$, then we have to set
$$\begin{aligned}
\langle Φ(l+t), Φ(l'+t')\rangle & := \langle \dot{F}(l)+F(t), \dot{F}(l')+F(t')\rangle\\
& = {}^{V^{-1}} \langle l,l'\rangle_{ψ_0} + {}^F\langle l, t'\rangle_{ψ_0} + {}^F\langle t, l'\rangle_{ψ_0} + π_K{}^F\langle t, t'\rangle_{ψ_0}.\end{aligned}$$
Since $Φ$ is a Frobenius-linear isomorphism, this is well-defined and extends in a unique way to all of $P_{{}^Fψ_0}\oplus P_{{}^Fψ_0σ}$. We apply the same formulas to define $\langle\ ,\ \rangle$ on $P_{{}^{F^{i}}ψ_0}\oplus P_{{}^{F^{i}}ψ_0σ}$ for $i = 1,\ldots, f-1$. Note that due to the special form of the normal decomposition at these indices, we get
$$\langle\ ,\ \rangle\vert_{P_{{}^{F^{i+1}}ψ_0}\oplus P_{{}^{F^{i+1}}ψ_0σ}} = \langle\ ,\ \rangle\vert_{P_{{}^{F^i}ψ_0}\oplus P_{{}^{F^i}ψ_0σ}}^{(F)}.$$
We leave it to the reader to check that this defines a principal polarization on $\mcP$ which finishes the proof.
\qed

\begin{rmk}
Note that we did not use the assumption of $R$ being noetherian. We will only need this assumption in the ramified situation.
\end{rmk}

\subsection{The totally ramified case}

\begin{prop}\label{propRam}
Let $\mbQ_p\subset K\subset E_0$ be an intermediate field such that $E_0/K$ is totally ramified. There is an isomorphism of stacks over $\textsl{Sch}/\Spf \mcO_{\breve E}$
$$\mcC:\mcO_E\text{-}\mcO_K\text{-}\textsl{Herm}\overset{\iso}{→} \mcO_E\text{-}\textsl{Herm}$$
that is equivariant for the Rosati involution and that sends objects of signature $(r,s)$ to objects of signature $(r,s)$.
\end{prop}

\emph{Proof.}
The proof consists of three main steps. First, we will construct the functor $\mcC$. Second, we will prove that $\mcC$ is an equivalence on reduced $\mcO_{\breve E}$-schemes in characteristic $p$. Finally, we will prove that $\mcC$ identifies the deformation theories of $X$ and $\mcC(X)$. Together with the restriction to locally noetherian schemes, this will imply the statement.

\emph{First step: Construction of the functor $\mcC$.}

Let $S = \Spec R$ be a scheme over $\Spf \mcO_{\breve E}$. We begin by briefly spelling out the properties of the $\mcO_K$-display of a hermitian $\mcO_E$-$\mcO_K$-module $(X,ι,λ)$ of signature $(r,s)$ over $S$. For this and in the following, we identify $Ψ$ with $\{0,1\}$ such that $ψ_0$ corresponds to $0$.

Let $\mcP:=(P,Q,F,\dot F)$ be the $\mcO_K$-display of $X$. The action $ι$ of $\mcO_E$ on $X$ induces an action of $\mcO_E$ on the display. This makes $P$ and $Q$ into $\mcO_E\tensor_{\mcO_K} W_{\mcO_K}(R)$-modules with $P$ being projective by \cite[Prop. 2.22]{ACZ}. Both $F$ and $\dot F$ are $\mcO_E$-linear. The action of the unramified part induces bigradings $P=P_0\oplus P_1$ and $Q=Q_0\oplus Q_1$ such that both $F$ and $\dot F$ are homogeneous of degree $1$.

The signature condition implies that $P_0/Q_0$ is projective of rank $r$ over $R$ and that $P_1/Q_1$ is projective of rank $ne-r$ over $R$. Furthermore, $\mcO_E$ acts on $P_0/Q_0$ via the structure morphism. In other words,
$$J_{\mcO_{E_0}}(R)P_0\subset Q_0,$$
where $J_{\mcO_{E_0}}(R) := \ker (\mcO_{E_0}\tensor_{\mcO_K} W_{\mcO_K}(R)→R)$. The polarization induces a perfect alternating form $\langle\ ,\ \rangle:P\times P→W_{\mcO_K}(R)$ which satisfies $\langle a\ ,\ \rangle = \langle\ ,a^{σ}\ \rangle$ for all $a\in \mcO_E$. In particular, $P_0$ and $P_1$ are maximal isotropic subspaces of $P$ which are put into duality by $\langle\ ,\ \rangle$.
Furthermore, $\langle Q,Q\rangle \subset I_{\mcO_K}(R)$ and the pairing satisfies $\langle\dot F\ ,\dot F\ \rangle = {}^{V^{-1}}\langle\ ,\ \rangle$. In other words, the pairing takes values in the Witt $\mcO_K$-frame from Definition \ref{defWitt},
$$\mcW_{\mcO_K}(R)=(W_{\mcO_K}(R),I_{\mcO_K}(R),R,{}^F,{}^{V^{-1}}).$$

\emph{Construction of $\mcC(X)$: }
Let $(X,ι,λ)/S$ and $\mcP$ be as above. As an intermediate step, we construct a polarized window $\mcP'=(P',Q',F',\dot F')$ with $\mcO_E$-action over a Lubin-Tate frame over $R$, see Definition \ref{defLT}.

Let $ϑ_{E_0}$ be a generator of the inverse different of $E_0/K$. Consider the $W_{\mcO_K}(R)$-linear extension of the trace
$$\begin{aligned}
\mathbf{tr}:\mcO_{E_0}\tensor_{\mcO_K} W_{\mcO_K}(R)&→W_{\mcO_K}(R)\\
a\tensor w & \longmapsto \mr{tr}_{E_0/K}(ϑ_{E_0}a)w.
\end{aligned}$$
Since $\langle\ ,\ \rangle$ is $\mcO_{E_0}$-equivariant, it has a unique lifting to a perfect 
$\mcO_{E_0}\tensor_{\mcO_K} W_{\mcO_K}(R)$-bilinear alternating form $(\ ,\ ): P\times P→\mcO_{E_0}\tensor_{\mcO_K} W_{\mcO_K}(R)$ such that
$$\langle\ ,\ \rangle = \mathbf{tr}\circ (\ ,\ ).$$
We set $P':=P$ with its given $\mcO_E$-action. Again, $P'$ is bigraded and we set $Q_0' := Q_0$. The form $(\ ,\ )$ automatically satisfies $(a\ ,\ ) = (\ ,a^{σ}\ )$ and hence $Q_0'$ is totally isotropic. We define
$$Q_1' :=\{ p_1\in P_1 \mid (p_1,Q_0') \subset J_{\mcO_{E_0}}(R)\}.$$
Note that $(\ ,\ )$ induces a perfect pairing $(P/J_{\mcO_{E_0}}(R)P) \times (P/J_{\mcO_{E_0}}(R)P)→R$ by base change. Then $Q_1'$ is the inverse image of $(Q_0/J_{\mcO_{E_0}}(R)P_0)^\perp$ under the projection $P_1→P_1/J_{\mcO_{E_0}}(R)P_1$. In particular, $P_1/Q_1'$ is a projective $R$-module of rank $s = n-r$.

Let $θ\in \mcO_{E_0}\tensor_{\mcO_K} W_{\mcO_K}(\mcO_{E_0})$ be an element such that
$$θJ_{\mcO_{E_0}}(\mcO_{E_0}) \subset \mcO_{E_0}\tensor_{\mcO_K}I_{\mcO_K}(\mcO_{E_0})$$
and such that the image of $θ$ in $\mcO_{E_0}\tensor_{\mcO_K}W_{\mcO_K}(\mcO_{E_0}/π_{E_0})\iso \mcO_{E_0}$ is of valuation $e-1$. Its existence is given by Lemma \ref{lemMagic}.
\begin{lem}
There is an inclusion
$$θQ_1' \subset Q_1.$$
\end{lem}
\begin{proof}
By definition, $Q_1 = \{ p_1\in P_1 \mid \langle p_1,Q_0\rangle \subset I_{\mcO_K}(R)\}.$ So given $q_1\in Q_1'$, we need to verify that
$$
\begin{aligned}
\langle θq_1,Q_0\rangle & = \mathbf{tr}\left( ( θq_1,Q_0 ) \right)\\
            & = \mathbf{tr}\left( θ ( q_1,Q_0) \right) \subset I_{\mcO_K}(R).
\end{aligned}
$$
For the second equality, we used that $(\ ,\ )$ is $\mcO_{E_0}\tensor_{\mcO_K} W_{\mcO_K}(R)$-bilinear. It is enough to show
$$\mathbf{tr}(θ J_{\mcO_{E_0}}(R))\subset I_{\mcO_K}(R)$$
which follows from the fact that
$$θJ_{\mcO_{E_0}}(R)\subset \mcO_{E_0}\tensor_{\mcO_K}I_{\mcO_K}(R).$$
\end{proof}
In particular, we can define $\dot F_1':Q_1'→P_0$ as $\dot F_1'(q_1) := \dot F_1(θq_1).$ Then $\dot F_1'$ is a Frobenius-linear epimorphism, which can be checked at closed points of $S$. On $Q_0'$, we set $\dot F_0' = \dot F_0$.

Let $\mcL_{\mcO_{E_0}/\mcO_K,κ}(R)$ be the so-called Lubin-Tate $\mcO_{E_0}$-frame
$$\mcL_{\mcO_{E_0}/\mcO_K,κ}(R) := (\mcO_{E_0}\tensor_{\mcO_K} W_{\mcO_K}(R), J_{\mcO_{E_0}}(R), R, σ, \dot{σ})$$
where $\dot{σ}(ξ) = {}^{V^{-1}}(θξ)$, see Example \ref{exLT} (2). The unit $κ\in \mcO_{E_0}\tensor_{\mcO_K} W_{\mcO_K}(R)$ is the element $\dot{σ}(π_{E_0}\tensor 1 - 1\tensor [π_{E_0}])$.
We define the Frobenius $F':P→P$ through the relation
$$F'(x) = κ^{-1}\dot F'((π_{E_0}\tensor 1 - 1\tensor [π_{E_0}])x).$$
The reader may check that this defines the structure of a $\mcL_{\mcO_{E_0}/\mcO_K,κ}(R)$-window $\mcP'=(P',Q',F',\dot{F}')$.

\begin{lem}
The pairing $(\ ,\ )$ is a principal polarization of the $\mcL_{\mcO_{E_0}/\mcO_K,κ}(R)$-window $\mcP'$.
\end{lem}
\begin{proof}
By definition of $Q'$, the pairing satisfies $(Q',Q')\subset J_{\mcO_{E_0}}(R)$. We verify for $q_0\in Q_0',q_1\in Q_1'$ that
\begin{equation}
\begin{aligned}
(\dot F'q_0,\dot F'q_1) & = (\dot F q_0, \dot F(θq_1))\\
& = {}^{V^{-1}}(q_0, θq_1)\\
& = {}^{V^{-1}}(θ(q_0,q_1)) = \dot{σ}(q_0,q_1).
\end{aligned}
\end{equation}
Here, the second equality holds since it does for the pairing $\langle\ ,\ \rangle$. The third equality used the $\mcO_{E_0}$-bilinearity of the pairing $(\ ,\ )$.
\end{proof}

By Proposition \ref{propLTFunct}, there is a strict morphism of $\mcO_{E_0}$-frames,
$$\mcL_{\mcO_{E_0}/\mcO_K,κ}(R)→\mcL_{\mcO_{E_0}/\mcO_{E_0},κ}(R).$$
Base change along this morphism defines a supersingular strict $\mcO_{E_0}$-module with $\mcO_E$-action and a principal polarization with values in the $\mcO_{E_0}$-frame $\mcL_{\mcO_{E_0}/\mcO_{E_0},κ}(R)$.

The identity on $W_{\mcO_{E_0}}(R)$ defines a $κ/u$-isomorphism to the Witt $\mcO_{E_0}$-frame
$$\mcL_{\mcO_{E_0}/\mcO_{E_0},κ}(R)→ \mcW_{\mcO_{E_0}}(R)$$
where $u$ is the unit $u = V_{π_{E_0}}^{-1}(π_{E_0} - [π_{E_0}])$. There exists a unit $ε\in W_{\mcO_{E_0}}(\mcO_{\breve E})$ such that
$$σ(ε)ε^{-1} = κ/u.$$
Scaling the polarization by $ε$, see Lemma \ref{lemDual}, we get a principally polarized $\mcO_{E_0}$-display $\mcP''=(P'',Q'',F'',\dot{F}'')$ with $\mcO_E$-action. The corresponding strict formal $\mcO_{E_0}$-module is then the hermitian $\mcO_E$-module $\mcC(X)$ we wanted to construct. It is of signature $(r,s)$.

The construction of we made is functorial (in $R$) and compatible with the Rosati involutions.

\emph{Second step: $\mcC$ is an equivalence over reduced schemes in characteristic $p$.}

We will construct a quasi-inverse. Let $R$ be a reduced $\Spf \mcO_{\breve E}$-scheme and let $\mcP' = (P',Q',F',\dot F')$ be the $\mcL_{\mcO_{E_0}/\mcO_K,κ}(R)$-window equipped with an $\mcO_E$-action $ι$ and a principal polarization $λ$ associated to a hermitian $\mcO_E$-module over $R$. We assume that $\mcP'$ is of signature $(r,s)$. Here, $κ$ is the unit ${}^{V^{-1}}(θ(π_{E_0}\tensor 1 - 1\tensor [π_{E_0}]))$ from the previous paragraph.

The $\mcO_E$-action induces gradings $P' = P_0'\oplus P_1'$ and $Q' = Q_0'\oplus Q_1'$. We set $P := P'$ together with its $\mcO_E$-action.

The polarization $λ$ induces a perfect pairing
$$(\ ,\ ):P_0\times P_1→\mcO_{E_0}\tensor_{\mcO_K} W_{\mcO_K}(R)$$
such that
$$Q_1' = \{ p_1\in P_1 \mid (p_1,Q_0') \subset J_{\mcO_{E_0}}(R)\}.$$
We set $Q_0 := Q_0'$ and
$$Q_1 := \{ p_1\in P_1 \mid (p_1, Q_0) \subset \mcO_{E_0}\tensor_{\mcO_K} I_{\mcO_K}(R)\}.$$

\emph{Claim: $Q_1 = θQ_1'+I_{\mcO_K}(R)P_1$.}\\
The relation $\supseteq$ is clear, so we only need to check that $Q_1 \subseteq θQ_1'$ mod $I_{\mcO_K}(R)P_1$. Let us denote by $\overbar P_0$, $\overbar Q_0$, etc.\! the quotients of the various modules by $I_{\mcO_K}(R)P$. Then the pairing $(\ ,\  )$ induces a perfect pairing
$$\overbar P_0 \times \overbar P_1 → \mcO_{E_0}\tensor_{\mcO_K} R \iso \mcO_{E_0}/π_K \tensor_{\mcO_K/π_K} R.$$
Here, we used that $π_K = 0$ in $R$. We have
$$\begin{aligned}
π_{E_0}\overbar P_0 \subset \overbar{Q}_0 &\subset \overbar{P}_0 &\tand\\
π_{E_0}\overbar P_1 \subset \overbar{Q}'_1 &\subset \overbar{P}_1.&
\end{aligned}$$
Then by definition, $\ob{Q}_1$ is the orthogonal complement of $\ob{Q}_0$. In particular, it is projective of rank $r$ over $R$ and locally a direct summand of $\ob{P}_1$. Let $\ob{θ}$ be the image of $θ$ in $\mcO_{E_0}\tensor_{\mcO_K} \mcO_K/π_K \iso \mcO_{E_0}/π_K$. Then $\ob{θ}\neq 0$, $π_{E_0}\ob{θ} = 0$ and multiplication by $\ob{θ}$ induces an isomorphism
$$\mcO_{E_0}/π_{E_0} \tensor_{\mcO_K} R \iso (π_{E_0}^{e-1})/(π_K) \tensor_{\mcO_K} R.$$
Thus $\ob{θ} \ob{Q}'_1$ is also of rank $r$ over $R$ and locally a direct summand of $\ob{P}_1$. The claim follows from the relation $\ob{θ}\ob{Q}'_1\subseteq \ob{Q}_1$.

Since $R$ is reduced, $W_{\mcO_K}(R)$ is $π_K$-torsion free. Also, there exists an inclusion $R\hookrightarrow \prod_{i\in I} k_i$ into a product of perfect fields and hence $σ(θ)$ is not a zero-divisor in $\mcO_{E_0}\tensor_{\mcO_K} W_{\mcO_K}(R).$ It follows from the claim that $\dot F'\vert_{Q_1}$ is divisible by $σ(θ)$. So we can define a $σ$-linear epimorphism $\dot F:Q_0\oplus Q_1→P_1\oplus P_0$ as
$$\dot F = \dot F'\vert_{Q_0}\oplus σ(θ)^{-1}\dot F'\vert_{Q_1}.$$

We leave it to the reader to check that $(P,Q,F,\dot F)$ is an $\mcO_K$-display with $\mcO_E$-action of signature $(r,s)$ and that $\langle\ ,\ \rangle:=\mathbf{tr}\circ (\ ,\ )$ defines a principal polarization, compatible with the $\mcO_E$-action. This finishes the construction of the quasi-inverse.

\emph{Third step: Identifying the deformation theories of $X$ and $\mcC(X)$.}

Denote by $\mbD$ (resp.\! by $\mbD''$) the $\mcO_K$-crystal of $X$ (resp.\! the $\mcO_{E_0}$-crystal of $\mcC(X)$) on the category of $\mcO_K$-pd-thickenings of $S$ (resp.\! on the category of $\mcO_{E_0}$-pd-thickenings of $S$). Note that both crystals are bigraded by the $\mcO_E$-action and that, in the notation of Step 1 above, $\mbD(S) = P/I_{\mcO_K}(R)P$ and $\mbD''(S) = P''/I_{\mcO_{E_0}}(R)P''$. Furthermore, the Hodge filtration $\mcF\subset \mbD(S)$ is given by $Q/I_{\mcO_K}(R)P$ and is similarly bigraded.

We denote by $\mcD$ the \emph{contraction} of the crystal $\mbD$,
$$
\mcD(\tilde S) := \mbD(\tilde S) \tensor_{\mcO_E\tensor_{\mcO_K} \mcO_{\tilde S}} \mcO_{\tilde S},
$$
for any $\mcO_K$-pd-thickening $S→\tilde {S}$ over $\mcO_{\breve E}$. Let $\ob{J}$ be the kernel of the projection $\mcO_E\tensor_{\mcO_K}\mcO_S→\mcO_S$ so that we have $\mcD(S) = \mbD_0(S)/\ob{J}\mbD_0(S)$. By the signature condition, there is an inclusion $\ob{J}\mbD_0(S)\subset \mcF_0\subset \mbD_0(S)$ which defines a filtration $\mcF_0/\ob{J}\mbD_0(S)\subset \mcD(S)$. We call it the Hodge filtration on $\mcD(S)$. 

\begin{lem}
Let $S→\tilde S$ be a square-zero thickening. There is an $\mcO_E$-linear identification of the contraction $\mcD$ of the crystal of $X$ and the $0$-component of the crystal of $\mcC(X)$ evaluated at $\tilde S$,
$$\mcD(\tilde S)\iso \mbD_0''(\tilde S).$$
This identification is functorial in $X$. In the case $S = \tilde S$, the Hodge filtrations on both sides agree.
\end{lem}

\begin{cor}\label{corCrucial}
Let $S→\tilde S$ be a square-zero thickening and let $X$ be a hermitian $\mcO_E$-$\mcO_K$-module of signature $(r,s)$ over $S$. Then there is a natural bijection between deformations of $X$ to $\tilde S$ and deformations of $\mcC(X)$ to $\tilde S$.

Furthermore, let $r',s'\in \mbZ_{\geq 0}$ and let $Y$ be an $\mcO_E$-$\mcO_K$-module of signature $(r',s')$ over $S$. Let $\tilde{X},\tilde{Y}$ be deformation of the two modules to $\tilde S$. Then an $\mcO_E$-linear homomorphism $f:X→Y$ lifts to $\tilde{X}→\tilde{Y}$ if and only if $\mcC(f)$ lifts to $\mcC(\tilde{X})→\mcC(\tilde{Y})$.
\end{cor}

\begin{proof}
We first do the case $\tilde S = S$. Keeping the notation from Step 1, by construction,
$$\mcD(S)=(P_0/J_{\mcO_{E_0}}(R)P_0) = (P'_0/J_{\mcO_{E_0}}(R)P'_0) = P''_0/I_{\mcO_{E_0}}(R)P''_0 = \mbD_0''(S).$$
The submodule $Q_0''\subset P_0''$ was defined as the inverse image of the Hodge filtration
$$(Q_0+J_{\mcO_{E_0}}(R)P_0)/I_{\mcO_K}(R)P_0\subset P_0/I_{\mcO_K}(R)P_0$$
and hence the Hodge filtrations on both sides agree.

For a non-trivial square-zero thickening $S→\tilde S$, we argue as follows. Locally on $S$, we can deform $X$ to a hermitian $\mcO_E$-$\mcO_K$-module $\tilde X$ of signature $(r,s)$ on $\tilde S$. Then $\mcC(\tilde X)$ is a deformation of $\mcC(X)$. The values $\mbD(\tilde S)$ and $\mbD''(\tilde S)$ can then be computed from the displays of $\tilde X$ and $\mcC(\tilde X)$ and the above arguments apply.

The corollary is now an immediate application of Grothendieck-Messing deformation theory. Namely as explained in the proof of Proposition \ref{propSmoothness}, deformations of $X$ (resp.\! of $\mcC(X)$) are in bijection with liftings of the Hodge filtration in $\mcD(S)$ (resp.\! of the Hodge filtration in $\mbD_0''(S)$). A similar result holds for homomorphisms.
\end{proof}

\emph{End of Proof:} Let $R$ be a noetherian $\mcO_{\breve E}$-algebra in which $p$ is nilpotent and set $\mfn:=\ker(R→R_{\mr{red}})$. Then Step 2 applies to $R_{\mr{red}}$. Applying Corollary \ref{corCrucial} to the successive quotients $R/\mfn^{i+1}→R/\mfn^{i}$, we get both the essential surjectivity and fully faithfulness of $\mcC$ on $R$-valued points. This finishes the proof of the proposition and hence of Theorem \ref{thmEquiv}. \qed

\subsection{Proof of Theorem \ref{mainThm1}}
We fix a quasi-isogeny on framing objects, which is possible by Proposition \ref{propIsocrys} and Lemma \ref{lemExistframing},
\begin{equation}\label{eqFram}
α:\mcC(\mbX_{E_0/K,(r,s)}) \iso \mbX_{E_0,(r,s)}.
\end{equation}
Alternatively, we \emph{choose} $\mcC(\mbX_{E_0/K,(r,s)})$ as the framing object in the definition of $\mcN_{E_0,(r,s)}$. Together with $\mcC$, this induces a morphism
$$c:\mcN_{E_0/K,(r,s)}→\mcN_{E_0,(r,s)}$$
which is an isomorphism by Theorem \ref{thmEquiv}. Furthermore, $c$ is equivariant with respect to the isomorphism
$$\begin{aligned}
U(\mbX_{E_0/K,(r,s)})&→U(\mcC(\mbX_{E_0/K,(r,s)})) & → & U(\mbX_{E_0,(r,s)})\\
g & \longmapsto \mcC(g) & \longmapsto & α\mcC(g)α^{-1}.
\end{aligned}$$
\qed

\section{Cycles on unitary RZ-spaces}
\label{sect:cycles}
\subsection{Definitions and compatibility with $\mcC$}

Let $\mbQ_p\subset K \subset E_0$ be finite extensions and let $E/E_0$ be an unramified quadratic extension. We choose uniformizers $π_{K}\in \mcO_{K}$ and $π_{E_0}\in \mcO_{E_0}$ in order to make sense of polarizations of strict $\mcO_{K}$-modules (resp.\! strict $\mcO_{E_0}$-modules).
Let $\mbX_{E_0/K,(r,s)}$ over $\mbF$ be the framing object for $\mcN_{E_0/K,(r,s)}$.
Recall that $\End^0_{E}(\mbX_{E_0/K,(r,s)})$ $\iso M_n(E)$ is isomorphic to a matrix ring over $E$, see Proposition \ref{propIsocrys}.

\begin{defn}\label{defZg}
(1) For a quasi-endomorphism $x\in \End^0_{E}(\mbX_{E_0/K,(r,s)})$, we denote by $\mcZ(x)\subset \mcN_{E_0/K,(r,s)}$ the closed formal subscheme of points $(X,ρ)$ such that the quasi-endomorphism $ρ^{-1}xρ$ is an actual endomorphism, see \cite[Proposition 2.9]{RZ}.\\
(2) For a subring $R\subset \End^0_{E}(\mbX_{E_0/K,(r,s)})$, we denote by $\mcZ(R)\subset \mcN_{E_0/K,(1,n-1)}$ the closed formal subscheme of points $(X,ρ)$ such that $ρ^{-1}Rρ\subset \End(X)$. In other words,
$$\mcZ(R) = \bigcap_{x\in R} \mcZ(x).$$
\end{defn}

\begin{rmk}
Note that $\mcZ(x) = \mcZ(x^*)$ where $*$ denotes the Rosati involution, see Definition \ref{defRosati}. Also, $\mcZ(x)$ only depends on the $\mcO_E$-algebra $\mcO_E[x]$ spanned by $x$. In particular, there are equalities $\mcZ(x) = \mcZ(\mcO_E[x,x^*])$ and $\mcZ(R) = \mcZ(\mcO_E[R,R^*]).$
\end{rmk}

The second kind of cycle we want to consider is defined as follows. Let $\mbY_{E_0/K}$ be a hermitian $\mcO_E$-$\mcO_K$-module over $\mbF$ of signature $(0,1)$. Such an object is even unique up to isomorphism which can be checked with Dieudonné theory. Let $\mcY_{E_0/K}$ be a deformation of it to $\Spf \mcO_{\breve E}$. Such a deformation is unique up to isomorphism according to Proposition \ref{propSmoothness}. Via base change, $\mcY_{E_0/K}$ is defined on any $\Spf \mcO_{\breve E}$-scheme. We also set $\mbY_{E_0}:=\mbY_{E_0/E_0}$ and $\mcY_{E_0} := \mcY_{E_0/E_0}$.

\begin{defn}\label{defKRdiv}
For a quasi-homomorphism $j\in \Hom^0_{E}(\mbY_{E_0/K}, \mbX_{E_0/K,(r,s)})$, we denote by $\mcZ(j)\subset \mcN_{E_0/K,(r,s)}$ the closed formal subscheme of points $(X,ρ)$ such that the quasi-homomorphism
$$ρ^{-1}j:\mcY_{E_0/K}→X$$
is an actual homomorphism.
\end{defn}

Again, the existence of such a closed formal subscheme follows from \cite[Proposition 2.9]{RZ}. As above, $\mcZ(j) = \mcZ(j^*)$ where $j^*:\mbX_{E_0/K,(r,s)}→\mbY_{E_0/K}$ is the Rosati adjoint of $j$ and where $\mcZ(j^*)$ denotes the locus of $(X,ρ)$ such that $j^*ρ:X→\mcY_{E_0/K}$ is a homomorphism. Also, $\mcZ(j)$ only depends on the span $\mcO_Ej$.
The case of interest for the cycles $\mcZ(j)$ is that of signature $(1,n-1)$. In this case, $\mcZ(j)\subset \mcN_{E_0/K,(1,n-1)}$ is a divisor whenever $j\neq 0$. These divisors were first considered by Kudla and Rapoport, see \cite{KRlocal}.

\begin{rmk} \label{rmkComp1}
Let $\mcC$ be the functor from Theorem \ref{thmEquiv} and fix a quasi-isogeny ($E$-linear, preserving the polarization)
$$α:\mcC(\mbX_{E_0/K,(r,s)}) \overset{\iso}{→} \mbX_{E_0,(r,s)}$$
and an isomorphism ($E$-linear, preserving the polarization)
$$β:\mcC(\mcY_{E_0/K}) \overset{\iso}{→}  \mcY_{E_0}.$$
It then follows from Theorem \ref{thmEquiv} and the construction of $c$ in Theorem \ref{mainThm1} that
$$c:\mcZ(x)\iso \mcZ(α\mcC(x)α^{-1})$$
and
$$c:\mcZ(j) \iso \mcZ(α\mcC(j)β^{-1}).$$
\end{rmk}

\subsection{The cycle $\mcZ(\mcO_A)$ for $A/E$ a field extension}
\label{subsect:ZOE}
From now on, we restrict to the case of signature $(1,n-1)$. Our aim is to study the formal scheme $\mcZ(R),\ R\subset \End^0(\mbX_{E_0/K,(1,n-1)})$ in the special case that $R$ is the ring of integers in a field extension $A/E$ whose image in $\End^0(\mbX_{E_0/K,(1,n-1)})$ is stable under the Rosati involution. By the previous results, we may assume $K=E_0$. Let $A_0:=A^{*=\mr{id}}$ and $d:=[A_0:E_0]$ which divides $n$, say $n = dn'$. Since $*$ acts non-trivially on the residue field, $A/A_0$ is unramified quadratic and the inertia degree $f$ of $A_0/E_0$ is odd.

Let $A^u\subset A$ be the maximal subfield that is unramified over $E_0$.
Set $Ψ := \Hom_{E_0}(A^u,\breve E) =Ψ_0\sqcup Ψ_1$ as the unique extension of the decomposition $\{0,1\} = \{0\}\sqcup \{1\}$ for $E/E_0$. Together with the action of $\mcO_A$, $\mbX_{E_0,(1,n-1)}$ becomes a framing object for $\mcN_{A_0/E_0,(1,n'-1)}$.
Recall that in the definition of the signature condition \ref{defSignature}, we also had to choose an element $ψ_0\in Ψ_0$. We now deviate slightly from the previous setting by letting $ψ_0$ vary and denote the associated moduli problem by $\mcN_{A_0/E_0,(1,n'-1)}^{ψ_0}$. Then forgetting the $A$-action induces a morphism
\begin{equation}\label{eqEmbed}
\mcN_{A_0/E_0,(1,n'-1)}^{ψ_0}→\mcN_{E_0,(1,n-1)}
\end{equation}
which is a closed immersion.

\begin{thm}\label{thmZoe}
The above forgetful map induces an isomorphism
$$\coprod_{ψ_0\in Ψ_0} \mcN^{ψ_0}_{A_0/E_0,(1,n'-1)} \overset{\iso}{→} \mcZ(\mcO_A).$$
\end{thm}
\begin{proof}
Clearly, the map is a monomorphism with image contained in $\mcZ(\mcO_A)$. So we only have to show $\mcZ(\mcO_A)\subset \coprod \mcN^{ψ_0}_{A_0/E_0,(1,n'-1)}$ which we do on $S$-valued points with $S$ connected. Let $(X,ι,λ,ρ)\in \mcZ(\mcO_A)(S)$. Then the $\mcO_{A^u}$-action induces a decomposition of the Lie algebra as
$$\Lie(X)=\prod_{ψ\in Ψ} \Lie(X)_{ψ}$$
which extends the bigrading from the $\mcO_E$-action. Since $X$ has signature $(1,n-1)$, there is a unique index $ψ_0\in Ψ_0$ (amongst the indices in $\Psi_0$) such that $\Lie(X)_{ψ_0} \neq 0$. This summand is then a line bundle on $S$. The $\mcO_A$-action on $\Lie(X)_{ψ_0}$ induces a morphism
$$S→\Spf \mcO_A\tensor_{\mcO_{A^u},ψ_0}\mcO_{\breve E} = \Spf \mcO_{\breve A}.$$
(This explains why we restricted ourselves to the case of signature $(1,n-1)$.) Then by definition, $X$ is an $S$-valued point of $\mcN^{ψ_0}_{A_0/E_0,(1,n'-1)}$.
\end{proof}

\begin{rmk}\label{rmkComp2}
The identification is compatible with the formation of the cycle $\mcZ(R)$ in the following sense. Assume that $R\subset \End^0_E(\mbX_{E_0,(1,n-1)})$ is such that $\mcO_A\subset R$. Then
$$\mcZ(R)\subset \mcZ(\mcO_A)$$
and
$$c:\coprod_{ψ_0\in Ψ_0} \mcZ(R)^{ψ_0}\iso \mcZ(R).$$
Here, the source are the cycles in the RZ-spaces $\mcN_{A_0/E_0,(1,n'-1)}^{ψ_0}$. Furthermore, the $\mcZ(R)^{ψ_0}$ are all isomorphic which can be seen by identifying the $\mcN_{A_0/E_0,(1,n'-1)}^{ψ_0}$ with $\mcN_{A_0,(1,n'-1)}$ and using Remark \ref{rmkComp1}.
\end{rmk}

We fix a generator $ϑ_A$ of the inverse different of $A_0/E_0$ which induces an $\mcO_{A_0}$-linear isomorphism
$$ϕ:\mcO_{A_0}→\Hom_{\mcO_{E_0}}(\mcO_{A_0},\mcO_{E_0}),\ a \mapsto ϕ(a)(-) := \mr{tr_{A_0/E_0}}(ϑ_Aa-).$$
\begin{defn}\label{defSerre}
To any hermitian $\mcO_E$-module $(Y,ι,λ)$, we associate a hermitian $\mcO_A$-$\mcO_{E_0}$-module $(\mcO_{A_0}\tensor_{\mcO_{E_0}} Y,ι',λ')$ as follows. The $\mcO_{E_0}$-module $\mcO_{A_0}\tensor_{\mcO_{E_0}} Y$ is given by the Serre tensor construction. The $\mcO_A= \mcO_{A_0}\tensor_{\mcO_{E_0}} \mcO_E$-action $ι'$ is given by the natural $\mcO_{A_0}$-action on the first factor and $λ'$ is defined as,
$$λ'(a\tensor y) := ϕ(a)\tensor λ(y).$$
\end{defn}

Let $j\in \Hom^0(\mbY_{E_0}, \mbX_{E_0, (1,n-1)})$ be a quasi-homomorphism with associated divisor $\mcZ(j)\subset \mcN_{E_0,(1,n-1)}$. 

\begin{lem}\label{lemIsoCan}
There is an isomorphism of hermitian $\mcO_A$-$\mcO_{E_0}$-modules over $\mbF$,
$$\mcO_{A_0}\tensor_{\mcO_{E_0}} \mbY_{E_0} \iso \mbY_{A_0/E_0}.$$
Similarly for their deformations to $\Spf \mcO_{\breve E}$,
$$\mcO_{A_0}\tensor_{\mcO_{E_0}} \mcY_{E_0} \iso \mcY_{A_0/E_0}.$$
\end{lem}
\begin{proof}
Up to isomorphism, there is a unique $\mcO_A$-$\mcO_{E_0}$-module of signature $(0,1)$ over $\mbF$. So the statement reduces to the fact that $\mcO_{A_0}\tensor_{\mcO_{E_0}} \mbY_{E_0}$ has signature $(0,1)$ which follows from
$$\Lie(\mcO_{A_0}\tensor_{\mcO_{E_0}} \mbY_{E_0}) = \mcO_{A_0}\tensor_{\mcO_{E_0}}\Lie(\mbY_{E_0}).$$
\end{proof}

\begin{lem}\label{lemCompKR}
Let $\mcN_{A_0/E_0,(1,n'-1)}^{ψ_0}→\mcN_{E_0,(1,n-1)}$ be the embedding from \eqref{eqEmbed}. Then
$$\mcZ(j)\cap \mcN_{A_0/E_0,(1,n'-1)}^{ψ_0} = \mcZ(ι\tensor j)$$
where $ι\tensor j$ is the homomorphism
$$\mbY_{A_0/E_0} \iso \mcO_{A_0}\tensor_{\mcO_{E_0}} \mbY_{E_0} \overset{ι \tensor j}{→} \mbX_{E_0,(1,n-1)} \iso \mbX_{A_0/E_0,(1,n'-1)}.$$
\end{lem}
\begin{proof}
The universal $\mcO_{E_0}$-module over $\mcN_{A_0/E_0,(1,n'-1)}^{ψ_0}$ has an $\mcO_{A_0}$-action, which implies the relation $\subseteq$. Conversely, if $ι\tensor j:\mcO_{A_0}\tensor_{\mcO_{E_0}} \mbY_{E_0} → \mbX_{A_0/E_0,(1,n'-1)}$ lifts, then so does its composition with
$$\mbY_{E_0}\overset{1\tensor \mr{id}}{→}\mcO_{A_0}\tensor_{\mcO_{E_0}} \mbY_{E_0}.$$
\end{proof}

\subsection{The cycle $\mcZ(\mcO_A)$ for $A/E$ an étale algebra}
We first define a variant of $\mcN_{E_0/K,(0,n)}$ for a split quadratic extension $E = E_0\times E_0$. Set $\breve E:=\breve E_0$. As in the non-split case, we choose a decomposition $Ψ:=\Hom_{K}(E^u,\breve E)=Ψ_0\sqcup Ψ_1$ such that $σ(Ψ_0) = Ψ_1$. A hermitian $\mcO_E$-$\mcO_{K}$-module over an $\Spf \mcO_{\breve E}$-scheme $S$ is a triple $(X,ι,λ)$ where $X/S$ is a supersingular strict $\mcO_{K}$-module, $ι:\mcO_E→\End(X)$ an action and $λ:X→X^\vee$ a principal polarization whose Rosati involution is the Galois conjugation on $ι(\mcO_E)$. It is of signature $(0,n)$ if, in the notation of Lemma \ref{lemSign},
$$\mr{rk}_{\mcO_S}\Lie_ψ(X) = \begin{cases} 0 & \mr{if}\ ψ\in Ψ_0\\ n & \mr{if}\ ψ\in Ψ_1.\end{cases}$$
Such hermitian $\mcO_E$-$\mcO_{K}$-modules exist if and only if the inertia degree of $E_0/K$ is even and exactly half of the elements of $Ψ_0$ factor over the first component $E^u→E^u_0\times 0$. From now on, we assume that these conditions are satisfied.
Again, we fix a hermitian $\mcO_E$-$\mcO_{K}$-module $(\mbX_{E/E_0/K},ι,λ)$ of signature $(0,n)$ over $\mbF$ and as in the non-split case, such a choice is unique up to quasi-isogeny.

\begin{defn}
The functor $\mcN_{E/E_0/K,(0,n)}$ on schemes over $\Spf \mcO_{\breve E}$ associates to $S$ the set of isomorphism classes of quadruples $(X,ι,λ,ρ)$ where $(X,ι,λ)$ is a supersingular hermitian $\mcO_E$-$\mcO_{K}$-module of signature $(0,n)$ and where
$$ρ:X\times_S\overbar{S}→\mbX_{E/E_0/K,(0,n)}\times_{\Spec\mbF} \ob{S}$$
is an $\mcO_E$-linear quasi-isogeny which preserves the polarizations.
\end{defn}

\begin{prop}\label{propStructureSplit}
The functor $\mcN_{E/E_0/K,(0,n)}$ is representable by a formal scheme which is étale over $\Spf \mcO_{\breve E}$. Let $U:=U(\mbX_{E/E_0/K,(0,n)})$ (resp.\! $H$) denote the group of $E$-linear quasi-isogenies (resp. isomorphisms) of the framing object that preserve the polarization. Then there is an $U$-equivariant isomorphism of formal schemes
$$\mcN_{E/E_0/K, (0,n)} \iso \coprod_{U/H} \Spf \mcO_{\breve E}.$$
\end{prop}
\begin{proof}
Representability follows from \cite{RZ}. The formal scheme is étale since the Hodge filtration is already uniquely determined by the signature condition. The group $U$ acts on $\mcN_{E/E_0/K,(0,n)}$ by composition in the framing,
$$g.(X,ρ) = (X,gρ).$$
The stabilizer of the $\mbF$-valued point $(\mbX_{E/E_0/K,(0,n)},\mr{id})$ is the subgroup $H\subset U$. The action of $U$ is transitive on $\mbF$-points, which is analogous to Lemma \ref{lemExistCan}. The result follows from the étaleness of $\mcN_{E/E_0/K,(0,n)}$ over $\Spf \mcO_{\breve E}$.
\end{proof}

We now apply this to the study of $\mcZ(\mcO_A)$ as in the previous section but where now $A/E$ is allowed to be an étale algebra. Let again $A^0:=A^{*=\mr{id}}$ which is finite étale over $E_0$. Let
$$A_0:=\prod_{i\in I}A_{0,i}$$
be its decomposition into fields and set $A_i:=A_{0,i} \tensor_{E_0} E$. Then $A=\prod_iA_i$ and each factor is stable under the Rosati involution. The isocrystal $N:=N(\mbX_{E_0,(1,n-1)})$ of $\mbX_{E_0,(1,n-1)}$ becomes an $A$-module and hence decomposes, $N = \prod_{i\in I} N_i.$ This decomposition is preserved by the Frobenius and is orthogonal with respect to the skew-hermitian $\breve E_0$-valued form on $N$. So each factor $N_i$ is a skew-hermitian $A_i$-$E_0$-isocrystal and we set $n_i := \mr{rk}_{A_i\tensor_{E_0} \breve E_0}(N_i)$.

\begin{defn}\label{defParityIndex}
An index $i\in I$ is called \emph{odd} if $A_i/A_{0,i}$ is a field extension and if $N_i$ is an odd skew-hermitian $A_i$-$E_0$-isocrystal, see Proposition \ref{propIsocrys}. Otherwise, there exists a self-dual and Frobenius-stable $\mcO_{A_i}\tensor_{\mcO_{E_0}} \mcO_{\breve E_0}$-lattice of signature $(0,n_i)$ in $N_i$ and we call $i$ \emph{even}.
\end{defn}
An equivalent way to define the parity of an index is as follows. Let $α$ be the operator from \eqref{eqalpha} and let $V := N^{α=1}$. Then $V$ is a skew-hermitian $A$-$E_0$-module and there is a decomposition $V = \prod_{i\in I} V_i$ such that each $V_i$ is a skew-hermitian $A_i$-$E_0$-module. Then an index $i$ is called even (resp.\! odd) if there exists (resp.\! does not exist) a self-dual $\mcO_{A_i}$-lattice in $V_i$.
Since $N$ itself is odd, there is an odd number of odd indices.

\begin{lem}\label{lemoddindex}
If there is more than one odd index, then $\mcZ(\mcO_A)= \emptyset$.
\end{lem}
\begin{proof}
Using the idempotents in $\mcO_A=\prod_{i\in I}\mcO_{A_i}$, any point $(X,ι,λ)\in \mcZ(\mcO_A)(S)$ has an orthogonal decomposition
$$(X,ι,λ)=\prod_{i\in I} (X_i,ι,λ)$$
where each factor $(X_i,ι,λ)$ is a supersingular hermitian $\mcO_{A_i}$-$\mcO_{E_0}$-module. Also, assuming that $S$ is connected, each factor has a well-defined signature $(r_i,s_i)$ and these signatures add up to $(1,n-1)$. In particular, there is exactly one index $i_{0}\in I$ with $r_{i_0} = 1$ and $r_i = 0$ for all $i\neq i_0$. Then $N_{i_{0}}$ is odd and all other indices are even by Lemma \ref{lemExistframing}.
\end{proof}

From now on, we assume that there is a unique odd index $i_0$. We denote by $U_{A_i}(N_i)$ the group of $A_i$-linear automorphisms of $N_i$ which preserve the polarization. Similarly, $U_A(N) = \prod_{i\in I} U_{A_i}(N_i)$ denotes the group of $A$-linear automorphisms of $N$ which preserve the polarization. For even $i$, we also fix some self-dual $\mcO_{A_i}$-stable Dieudonné lattice $M_i\subset N_i$ of signature $(0,n_i)$ and denote by $H_i\subset U_{A_i}(N_i)$ its stabilizer.
At the index $i_0$, we choose the decomposition
$$Ψ := \Hom_{E_0}(A^u_{i_0},\breve E_0) = Ψ_0\sqcup Ψ_1$$
that extends the decomposition from $E/E_0$. Proposition \ref{propStructureSplit} and the proof of Lemma \ref{lemoddindex} imply the following result.
\begin{prop}\label{propgeometale}
There is an $U_A(N)$-equivariant isomorphism of formal schemes over $\Spf \mcO_{\breve A_{i_0}}$,
$$\mcZ(\mcO_A) \iso \coprod_{ψ_0\in Ψ_0}\left(\mcN^{ψ_0}_{A_{0,i_0}/E_0,(1,n_{i_0}-1)}\right) \times \prod_{i\neq i_0} U_{A_i}(N_i)/H_i.$$\qed
\end{prop}
By Theorem \ref{thmMain1}, each connected component can be identified with $\mcN_{A_{0,i_0},(1,n_{i_0}-1)}$.

\part{Application to the Arithmetic Fundamental Lemma}

In the following two sections, we recall the Jacquet-Rallis Fundamental Lemma (FL) and the Arithmetic Fundamental Lemma (AFL) in both the group and the Lie algebra version. In Section \ref{sect:analytic}, we recall the description of the various orbital integrals in terms of lattices from \cite{RTZ}. We use these results to reformulate the FL and the AFL uniformly for group and Lie algebra, at least in the artinian case. Finally, we will prove our main result, see Theorems \ref{thmAFLgenmain} and \ref{thmMainEtale}, together with their corollaries.

\section{The Fundamental Lemma}
\subsection{Symmetric space side}
Let $p>2$ be a prime and let $E/E_0$ be an unramified quadratic extension of $p$-adic local fields. We denote their rings of integers by $\mcO_{E_0}\subset \mcO_E$. Let $q$ be the cardinality of the residue field of $\mcO_{E_0}$ and let $σ$ or $a\mapsto \overbar{a}$ denote the Galois conjugation on $E$. Let $v$ be the normalized valuation on $E_0$ and let $|\cdot | = q^{-v(\cdot )}$  be the associated absolute value. Let $η:E_0^\times→\{\pm 1\}, a\mapsto (-1)^{v(a)}$ be the quadratic character associated to $E/E_0$ by local class field theory.

We fix an $E_0$-vector space $W_0$ of dimension $n-1$ with $n\geq 2$ and set $W:=E\tensor_{E_0} W_0$. We also form $V_0 := W_0\oplus E_0u$ and $V:=E\tensor_{E_0} V_0$, where $u$ is some additional vector. Via the embedding
$$GL(W)\hookrightarrow GL(V),\ \ g\mapsto \left(\begin{smallmatrix} g & \\ & 1\end{smallmatrix}\right),$$
$GL(W)$ acts by conjugation on $\End(V)$.

\begin{defn}\label{defrs}
An element $γ\in \End(V)$ is \emph{regular semi-simple (with respect to the decomposition $V = W \oplus Eu$)} if its stabilizer in $GL(W)$ is trivial and if its orbit is Zariski closed.
For a subset $X\subset \End(V)$, we denote by $X_{\mr{rs}}$ the regular semi-simple elements in $X$.
\end{defn}

\begin{lem}[\protect{\cite[Lemma 2.1]{Zhang}}]\label{lemrs}
The element $γ\in \End(V)$ is regular semi-simple if and only if
$$\{γ^iu\}_{i\geq 0}\ \ \tand\ \ \{(u^\vee)γ^i\}_{i\geq 0}$$
generate $V$ (resp.\! $V^\vee$). Here, $u^\vee:W\oplus Eu→E$ is the linear form $(w,λu)\mapsto λ$.\qed
\end{lem}

Let $S(E_0)$ be the symmetric space
$$S(E_0):=\{γ\in GL(V)\mid γ\overbar{γ}=1\}.$$
It is stable under the action of $GL(W_0)$. We form the set-theoretic quotient for the conjugation action,
$$[S(E_0)_{\mathrm{rs}}]:=GL(W_0)\backslash S(E_0)_{\mathrm{rs}}.$$
Let us fix some $\mcO_{E_0}$-lattice $Λ_0\subset W_0$ and set $Λ:=(\mcO_E\tensor_{\mcO_{E_0}} Λ_0)\oplus \mcO_Eu$. We normalize the Haar measure on $GL(W_0)$ such that the volume of $GL(Λ_0)$ is $1$.

For a regular semi-simple element $γ\in S(E_0)_{\mathrm{rs}}$, for a test function $f\in C^\infty_c(S(E_0))$ and for a complex parameter $s\in \mbC$, we define the \emph{orbital integral}
$$O_γ(f,s):=\int_{GL(W_0)} f(h^{-1}γh)η(\det h)|\det h|^s dh,$$
with special value $O_{γ}(f) := O_γ(f,0)$ and derivative
$$\partial O_γ(f):=\left.\frac{d}{ds}\right|_{s=0} O_γ(f,s).$$
These integrals converge absolutely. Note that $O_γ(f)$ transforms with $η \circ \det$ under the action of $GL(W_0)$ on $γ$. 

\begin{defn}\label{defTransferFact}
For $γ\in \End(V)_{\mr{rs}}$, we define $l(γ):=[\mr{span}\{γ^iu\}_{i=0}^{n-1}:Λ]\in \mbZ$ to be the relative index of the two $\mcO_E$-lattices $\mr{span}\{γ^iu\}$ and $Λ$.
We define the \emph{transfer factor (with respect to $Λ_0$)} $Ω:\End(V)_{\mr{rs}}→\{\pm 1\}, Ω(γ):=(-1)^{l(γ)}.$ It is $η\circ \det$-invariant and hence the product $Ω(γ)O_γ(f)$ descends to the quotient $[S(E_0)_{\mathrm{rs}}]$.
\end{defn}

We also introduce a tangent space version of the notions defined so far. Let
$$\mfs(E_0) := \{y\in \End(V)\mid y+\overbar{y} = 0\}$$
be the tangent space at $1$ of $S(E_0)$. Again we form the quotient by the $GL(W_0)$-action,
$$[\mfs(E_0)_{\mr{rs}}] := GL(W_0)\backslash \mfs(E_0)_{\mr{rs}}.$$
For $y\in \mfs(E_0)_{\mr{rs}}$, for a test function $f\in C^\infty_c(\mfs(E_0))$ and for a complex parameter $s\in \mbC$, we define the orbital integrals $O_y(f,s),O_y(f)$ and $\partial O_y(f)$ by the same formulas as above. Again, $y\mapsto Ω(y)O_y(f)$ descends to the quotient $[\mfs(E_0)_{\mr{rs}}]$.

\subsection{Unitary Side and orbit matching}
Let $J^\flat_0$ and $J^\flat_1$ be two hermitian forms on $W$ such that $J^\flat_0$ is even and such that $J^\flat_1$ is odd. By this we mean that there exists a self-dual lattice for $J^\flat_0$, resp.\! no self-dual lattice for $J^\flat_1$. Equivalently, we assume that $η(\det(J^\flat_0)) = 1$ and $η(\det(J^\flat_1))= -1$. For $i = 0,1$, we extend the form $J_i^\flat$ to a form $J_i$ on $V$ by setting $J_i(u,u) = 1$ and $u\perp W$. Let $U(J^\flat_i)$ (resp.\! $U(J_i)$) be the associated unitary group and $\mfu(J^\flat_i)$ (resp. $\mfu(J_i)$) its Lie algebra. Then $U(J_i)$ and $\mfu(J_i)$ are subsets of $\End(V)$ and our definition of regular semi-simple applies to them. 
The group $U(J^\flat_i)$ acts by conjugation and we form the quotients
$$\begin{aligned}[]
[U(J_i)_{\mathrm{rs}}] & := U(J^\flat_i)\backslash U(J_i)_{\mr{rs}}\\
[\mfu(J_i)_{\mr{rs}}] & := U(J^\flat_i)\backslash \mfu(J_i)_{\mr{rs}}.
\end{aligned}$$

\begin{defn}
Two elements $γ\in S(E_0)_{\mathrm{rs}}$ and $g\in U(J_i)_{\mathrm{rs}}$ (resp. $y\in \mfs(E_0)_{\mr{rs}}$ and $x\in \mfu(J_i)_{\mr{rs}}$) are said to \emph{match} if they are conjugate under $GL(W)$ in $\End(V)$.
\end{defn}

\begin{lem}[\protect{\cite[Lemma 2.3]{Zhang}}]\label{lemmatching}
The matching relation induces bijections
$$α:[S(E_0)_{\mathrm{rs}}]\iso [U(J_0)_{\mathrm{rs}}]\sqcup [U(J_1)_{\mathrm{rs}}]$$
and
$$α:[\mfs(E_0)_{\mathrm{rs}}]\iso [\mfu(J_0)_{\mathrm{rs}}]\sqcup [\mfu(J_1)_{\mathrm{rs}}].$$\qed
\end{lem}

To normalize the Haar measure on $U(J^\flat_0)$, we fix a self-dual lattice $L\subset (W,J^\flat_0)$ and give volume $1$ to its stabilizer $K^\flat_0\subset U(J^\flat_0)$. The normalization of the Haar measure on $U(J^\flat_1)$ is not important for us.

For a test function $f'\in C^\infty_c(U(J_i))$ (resp.\! $f'\in C^\infty_c(\mfu(J_i))$) and a regular semi-simple element $g\in U(J_i)_{\mr{rs}}$ (resp.\! $x\in \mfu(J_i)_{\mr{rs}}$), we define the orbital integral
$$O_g(f'):=\int_{U(J^\flat_i)}f'(h^{-1}gh)dh,\ \ \left(\text{resp.}\ O_x(f'):=\int_{U(J^\flat_i)}f'(h^{-1}xh)dh\right).$$
For fixed $f'$ this function is invariant under the conjugation action of $U(J^\flat_i)$ on $g$ (resp.\! on $x$) and hence descends to the quotient $[U(J_i)_{\mr{rs}}]$ (resp.\! $[\mfu(J_i)_{\mr{rs}}]$).

\begin{defn}\label{deftransfer}
A function $f\in C^\infty_c(S(E_0))$ and a pair of functions $(f'_0,f'_1)$ in $C^\infty_c(U(J_0))\times C^\infty_c(U(J_1))$
are said to be \emph{transfers of each other} if, for all $γ\in S(E_0)_{\mr{rs}}$, there is an equality
$$Ω(γ)O_{γ}(f) = \begin{cases} O_g(f'_0) & \tif γ\text{ matches }g\in U(J_0)_{\mr{rs}}\\
O_g(f'_1) & \tif γ\text{ matches }g\in U(J_1)_{\mr{rs}}.\end{cases}$$

Similarly, a function $f\in C^\infty_c(\mfs(E_0))$ and a pair of functions $(f'_0,f'_1)\in C^\infty_c(\mfu(J_0))\times C^\infty_c(\mfu(J_1))$ are said to be transfers of each other if, for all $y\in \mfs(E_0)_{\mr{rs}}$, there is an equality
$$Ω(y)O_{y}(f) = \begin{cases} O_x(f'_0) & \tif y\text{ matches }x\in \mfu(J_0)_{\mr{rs}}\\
O_x(f'_1) & \tif y\text{ matches }x\in \mfu(J_1)_{\mr{rs}}.\end{cases}$$
\end{defn}

\subsection{Jacquet-Rallis Fundamental Lemma}
\label{subsect:FL}

Recall that for the definition of the transfer factor, Definition \ref{defTransferFact}, we fixed some $\mcO_{E_0}$-lattice $Λ_0\subset W_0$ and formed the lattice $Λ\subset V$. We define $S(\mcO_{E_0}):= S(E_0)\cap \End(Λ)$ and $\mfs(\mcO_{E_0}) := \mfs(E_0)\cap \End(Λ)$. Also, let $K_0\subset U(J_0)$ denote the stabilizer of $L\oplus \mcO_Eu$ where $L$ is some self-dual lattice in $(W,J^\flat_0)$. Similarly, we define $\mfu(J_0)(\mcO_{E_0}) := \mfu(J_0)\cap \End(L\oplus \mcO_Eu)$.

\begin{thm}[Jacquet-Rallis Fundamental Lemma]\label{conjJRFL}
The function $1_{S(\mcO_{E_0})}$ and the pair $(1_{K_0},0)$ are transfers of each other. Equivalently, for all $γ\in S(E_0)_{\mr{rs}}$,
\begin{equation}\label{eqJR}
Ω(γ)O_{γ}(1_{S(\mcO_{E_0})}) = \begin{cases} O_{g}(1_{K_0})& \tif γ\text{ matches }g\in U(J_0)\\
0& \tif γ\text{ matches }g\in U(J_1).\end{cases} \tag{FL$_{E/E_0,(V,J_0),u,g}$}
\end{equation}
\end{thm}

The proof of this Theorem was completed by Beuzart-Plessis \cite{BP} recently. More precisely, he proves the Lie algebra version (cf. below) which by \cite[Section 2.6]{Yun} implies the group version. Previously, the equal characteristic analogue of the theorem was known by Yun \cite{Yun} in the case $p>n$, from which Gordon \cite{Gordon} deduced the $p$-adic case for $p$ sufficiently large.

\begin{rmk}\label{rmkNotation}
(1) Note that the left hand side does not depend on the choice of the lattice $Λ_0$. Namely if we replace it by $hΛ_0$, $h\in GL(W_0)$, then $Ω$ is changed by the sign $(-1)^{v(\det h)}$. But also
$$O_γ(1_{hS(\mcO_{E_0})h^{-1}}) = O_{h^{-1}γh} (1_{S(\mcO_{E_0})}) = (-1)^{v(\det h)} O_γ(1_{S(\mcO_{E_0})}).$$

(2) The quadruple $(E/E_0,(V,J_0),u,g)$ (in the case $g\in U(J_0)$) is sufficient to formulate the corresponding Fundamental Lemma identity. Namely we can define $W := u^\perp$ with form $J_0^\flat := J_0\vert_W$. The form $J^\flat_1$ can be chosen arbitrarily since it does not play a role in the identity. Finally, the left hand side of $(\mr{FL}_{E/E_0,(V,J_0),u,g})$ does not depend on the chosen $E_0$-structure $W_0\subset W$, which will follow from Corollary \ref{lemorbints} below.
\end{rmk}

\begin{thm}[Jacquet-Rallis Fundamental Lemma, Lie algebra version, \protect{\cite[Theorem 1]{BP}}]
The function $1_{\mfs(\mcO_{E_0})}$ and the pair $(1_{\mfu(J_0)(\mcO_{E_0})},0)$ are transfers of each other. Equivalently, for all $y\in \mfs(E_0)_{\mr{rs}}$,
\begin{equation}\label{eqJRLie}
Ω(y)O_{y}(1_{\mfs(\mcO_{E_0})}) = \begin{cases} O_{x}(1_{\mfu(J_0)(\mcO_{E_0})})& \tif y\text{ matches }x\in \mfu(J_0)\\
0& \tif y\text{ matches }x\in \mfu(J_1).\end{cases} \tag{$\mathfrak{fl}_{E/E_0,(V,J_0),u,x}$}
\end{equation}
\end{thm}

As with the group version, the left hand side does not depend on the choice of $Λ_0$. Similarly, the quadruple $(E/E_0,(V,J_0),u,x)$ is enough to formulate the identity $(\mathfrak{fl}_{E/E_0,(V,J_0),u,x})$, for $x\in \mfu(J_0)$. Both orbital integrals appearing in the Fundamental Lemma can be expressed in terms of lattices. We formulate this now for the unitary side. For the symmetric space side, see Corollaries \ref{lemorbints} and \ref{lemorbintsLie}. We denote by $Λ^\vee$ the $J_0$-dual of a lattice $Λ\subset V$.

\begin{lem}[\protect{\cite[Lemma 7.1]{RTZ}}]\label{lemorbunitary}
(1) Let $g\in U(J_0)_{\mr{rs}}$ be regular semi-simple and let $L = \mcO_E[g]u\subset V$ be the $g$-stable lattice spanned by $u$. Then
$$O_g(1_{K_0}) = \left| \{Λ\subset V\ \mid\ L\subset Λ \subset L^\vee,gΛ = Λ, Λ^\vee = Λ\}\right|.$$

(2) Let $x\in \mfu(J_0)_{\mr{rs}}$ be regular semi-simple and let $L = \mcO_E[x]u\subset V$ be the $x$-stable lattice spanned by $u$. Then
$$O_x(1_{\mfu(J_0)(\mcO_{E_0})}) = \left| \{Λ\subset V\ \mid\ L\subset Λ \subset L^\vee,xΛ \subseteq Λ, Λ^\vee = Λ\}\right|.$$
\end{lem}
\begin{proof}
Part (1) is precisely \cite[Lemma 7.1]{RTZ} and part (2) is proved in the same way.
\end{proof}

\section{The Arithmetic Fundamental Lemma}
\subsection{Intersection numbers}
\label{ss:intersect}
Let $n\geq 2$ and let $\mcN_{E_0,(1,n-2)}$ be the RZ-space from Definition \ref{defRZ} with framing object $\mbX_{E_0,(1,n-2)}$. Let $\mbY_{E_0} := \mbX_{E_0, (0,1)}$ be the hermitian $\mcO_E$-module of signature $(0,1)$ over $\mbF$ and take $\mbX_{E_0,(1,n-1)}:= \mbX_{E_0,(1,n-2)}\times \mbY_{E_0}$ as the framing object for $\mcN_{E_0,(1,n-1)}$. We make $\Hom^0(\mbY_{E_0}, \mbX_{E_0,(1,n-2)})$ into a hermitian $E$-vector space with form $h$ defined by
$$h(x,y)\mr{id_{\mbY_{E_0}}} = λ_{\mbY_{E_0}}^{-1}\circ x^\vee \circ λ_{\mbX_{E_0,(1,n-2)}} \circ y,$$
see \cite[Definition 3.1]{KRlocal}. Let $V(\mbY_{E_0})$ resp. $V(\mbX_{E_0,(1,n-2)})$ denote the hermitian $E$-vector spaces corresponding to the hermitian $\mcO_E$-modules under the equivalence from Proposition \ref{propIsocrys}. Then $V(\mbY_{E_0})$ is even one-dimensional and $V(\mbX_{E_0,(1,n-2)})$ is odd $(n-1)$-dimensional by Lemma \ref{lemExistframing}. It follows that $h$ is odd and we fix an isometry
$$(\Hom^0(\mbY_{E_0}, \mbX_{E_0,(1,n-2)}),h) \iso (W,J_1^\flat)$$
which we extend to an isometry
$$
\Hom^0(\mbY_{E_0}, \mbX_{E_0,(1,n-1)}) = \Hom^0(\mbY_{E_0}, \mbX_{E_0,(1,n-2)}) \times E\cdot \mr{id}_{\mbY_{E_0}}
\iso W \oplus Eu = V
$$
by sending $(0,\mr{id}_{\mbY_{E_0}})$ to $u$. Then $\End^0(\mbX_{E_0,(1,n-1)})$ acts on $V$ which induces an identification
\begin{equation}\label{eqIdent}
\End^0(\mbX_{E_0,(1,n-1)}) \iso \End(V)
\end{equation}
that is equivariant for the Rosati involution on the left and the adjoint involution of $J_1$ on the right. The product decompositions $\mbX_{E_0,(1,n-1)} = \mbX_{E_0,(1,n-2)}\times \mbY_{E_0}$ and $V = W\times Eu$ give rise to projection (resp.\! inclusion) operators to (resp.\! from) the vector spaces $\End^0(\mbX_{E_0,(1,n-2)})$, $\Hom(\mbY_{E_0},\mbX_{E_0,(1,n-2)})$, etc.\! and $\End(W),\Hom(Eu,W)$, etc. The identification in \eqref{eqIdent} is compatible with all these homomorphisms. Finally, the isomorphism \eqref{eqIdent} identifies the unitary group $U(J^\flat_1)$ (resp.\! $U(J_1)$) with the group of quasi-isogenies of $\mbX_{E_0,(1,n-2)}$ (resp.\! $\mbX_{E_0,(1,n-1)}$). From now on, we will take the identification from \eqref{eqIdent} as self-evident.

By Proposition \ref{propSmoothness}, there is a unique deformation $\mcY_{E_0}$ of the hermitian $\mcO_E$-module $\mbY_{E_0}$ to $\Spf \mcO_{\breve E}$. We define $\mcY_{E_0}$ on every $\Spf \mcO_{\breve E}$-scheme by base change. This induces a closed immersion
\begin{equation}\label{eqDelta}
δ:\mcN_{E_0,(1,n-2)}\hookrightarrow \mcN_{E_0,(1,n-1)},\ \ X\mapsto X\times \mcY_{E_0}
\end{equation}
which is equivariant with respect to the inclusion $U(J^\flat_1)\hookrightarrow U(J_1)$. Here, the groups act on the RZ-spaces by composition in the framing, $g.(X,ρ) := (X,gρ).$
We consider the graph of $δ$,
$$\Delta:\mcN_{E_0,(1,n-2)}→\mcN_{E_0,(1,n-2)}\times_{\Spf \mcO_{\breve E}} \mcN_{E_0,(1,n-1)}.$$
By abuse of notation, we denote its image also by $\Delta$. Note that the source is regular of dimension $n-1$ while the target is regular of dimension $2(n-1)$. Hence $\Delta$ defines a cycle in middle dimension.  For $g\in U(J_1)$, we denote by
$$\Delta_g := (1,g)\Delta$$
its translate under $g$.

\begin{lem}[\protect{\cite[Lemma 2.8]{Zhang}}]\label{lemfininter}
For regular semi-simple $g\in U(J_1)_{\mr{rs}}$, the schematic intersection
$Δ\cap Δ_g$ is a projective scheme over $\Spf \mcO_{\breve E}$.
\end{lem}
\begin{proof}
The proof in \cite{Zhang} is global and only applies to the case $E_0=\mbQ_p$, so we provide a local argument in the general case. The schematic intersection $\Delta\cap \Delta_g$ has the following moduli-theoretic interpretation. The image $δ(\mcN_{E_0,(1,n-2)})\subset \mcN_{E_0,(1,n-1)}$ can be identified with the KR-divisor $\mcZ(u)$ from Definition \ref{defKRdiv}. Via the second projection, $\Delta \cap \Delta_g$ can then be identified with the formal scheme $\mcZ(u)\cap \mcZ(g)$, where $\mcZ(g)$ is defined in Definition \ref{defZg}. We assume that $g$ is integral over $\mcO_E$, otherwise $\mcZ(g)=\emptyset$. Let
$$L:=\mcO_E [g]u\subset \Hom^0(\mbY_{E_0},\mbX_{E_0,(1,n-1)})$$
be the lattice generated by the $\{g^iu\}_{i\geq 0}$. It has full rank by Lemma \ref{lemrs}. Since $\mcZ(g^iu)\subset \mcZ(u)\cap \mcZ(g)$ for all $i\geq 0$, the natural quasi-isogeny
\begin{equation}\label{eqisog}
L\tensor_{\mcO_E} \mbY_{E_0}→\mbX_{E_0,(1,n-1)}
\end{equation}
lifts to an isogeny $L \tensor_{\mcO_E}\mcY_{E_0} → \mcX$ on $\mcZ(u)\cap \mcZ(g)$, where $\mcX$ denotes the universal hermitian $\mcO_E$-module. The space of isogenies from $L\tensor_{\mcO_E}\mcY_{E_0}$ of height equal to the height of \eqref{eqisog} is representable by a projective scheme over $\mcO_{\breve E}$. The intersection $\mcZ(u)\cap \mcZ(g)$ is then the $p$-adic completion of the closed subscheme of isogenies such that:\\
1) The $\mcO_E[g]$-action descends to the quotient.\\
2) The $\mcO_E$-action on the quotient is of signature $(1,n-1)$ on its Lie algebra.\\
3) The (non-principal) polarization on $L\tensor_{\mcO_E}\mbY_{E_0}$ induced from the hermitian form on $L$ and the polarization on $\mbY_{E_0}$ descends to a (for degree reasons necessarily principal) polarization on the quotient.\\
This closed subscheme has empty generic fiber since $L\tensor_{\mcO_E}\mcY_{E_0}$ has signature $(0,n)$.
\end{proof}

\begin{defn}\label{defHigherTor}
For regular semi-simple $g\in U(J_1)$, we define
$$\mr{Int}(g) := χ(\mcO_{\Delta} \tensor ^{\mbL} \mcO_{\Delta_g}).$$
This number is finite by the previous lemma. Moreover, the function $U(J_1)_{\mr{rs}}\ni g\mapsto \mr{Int}(g)$ descends to the quotient $[U(J_1)_{\mr{rs}}]$.
\end{defn}

\begin{rmk}\label{rmkModg}
All spaces occurring in the definition of $\mr{Int}(g)$ are regular. So if the schematic intersection $\Delta\cap \Delta_g$ is $0$-dimensional, then there is an equality
$$\mr{Int}(g) = \mr{len}_{\mcO_{\breve E}} \mcO_{\Delta \cap \Delta_g},$$
see \cite[Proposition 4.2]{RTZ}. 
\end{rmk}

\begin{defn}\label{defInter}
A quasi-endomorphism $x\in \End^0_E(\mbX_{E_0,(1,n-1)})$ is called \emph{artinian (with respect to the quasi-homomorphism $u\in \Hom^0(\mbY_{E_0},\mbX_{E_0,(1,n-1)})$)}, if the intersection $\mcZ(x)\cap \mcZ(u)$ is an artinian scheme. For artinian $x$, we define
$$\mr{Int}(x) := \mr{len}_{\mcO_{\breve E}} \mcO_{\mcZ(x)\cap \mcZ(u)}.$$
\end{defn}

\begin{rmk}
There is no known group-theoretic characterization of the artinian elements  in $U(J_1)_{\mr{rs}}$.
\end{rmk}

It is also possible to give a moduli description of $\mcZ(x)\cap \mcZ(u)$ in the smaller space $\mcZ(u) = δ(\mcN_{E_0,(1,n-2)})$. If $x$ has the form
$$x = \left( \begin{smallmatrix} x^\flat & v \\ w & d\end{smallmatrix} \right)\in \End^0(\mbX_{E_0,(1,n-2)}\times \mbY_{E_0}),$$
then
\begin{equation}\label{eqModx}
\mcZ(x)\cap \mcZ(u) = \begin{cases} \emptyset & \text{if } d\notin \mcO_E\\
\mcZ(x^\flat)\cap \mcZ(v) \cap \mcZ(w^*) & \text{otherwise}\end{cases}
\end{equation}
where $w^*:\mbY_{E_0}→\mbX_{E_0,(1,n-2)}$ is the Rosati adjoint of $w$, see Definition \ref{defRosati}.

\begin{rmk}\label{rmkModx}
If $x\in \mfu(J_1)$, then $x$ has the form
$$x = \left( \begin{smallmatrix} x^\flat & j \\ -j^* & d\end{smallmatrix} \right)$$
with $x^\flat\in \mfu(J^\flat_1)$ and $\ob{d} = -d$. So in this case,
$$\Delta \cap \Delta_x \iso \begin{cases} \emptyset & \tif d\notin \mcO_E\\
\mcZ(x^\flat)\cap \mcZ(j) & \text{otherwise.}\end{cases}$$
\end{rmk}

\subsection{The Arithmetic Fundamental Lemma}
 
The Fundamental Lemma gives an expression of the orbital integral function $O_γ(1_{S(\mcO_{E_0})})$ on the unitary side. By contrast, the Arithmetic Fundamental Lemma conjecturally expresses the derivative $\partial O_γ(1_{S(\mcO_{E_0})})$ on the unitary side whenever $γ$ matches an element $g\in U(J_1)_{\mr{rs}}$. Note that for such $γ$, the orbital integral $O_γ(1_{S(\mcO_{E_0})})$ vanishes and thus $\partial O_γ(1_{S(\mcO_{E_0})})$ is $η\circ \det$-invariant.

\begin{conj}[Arithmetic Fundamental Lemma, Group version, \protect{\cite{Zhang}}]\label{conjAFLgrp}
Let $γ\in S(E_0)_{\mr{rs}}$ be a regular semi-simple element that matches an element $g\in U(J_1)$. Then there is an equality
\begin{equation}
Ω(γ)\partial O_γ(1_{S(\mcO_{E_0})}) = -\mr{Int}(g)\log(q). \tag{$\mr{AFL}_{E/E_0,(V,J_1),u,g}$}
\end{equation}
\end{conj}

For the Lie algebra version, we have to restrict to artinian elements since we defined the intersection product only in this case, see Definition \ref{defInter}.

\begin{conj}[Arithmetic Fundamental Lemma, Lie algebra version]\label{conjAFLLie}
For any $y\in \mfs(E_0)_{\mr{rs}}$ matching an artinian element $x\in \mfu(J_1)$, there is an equality
\begin{equation}
\tag{$\mathfrak{afl}_{E/E_0,(V,J_1),u,x}$}
Ω(y)\partial O_y(1_{\mfs(\mcO_{E_0})}) = -\mr{Int}(x)\log(q).
\end{equation}
\end{conj}

\begin{rmk}
Just as in the case of the Fundamental Lemma, see Remark \ref{rmkNotation}, the left hand side of the AFL identities does not depend on the chosen lattice $Λ_0$. Also, it does not depend on the chosen $E_0$-structure $W_0\subset W$, which will follow from the Corollaries \ref{lemorbints} and \ref{lemorbintsLie} below. Hence the quadruples $(E/E_0,(V,J_1),u,g)$ and $(E/E_0,(V,J_1),u,x)$ are sufficient to formulate the respective identity.
\end{rmk}

The group and the Lie algebra version of the AFL are related by the following result.

\begin{prop}[\protect{\cite{Mihatsch}}]\label{propEqAFL}
Assume that $q \geq n+2$. Then the AFL for all artinian elements $g\in U(J_1)_{\mr{rs}}$ is equivalent to the AFL in the Lie algebra formulation for all artinian $x\in \mfu(J_1)_{\mr{rs}}$.\footnote{In \cite{Mihatsch}, it is specified for which $x$ one needs $(\mr{AFl}_{E/E_0,(V,J_1),u,x})$ to obtain $(\mr{AFL}_{E/E_0,(V,J_1),u,g})$ and conversely.}
\end{prop}

The AFL conjecture has been verified for $n\leq 3$ in \cite{Zhang}. Note that if $n\leq 3$, then any regular semi-simple element $g\in U(J_1)_{\mr{rs}}$ is artinian. There exists a slight simplification of this computation in \cite{Mihatsch} which relies on Proposition \ref{propEqAFL}.

More cases of the AFL for any $n$, but under restrictive conditions on $g$, have been verified by Rapoport, Terstiege and Zhang in \cite{RTZ}. Note that in these cases, $g$ is also artinian. Their proof was subsequently simplified by Li and Zhu \cite{LZ1,LZ2} and He, Li and Zhu \cite{HLZ}.

Before we continue, we would like to modify the AFL for Lie algebras slightly. Namely let $y\in \mfs(E_0)_{\mr{rs}}$ match $x\in \mfu(J_1)$ of the form
$$x = \left( \begin{smallmatrix} x^\flat & j \\ -j^* & d\end{smallmatrix} \right).$$
Then $d$ is also the lower right entry of the matrix $y$. If $d\notin \mcO_E$, then it is easy to see that both sides of ($\mathfrak{afl}_{E/E_0,(V,J_1),u,x}$) vanish. If instead $d\in \mcO_E$, then we can replace $y$ by $y-d\cdot \mr{id}_V$ and $x$ by $x-d\cdot \mr{id}_V$ without changing either side of ($\mathfrak{afl}_{E/E_0,(V,J_1),u,x}$). Furthermore, $y-d\cdot\mr{id}_V$ lies in $\mfs(E_0)_{\mr{rs}}$ and matches $x-d\cdot\mr{id}_V\in \mfu(J_1)_{\mr{rs}}$.

\begin{defn}\label{defZero}
We define
$$\begin{aligned}
\mfs(E_0)^0 & := \left\{ \left(\begin{smallmatrix} y^\flat & w \\ v & d\end{smallmatrix}\right)\in \mfs(E_0) \mid d = 0\right\}\ \ \tand\\
\mfu(J)^0 & := \left\{ \left(\begin{smallmatrix} x^\flat & j \\ - j^* & d\end{smallmatrix}\right)\in \mfu(J) \mid d = 0\right\}.\end{aligned}$$
Then the matching relation induces a bijection
$$[\mfs(E_0)^0_{\mr{rs}}]\iso [\mfu(J_0)^0_{\mr{rs}}] \sqcup [\mfu(J_1)^0_{\mr{rs}}]$$
and it is enough to consider the Lie algebra formulation of the AFL for $x\in \mfu(J_1)^0$.
\end{defn}

\section{Orbital integrals as lattice counts}\label{sect:analytic}

In this section, we recall the expression of the orbital integrals $O_γ(1_K)$ and $\partial O_γ(1_K)$ in terms of lattices from \cite[Section 7]{RTZ}. We deduce analogous results for the Lie algebra formulation.

\subsection{Orbital integrals on $S_n$}
\label{sect:anaGrp}

Let $γ\in S(E_0)_{\mr{rs}}$ match the element $g\in U(J_0)_{\mr{rs}}\sqcup U(J_1)_{\mr{rs}}$. From now on, we consider $V$ with the hermitian form $J\in \{J_0,J_1\}$ determined by $g$. Recall that $V=W\os{\perp}{\oplus}  Eu$ with $(u,u)=1$ and note that $u,gu,\ldots,g^{n-1}u$ is a basis of $V$ since $g$ is regular semi-simple. We define a $σ$-linear involution $τ:V→V$ by $τ(g^iu)=g^{-i}u$ for $i=0,\ldots,n-1$.

\begin{rmk}\label{rmkAdjoint}
The involution $τ$ can also be defined as follows. The vector $u$ defines an isomorphism of $E[g]$-modules, $E[g]\iso E[g]u = V$. Under this isomorphism, $τ$ corresponds to the adjoint involution with respect to the hermitian form $J$ on $E[g]\subset \End(V)$.
\end{rmk}

Let $L:=\mcO_E[g]u$ be the $g$-stable lattice spanned by $u$ and denote by $L^\vee$ its dual with respect to $J$. Let
$$M:=\{Λ\subset V\mid L\subset Λ\subset L^\vee,\ gΛ\subset Λ,\ Λ^τ=Λ\}$$
and, for $i\in \mbZ$,
$$M_i:=\{Λ\in M\mid \mr{len}(Λ/L)=i\}.$$

\begin{lem}[\protect{\cite[Proof of Corollary 7.3]{RTZ}}]\label{lemorbint}
$$Ω(γ)O_γ(1_{S(\mcO_{E_0})},s)=\sum_{i\in \mbZ}(-1)^i|M_i|q^{-(i+l(γ))s}.$$
\qed
\end{lem}
(Here, $l(γ)$ was defined in Definition \ref{defTransferFact}.) Taking the value at $s=0$, resp., taking the derivative at $s=0$ yields
\begin{cor}[\protect{\cite[Corollary 7.3]{RTZ}}]\label{lemorbints}
$$Ω(γ)O_γ(1_{S(\mcO_{E_0})})=\sum_{i\in \mbZ}(-1)^i|M_i|$$
and, in the case $J = J_1$,
\begin{equation}\label{eqDeriv}
Ω(γ)\partial O_{γ}(1_{S(\mcO_{E_0})})=-\log(q) \sum_{i\in \mbZ}(-1)^ii|M_i|.
\end{equation}
\qed
\end{cor}

\subsection{Orbital integrals on $\mfs_n$}

Let $y\in \mfs(E_0)^0_{\mr{rs}}$ match the element
$$x = \left(\begin{smallmatrix} x^\flat & j \\ - j^* & \end{smallmatrix}\right)\in \mfu(J_0)^0_{\mr{rs}}\sqcup \mfu(J_1)^0_{\mr{rs}}$$
and let $J \in \{J_0,J_1\}$ be the hermitian form determined by $x$. Again, $τ:V→V$ is the adjoint involution on $V = E[x]u\subset \End(V)$, i.e.\! $τx^iu=(-1)^ix^iu$ for $i=0,\ldots,n-1$.
Let $L:=\mcO_E[x]u$ be the $x$-stable lattice spanned by $u$ and denote by $L^\vee$ its dual. Let
$$M:=\{Λ\subset V\mid L\subset Λ\subset L^\vee,\ xΛ\subset Λ,\ Λ^τ=Λ\}$$
and, for $i\in \mbZ$,
$$M_i:=\{Λ\in M\mid \mr{len}(Λ/L)=i\}.$$
Then the same formula for the orbital integral applies. Its proof is completely analogous to the one for the group version.
\begin{lem}\label{lemorbintLie}
There is an equality
$$Ω(y)O_y(1_{\mfs(\mcO_{E_0})},s)=\sum_{i\in \mbZ}(-1)^i|M_i|q^{-(i+l(y))s}.$$
\qed
\end{lem}

\begin{cor}\label{lemorbintsLie}
The value and the derivative of the orbital integral at $s=0$ have the expressions
$$Ω(y)O_y(1_{\mfs(\mcO_{E_0})})=\sum_{i\in \mbZ}(-1)^i|M_i|$$
and, in the case $J = J_1$,
\begin{equation}\label{eqDerivLie}
Ω(y)\partial O_y(1_{\mfs(\mcO_{E_0})})=-\log(q) \sum_{i\in \mbZ}(-1)^ii|M_i|.
\end{equation}
\qed
\end{cor}

Let us define $L^\flat := \mcO_E[x^\flat]j\subset W$. We denote the orthogonal projection $V→W$ by $\mr{pr}$ and denote by $\mr{pr}_u$ the projection to $Eu$.

\begin{prop}\label{propLatt}
There is an equality of sets of lattices in $V = W\oplus Eu$,
\begin{equation}\label{eqProperties}
M = \{ Λ^\flat \oplus \mcO_Eu\ \mid\ L^\flat \subset Λ^\flat \subset L^{\flat,\vee},\ x^\flat Λ^\flat \subset Λ^\flat,\ (Λ^\flat)^τ = Λ^\flat\}.
\end{equation}
\end{prop}

We first note the following formula for $v\in V$,
\begin{equation}\label{eqx}
xv = x^\flat v +(v,u)j - (j,v)u.
\end{equation}

\begin{lem}\label{lemLattice}
Let $Λ\subset V$ be any lattice such that $u\in Λ$ and $\mr{pr}_u(Λ) = \mcO_Eu$. Then $Λ = (Λ\cap W) \oplus \mcO_Eu$ and $Λ\cap W = \mr{pr}(Λ)$. Moreover, $\mr{pr}(Λ^\vee) = \mr{pr}(Λ)^\vee$.
\end{lem}
\begin{proof}
This is immediate.
\end{proof}

\begin{lem}
The inclusion $L\subset L^\vee$ holds if and only if $L^\flat \subset L^{\flat,\vee}$. In this case, $L^\flat = \mr{pr}(L)$ and $L^{\flat,\vee} = \mr{pr}(L^\vee)$.
\end{lem}
\begin{proof}
\emph{We first assume that $L\subset L^\vee$.} Then $L = \mr{pr}(L) \oplus \mcO_E u$ and $\mr{pr}(L)^\vee = \mr{pr}(L^\vee)$ by the previous lemma. Hence it is enough to show that $L^\flat = \mr{pr}(L)$.

First note that $\mr{pr}(L)$ is stable under $\mr{pr}\circ x \circ \mr{pr} = x^\flat$. Since also $j\in \mr{pr}(L)$, we get $L^\flat \subset \mr{pr}(L)$. To prove the opposite inclusion, we prove $\mr{pr}(x^iu)\subset L^\flat$ by induction.

The case $i = 0$ is clear. Then we use formula \eqref{eqx},
$$\mr{pr}(x^{i+1}u) = x^\flat x^iu + j(x^iu,u).$$
The first summand lies in $L^\flat$ by induction and by the fact that $L^\flat$ is $x^\flat$-stable. The second summand lies in $L^\flat$ since $j\in L^\flat$ and $(x^iu,u)\in \mcO_E$ by the assumption $L\subset L^\vee$.

\emph{Conversely, let us assume that $L^\flat \subset L^{\flat, \vee}$.} We need to show that $(x^iu,x^ju)\in \mcO_E$ for all $i,j$. Since $x$ is from the unitary Lie algebra, it is enough to prove $(x^iu,u)\in \mcO_E$ for all $i$. Again we prove this by induction, the cases $i = 0$ and $1$ being clear. We compute
$$-(x^{i+1}u,u) = (x^iu,j) = (x^\flat x^{i-1}u,j) + ((x^{i-1}u,u)j,j) - ((j,x^{i-1}u)u, j).$$
The first summand is integral by assumption on $L^\flat$. In the second summand, the pairing $(x^{i-1}u,u)$ is integral by induction. Hence the second summand is integral by assumption on $L^\flat$. The third summand vanishes.
\end{proof}

\begin{proof}[Proof of Proposition \ref{propLatt}]
Let $Λ\in M$. By Lemma \ref{lemLattice}, it is a direct sum, $Λ = Λ^\flat \oplus \mcO_Eu$ where $Λ^\flat = \mr{pr}(Λ)$. If $λ^\flat \in Λ^\flat$, then
$$x^\flat λ^\flat = xλ^\flat + (j,λ^\flat) u \in Λ$$
and hence $Λ^\flat$ is $x^\flat$-stable. Furthermore, $L^\flat \subset Λ^\flat \subset L^{\flat, \vee}$, since this is just the projection of the relation $L\subset Λ \subset L^\vee$. Finally, note that $τ$ commutes with the projection $\mr{pr}$. Hence $Λ^\flat$ has all the properties from \eqref{eqProperties}.

Conversely, let us now assume that $Λ^\flat$ satisfies all properties from \eqref{eqProperties}. We want to show that $Λ := Λ^\flat \oplus \mcO_Eu\in M$. By Lemma \ref{lemLattice}, $L\subset Λ\subset L^\vee$. Furthermore, $Λ$ is $τ$-stable, since both summands are. It is easy to prove that $Λ$ is also stable under $x$ which concludes the proof.
\end{proof}

\section{Uniform version of FL and AFL}
\label{sect:GeneralAFL}
\subsection{Adjoint-stable pairs}
The results of the previous section allow us to treat the AFL (resp.\! the FL) for groups and for Lie algebras at the same time. In this section, $V$ is endowed with either of the two hermitian forms, say $J\in \{J_0,J_1\}$. In particular, the adjoint involution $\End(V)\ni x\mapsto x^*$ and the dual lattice $Λ\mapsto Λ^\vee$ are taken with respect to this form.

\begin{defn}
(1) A pair $(x,j)\in \End_E(V)\times V$ is called \emph{regular semi-simple} if $E[x]j = V$.\\
(2) The pair (resp.\! the element $x$) is called \emph{adjoint-stable} if $\mcO_E[x] = \mcO_E[x^*]$.
\end{defn}

\begin{rmk}
(1) Note that any element $x\in \mfu(J)$ is adjoint-stable. An element $g\in U(J_1)$ is adjoint-stable if and only if $g^* = g^{-1}\in \mcO_E[g]$, which is equivalent to $g$ having integral characteristic polynomial.\\
(2) Let $x\in \End_E(V)$ be such that $E[x] = E[x^*]$ and let $(x,j)$ be regular semi-simple. Then also $E[x]\cdot (\ , j) = V^\vee$. In particular if $x\in \mfu(J)$ or $x\in U(J)$, then $x$ is regular semi-simple in the sense of Lemma \ref{lemrs} if and only if $(x,u)$ is a regular semi-simple pair.
\end{rmk}

\begin{defn}\label{defAFLgeneral}
Let $(x,j)$ be a regular semi-simple and adjoint-stable pair. Then we denote by $L(x,j) := \mcO_E[x]j$ the $x$-stable lattice generated by $j$.

(1) We define the $σ$-linear involution $τ(x,j):V→V$ as follows: The element $j$ induces an isomorphism $ϕ:E[x] \iso E[x]j = V$ and we set $τ(x,j)(v) = ϕ(ϕ^{-1}(v)^*)$. This is possible since $E[x] = E[x]^*$ by assumption.

(2) We define the sets
$$\begin{aligned}
M(x,j) & := \{ Λ\subset V\ \mid\ L(x,j)\subset Λ \subset L(x,j)^\vee, xΛ\subset Λ, τ(x,j)Λ = Λ\},\\
M(x,j)_i & := \{ Λ\in M\ \mid\ \mr{len}_{\mcO_E}(Λ/L(x,j)) = i\},\ \ i\in \mbZ.\end{aligned}$$

(3) For $s\in \mbC$, we define the following numbers.
$$\begin{aligned}
O(x,j;s) &:= \sum_{i\in \mbZ} (-1)^i |M(x,j)_i|q^{-is}\\
O(x,j) & := O(x,j;0) = \sum_{i\in \mbZ} (-1)^i |M(x,j)_i|\\
\partial O(x,j) & := \log(q)^{-1}\left.\frac{d}{ds}\right\vert_{s = 0} O(x,j;s) = -\sum_{i\in \mbZ} (-1)^ii |M(x,j)_i|.\end{aligned}$$
\end{defn}
We will now consider the two possibilities for $J$ separately.

\subsection{The Fundamental Lemma}
\label{subsect:JRFL}
In this section, $J = J_0$ is the even form.

\begin{defn}
Let $(x,j)$ be a regular semi-simple and adjoint-stable pair. We define
$$
I(x,j) := | \{Λ\subset V\ \mid\ L(x,j) \subset Λ \subset L(x,j)^\vee, xΛ \subset Λ, Λ^\vee = Λ\}|
$$
\end{defn}

\begin{conj}[Fundamental Lemma, uniform version]\label{conjFLgeneral}
Let $(x,j) \in \End_E(V)$ be a regular semi-simple and adjoint-stable pair. Then
\begin{equation}\label{eqFLgeneral}
I(x,j) = O(x,j). \tag{FL($x,j$)}
\end{equation}
\end{conj}

In the rest of this subsection, we explain some cases of the uniform version that are equivalent to the group and Lie algebra versions of the Fundamental Lemma. We do not know in general if the uniform version can be directly deduced from the Fundamental Lemma.

\begin{lem}\label{lemFLgenLie}
Let $(x,j)$ be a adjoint-stable pair with $x\in \mfu(J_0)$. We set $V' := V\oplus Eu'$ with form $J_0':=J_0\oplus 1$ and we define
$$x' := \left(\begin{smallmatrix} x & j\\ -j^* & \end{smallmatrix}\right) \in \mfu(J_0')^0.$$
Then $(x,j)$ is regular semi-simple if and only if $x'$ is regular semi-simple with respect to $V' = V \oplus Eu'$ in the sense of Definition \ref{defrs}. In the regular semi-simple case, $(\mr{FL}(x,j))$ is equivalent to $(\mathfrak{fl}_{E/E_0,(V',J_0'),u',x'})$.
\end{lem}
\begin{proof}
This follows from Lemma \ref{lemrs}, Lemma \ref{lemorbunitary} and Corollary \ref{lemorbintsLie}.
\end{proof}

\begin{lem}\label{lemFLvj1}
Let $(g,j)$ be a adjoint-stable pair with $g\in U(J_0)$ and $J_0(j,j)\in \mcO_{E_0}^\times$. Then $(g,j)$ is regular semi-simple if and only if $g$ is regular semi-simple with respect to $V = j^\perp \oplus Ej$ in the sense of Definition \ref{defrs}. In the regular semi-simple case, $(FL(g,j))$ is equivalent to $(\mr{FL}_{E/E_0,(V,J_0),j,g})$.
\end{lem}
\begin{proof}
This follows from Lemma \ref{lemrs}, Lemma \ref{lemorbunitary} and Corollary \ref{lemorbints}.
\end{proof}

\begin{lem}
Let $(g,j)$ be a regular semi-simple and adjoint-stable pair such that $J_0(j,j)\notin \mcO_{E_0}$. Then both sides of $(\mr{FL}(g,j))$ vanish.
\end{lem}
\begin{proof}
If $J_0(j,j)\notin \mcO_{E_0}$, then $L(x,j)\not\subset L(x,j)^\vee$.
\end{proof}

\begin{construction}
Let $(g,j)$ be a adjoint-stable pair with $g\in U(J_0)$ and $v(J_0(j,j))\geq 1$. We also assume that $q_{E_0}+1 > n$. Then we define
$V' := V \oplus E\tilde u$ and $u' := \tilde u + j$. We extend $J_0$ to $J_0'$ on $V'$ by setting $\tilde u \perp V$ and $J'_0(u',u') = 1$. We define $W' := (u')^\perp$, which is an even hermitian space.

Let $P(t)\in \mcO_E[t]$ be the characteristic polynomial of $g$. Note that $q_{E_0}+1$ is the number of residue classes mod $π_E$ of $E^1 := \{a\in E\ \mid\ \mr{Nm}_{E/E_0}(a) = 1\}$. By assumption this number is lager than $\deg(P)$ and hence there exists $a\in E^1$ such that $P(a) \not \equiv 0$ mod $π_E$. We define
$$g' := \left(\begin{smallmatrix} g & \\ & a\end{smallmatrix}\right)\in U(J'_0)$$
where the block matrix decomposition is with respect to $V' = W \oplus E\tilde u$.
\end{construction}

\begin{lem}
Let $(g,j)$ be a adjoint-stable pair with $g\in U(J_0)$ and $v(J_0(j,j))\geq 1$. We also assume that $q_{E_0}+1 > n$. Let $V',u'$ and $g'$ be as above. Then $(g,j)$ is regular semi-simple if and only if $g'$ is regular semi-simple with respect to $V' = W'\oplus Eu'$ in the sense of Definition \ref{defrs}. In the regular semi-simple case, $(\mr{FL}(g,j))$ is equivalent to $(\mr{FL}_{E/E_0,(V',J_0'),u',g'})$.
\end{lem}
\begin{proof}
As explained in Lemma \ref{lemFLvj1}, $(g',u')$ is regular semi-simple if and only if $g'$ is regular semi-simple with respect to $V' = W'\oplus Eu'$ in the sense of Definition \ref{defrs}. In this case, $(\mr{FL}_{E/E_0,(V',J_0'),u',g'})$ is equivalent to (FL$(g',u')$). So we have to prove that $(g',u')$ is regular semi-simple if and only if $(g,j)$ is and that in this case, (FL$(g',u')$) is equivalent to (FL$(g,j)$).

Due to our special choice of $a$, there is a decomposition $\mcO_E[g'] = \mcO_E[g]\times \mcO_E$. Its action on  $V' = V\oplus E\tilde u$ is then a factor-wise action. This already proves the claim about the regular semi-simpleness.
We leave it to the reader to check that there is a bijection
$$M(g,j) \iso M(g',u'),\ \ Λ \mapsto Λ\oplus \mcO_E\tilde u.$$
Then $\mr{len}(Λ/L(g,j)) = \mr{len}((Λ\oplus \mcO_E\tilde u)/L(g',u'))$ and hence $O(g',u') = O(g,j)$. Similarly, one gets that $I(g',u') = I(g,j)$.
\end{proof}

\subsection{The Arithmetic Fundamental Lemma}
\label{ssAflgen}
We now assume that $J = J_1$ is the odd hermitian form. In particular, $x \mapsto x^*$, the dual lattice $Λ\mapsto Λ^\vee$ and $M$ are now defined with respect to this form.

\begin{lem}\label{lemFLvanishing}
For all regular semi-simple and adjoint-stable pairs $(x,j)$, there is an equality
$$O(x,j) = 0.$$
\end{lem}
\begin{proof}
The argument is taken from \cite[Corollary 7.3]{RTZ}. Let us assume that $L(x,j)\subset L(x,j)^\vee$ since otherwise $M(x,j) = \emptyset$ and hence $O(x,j) = 0$. Let us also define $l:= [L(x,j)^\vee:L(x,j)]$ which is odd since $J_1$ is the odd form. Then $Λ\mapsto Λ^\vee$ induces an involution on the set $M(x,j)$ which is fixed point free since it interchanges $M(x,j)_i$ and $M(x,j)_{l-i}$. Thus $|M(x,j)_i| = |M(x,j)_{l-i}|$ and the two summands $(-1)^i|M(x,j)_i|$ and $(-1)^{l-i}|M(x,j)_{l-i}|$ in the definition of $O(x,j)$ cancel.
\end{proof}

\begin{defn}
The pair $(x,j)$ is called \emph{artinian}, if the schematic intersection $\mcZ(x)\cap \mcZ(j)\subset \mcN_{E_0,(1,n-1)}$ is an artinian scheme. In this case, we define
$$
\mr{Int}(x,j) := \mr{len}_{\mcO_{\breve E}}(\mcO_{\mcZ(x)\cap \mcZ(j)}).$$
\end{defn}

\begin{conj}[Arithmetic Fundamental Lemma]\label{conjAFLgeneral}
Let $(x,j)$ be a regular semi-simple, adjoint-stable and artinian pair. Then
\begin{equation}\label{eqAFLgeneral}
\partial O(x,j) = -\mr{Int}(x,j).\tag{AFL$(x,j)$}
\end{equation}
\end{conj}

Again, we explain the relation of this uniform version with the AFL from Section 7. This will be very similar to the explanations for the case $J = J_0$.

\begin{lem}\label{lemAFLLie}
Let $(x,j)$ be a adjoint-stable pair with $x\in \mfu(J_1)$. We set $V' := V\oplus Eu'$ with form $J_1':=J_1\oplus 1$ and we define
$$x' := \left(\begin{smallmatrix} x & j\\ -j^* & \end{smallmatrix}\right) \in \mfu(J_1')^0.$$
Then $(x,j)$ is regular semi-simple and artinian if and only if $x'$ is regular semi-simple with respect to $V' = V \oplus Eu'$ in the sense of Definition \ref{defrs} and artinian with respect to $u'$ in the sense of Definition \ref{defInter}. In the regular semi-simple and artinian case, $(\mr{AFL}(x,j))$ is equivalent to $(\mathfrak{afl}_{E/E_0,(V',J_1'),u',x'})$.
\end{lem}
\begin{proof}
This follows from Lemma \ref{lemrs}, Remark \ref{rmkModx} and Corollary \ref{lemorbintsLie}.
\end{proof}

\begin{lem}\label{lemAFLvj1}
Let $(g,j)$ be a adjoint-stable pair with $g\in U(J_1)$ and $J_1(j,j)\in \mcO_{E_0}^\times$. Then $(g,j)$ is regular semi-simple and artinian if and only if $g$ is regular semi-simple with respect to $V = j^\perp \oplus Ej$ in the sense of Definition \ref{defrs} and artinian with respect to $j$ in the sense of Definition \ref{defInter}. In the regular semi-simple and artinian case, $(\mr{AFL}(g,j))$ is equivalent to $(\mr{AFL}_{E/E_0,(V,J_1),j,g})$.
\end{lem}
\begin{proof}
This follows from Lemma \ref{lemrs}, Remark \ref{rmkModg} and Corollary \ref{lemorbints}.
\end{proof}

\begin{lem}
Let $(g,j)$ be a regular semi-simple, adjoint-stable and artinian pair such that $J_1(j,j)\notin \mcO_{E_0}$. Then both sides of \eqref{eqAFLgeneral} vanish.
\end{lem}
\begin{proof}
If $J_1(j,j)\notin \mcO_{E_0}$, then $\mcZ(j)= \emptyset$. Namely if $X\in \mcN_{E_0,(1,n-1)}$ is such that $j:\mbY_{E_0}→\mbX_{E_0,(1,n-1)}$ lifts to a homomorphism $\mcY_{E_0}→X$, then also the composition $j^* j\in \End_E^0(\mbY_{E_0})$ lifts to $\mcY_{E_0}$. But $\End(\mbY_{E_0}) = \mcO_E$ and $j^* j = J_1(j,j)$.

Also, if $J_1(j,j)\notin \mcO_{E_0}$, then $L(g,j)\not\subset L(g,j)^\vee$ and thus $M(g,j) = \emptyset$.
\end{proof}

\begin{construction}\label{constructionAFL}
Let $(g,j)$ be a adjoint-stable pair with $g\in U(J_1)$ and $v(J_1(j,j))\geq 1$. We also assume that $q_{E_0}+1>n$. Then we define
$V' := V \oplus E\tilde u$ and $u' := \tilde u + j$. We extend $J_1$ to $J_1'$ on $V'$ by setting $\tilde u \perp V$ and $J'_1(u',u') = 1$. We define $W' := (u')^\perp$, which is an odd hermitian space.

As explained in Section \ref{subsect:JRFL}, there exists $a\in E^1$ such that $P(a) \not \equiv 0$ mod $π_E$ where $P$ is the characteristic polynomial of $g$. We define
$$g' := \left(\begin{smallmatrix} g & \\ & a\end{smallmatrix}\right)\in U(J'_1)$$
where the block matrix decomposition is with respect to $V' = W \oplus E\tilde u$.
\end{construction}

\begin{lem}\label{lemAFLgen}
Let $(g,j)$ be a adjoint-stable pair with $g\in U(J_1)$ and $v(J_1(j,j))\geq 1$. We also assume that $q_{E_0}+1 >n$. Let $V',u'$ and $g'$ be as above. Then $(g,j)$ is regular semi-simple and artinian if and only if $g'$ is regular semi-simple with respect to $V' = W'\oplus Eu'$ in the sense of Definition \ref{defrs} and artinian with respect to $u'$ in the sense of Definition \ref{defInter}. In the regular semi-simple and artinian case, the identity $(\mr{AFL}(g,j))$ is equivalent to $(\mr{AFL}_{E/E_0,(V',J_1'),u',g'})$.
\end{lem}
\begin{proof}
By Lemma \ref{lemAFLvj1}, $(g',u')$ is regular semi-simple and artinian if and only if $g'$ is regular semi-simple with respect to $V' = W'\oplus Eu'$ and artinian with respect to $u'$. In this case, $(\mr{AFL}_{E/E_0,(V',J_1'),u',g'})$ is equivalent to $(\mr{AFL}(g',u'))$. So we have to prove that $(g',u')$ is regular semi-simple and artinian if and only if $(g,j)$ is, and that in this case the two identities $(\mr{AFL}(g',u'))$ and $(\mr{AFL}(g,j))$ are equivalent.

Due to our special choice of $a$, there is a decomposition $\mcO_E[g'] = \mcO_E[g]\times \mcO_E$. Its action on  $V' = V\oplus E\tilde u$ is then a factor-wise action. This already proves the claim about the regular semi-simpleness.
Note that $J_1'(\tilde u, \tilde u) \in \mcO_{E_0}^\times$. We leave it to the reader to check that there is a bijection
$$M(g,j) \iso M(g',u'),\ \ Λ \mapsto Λ\oplus \mcO_E\tilde u.$$
Then $\mr{len}(Λ/L(g,j)) = \mr{len}((Λ\oplus \mcO_E\tilde u)/L(g',u'))$ and hence $\partial O(g',u') = \partial O(g,j)$.
We still have to show $\mr{Int}(g',u') = \mr{Int}(g,j)$ which follows from an identification of formal schemes, $\mcZ(g)\cap \mcZ(u) = \mcZ(g')\cap \mcZ(u')$.

Namely note that because of the decomposition $\mcO_E[g'] = \mcO_E[g] \times \mcO_E$, there is an inclusion
$$\mcZ(g')\subset \mcZ(\tilde u)$$
which identifies the cycle $\mcZ(g')\subset \mcZ(\tilde u)\iso \mcN_{E_0,(1,n-1)}$ with $\mcZ(g)$. Moreover,
$$\mcZ(\tilde u)\cap \mcZ(u') = \mcZ(\tilde u)\cap \mcZ(j)$$
since the intersection only depends on the spanned module
$$\mcO_E\tilde u + \mcO_Eu'\subset \Hom^0(\mbY_{E_0},\mbX_{E_0,(1,n-1)} \times \mbY_{E_0}).$$
Thus
$$\mcZ(g')\cap \mcZ(u') = \mcZ(g')\cap \mcZ(\tilde u) \cap \mcZ(u') \iso \mcZ(g)\cap \mcZ(j).$$
\end{proof}

\section{The AFL in presence of additional multiplication}
We now study the AFL for pairs $(x,j)$ such that there exists a maximal order $\mcO_A\subset \mcO_E[x]$ as in Section \ref{sect:cycles}. We first do this for field extensions $A/E$ which uses the main result about cycles, Theorem \ref{thmZoe}. In the subsequent section, we deal with the case of an étale algebra $A/E$. Here, the Fundamental Lemma is needed as an additional input.

\subsection{Multiplication by a field extension $A/E$}
\label{subsect:mainfield}
Let $A_0/E_0$ be a field extension of degree $d$ such that $A:= A_0\tensor_{E_0}E$ is also a field. Then $A/A_0$ is an unramified quadratic extension and we denote its Galois conjugation also by $σ$. Let
$A\hookrightarrow \End_E(V)$
be an embedding that is equivariant for the Galois conjugation $σ$ on $A$ and the adjoint involution of $J_1$ on $\End(V)$. In other words,
$$J_1(a\ ,\ ) = J_1(\ ,σ(a)\ ),\ \ a\in A.$$
Let $ϑ_A$ be a generator of the inverse different of $A_0/E_0$ and let $J_1^A:V\times V→A$ be the $A/A_0$-hermitian form characterized by the property that
$$\mr{tr}_{A/E}\circ ϑ_AJ_1^A = J_1.$$
Note that for an $\mcO_A$-lattice $Λ\subset V$, the dual lattice $Λ^\vee$ with respect to the form $J_1$ is also the dual of $Λ$ with respect to $J_1^A$. In particular, $(V,J_1^A)$ is an odd hermitian space.

The action of $A$ by quasi-endomorphisms on $\mbX_{E_0,(1,n-1)}$ makes it into a framing object $\mbX_{A_0/E_0,(1,n'-1)}$ for $\mcN_{A_0/E_0,(1,n'-1)}$, where $n=dn'$. We also define $\mbX_{A_0,(1,n'-1)} := \mcC(\mbX_{A_0/E_0,(1,n'-1)})$ where $\mcC$ is the functor from Theorem \ref{thmEquiv}.
We set $\mbY_{A_0/E_0} := \mcO_{A_0} \tensor_{\mcO_{E_0}} \mbY_{E_0}$ with respect to $ϑ_A$ as in Definition \ref{defSerre} and set $\mbY_{A_0} := \mcC(\mbY_{A_0/E_0})$. There is an isomorphism
$$\Hom^0_E(\mbY_{E_0},\mbX_{E_0,(1,n-1)}) \iso \Hom^0_A(\mbY_{A_0/E_0},\mbX_{A_0/E_0,(1,n'-1)}),\ \ j\mapsto \mr{id}_{A_0}\tensor j$$
which is an isometry with respect to $J_1^A$ on the left and the natural form on the right. Via $\mcC$, these hermitian spaces are also isometric to $\Hom^0_A(\mbY_{A_0},\mbX_{A_0,(1,n'-1)})$.

The point is now that any regular semi-simple, adjoint-stable and artinian pair $(x,j)\in \End_E(V)\times V$ such that $x$ is $A$-linear gives rise to two AFL identities, one for the base field $E_0$ and one for $A_0$. We denote by (AFL$(x,j)_{E_0}$) the one for $\mcN_{E_0,(1,n-1)}$ where the $A$-action does not play a role. We denote by $(\mr{AFL}(x,j)_{A_0}):= (\mr{AFL}(\mcC(x),\mcC(\mr{id}_{A_0}\tensor j)))$ the one for $\mcN_{A_0,(1,n'-1)}$. Our main result is the following theorem and its corollaries.

\begin{thm}\label{thmAFLgenmain}
Let $(x,j) \in \End_E(V)\times V$ be a regular semi-simple, adjoint-stable and artinian pair such that $\mcO_A\subset \mcO_E[x]$. Then $(x,j)$ is also regular semi-simple, adjoint-stable and artinian when viewed over $A_0$ and the two identities $(\mr{AFL}(x,j)_{E_0})$ and $(\mr{AFL}(x,j)_{A_0})$ are equivalent.
\end{thm}
\begin{proof}
Let us keep the notation $L(x,j)$, $τ(x,j)$, $M(x,j)$, $\partial O(x,j), \mr{Int}(x,j)$ etc.\! for the setting over $E_0$. We denote by $L(x,j)^A$, $τ(x,j)^A$, $M(x,j)^A$, $\partial O(x,j)^A, \mr{Int}(x,j)^A$ etc.\! the respective notions for the setting over $A_0$. It is clear that $(x,j)$ is also regular semi-simple and adjoint-stable when viewed over $A_0$, since $\mcO_A[x]j = \mcO_E[x]j$ by assumption.

\emph{Comparison of the analytic sides of $(\mr{AFL}(x,j)_{E_0})$ and $(\mr{AFL}(x,j)_{A_0})$:}\\
Any $x$-stable $\mcO_E$-lattice $Λ\subset V$ is automatically an $\mcO_A[x]$-lattice. In particular, $L(x,j) = L(x,j)^A$. Similarly, the involutions $τ(x,j)$ and $τ(x,j)^A$ agree since they only depend on the Rosati involution on $E[x] = A[x]$ and $j$. This implies that
$$M(x,j) = M(x,j)^A.$$
Let $f$ be the inertia degree of $A_0/E_0$. Then $M(x,j)_i = \emptyset$ if $f\nmid i$ and $M(x,j)_{fi} = M(x,j)_i^A$. Note that $f$ is odd since we assumed that $A_0\tensor_{E_0} E$ is a field. In particular $(-1)^i = (-1)^{fi}$ and hence
$$\partial O(x,j) = \sum_{i\in \mbZ} (-1)^ii |M(x,j)_i| = f\sum_{i\in \mbZ} (-1)^ii |M(x,j)^A_i| = f\cdot \partial O(x,j)^A.$$

\emph{Comparison of the geometric sides of $(\mr{AFL}(x,j)_{E_0})$ and $(\mr{AFL}(x,j)_{A_0})$:}\\
By Remark \ref{rmkComp2}, the cycle $\mcZ(x)\subset \mcN_{E_0,(1,n-1)}$ can be identified with $f$ copies of $\mcZ(x)^A$, where $\mcZ(x)^A := \mcZ(\mcC(x))$ is the corresponding cycle in $\mcN_{A_0,(1,n'-1)}$. By Remark \ref{rmkComp1}, this identification is compatible with the formation of KR-divisors and hence
$$\mcZ(x)\cap \mcZ(j) \iso \coprod_{i = 1}^f \mcZ(x)^A\cap \mcZ(j)^A$$
where again $\mcZ(j)^A = \mcZ(\mcC(\mr{id}_{A_0}\tensor j))$ is the respective cycle in $\mcN_{A_0,(1,n'-1)}$. It follows that $(x,j)$ is also artinian when viewed over $A_0$ and
$$\mr{Int}(x,j) = f\mr{Int}(x,j)^A.$$
The theorem follows.
\end{proof}
We now translate this back into statements about the AFL in the original formulation from Sections 6 and 7.
\begin{cor}
\label{corLieFlat}
Let $x\in \mfu(J_1)^0_{\mr{rs}}$ be regular semi-simple and artinian, of the form
$$x = \left(\begin{smallmatrix} x^\flat & j \\ -j^* & \end{smallmatrix}\right).$$
Let $A_0/E_0$ be a field extension such that $A := A_0\tensor_{E_0} E$ is again a field, together with an embedding $\mcO_A \hookrightarrow \mcO_E[x^\flat]$ that is equivariant for the Galois conjugation $σ$ on $\mcO_A$ and the adjoint involution ${}^*$ on $\mcO_E[x^\flat]$. Let $V^A:=W\oplus Au^A$ be the hermitian $A$-vector space with form $J_1^A:=J_1^{\flat,A}\oplus 1$.

(1) Then $x$ can also be viewed as an element of $\mfu(J_1^A)^0_{\mr{rs}}$ and there is an equivalence
$$(\mathfrak{afl}_{E/E_0,(V,J_1),u,x})\ \Leftrightarrow\ (\mathfrak{afl}_{A/A_0,(V^A,J_1^A),u^A,x}).$$

(2) In particular, if $\dim_A(W) \leq 2$, then $(\mathfrak{afl}_{E/E_0,(V,J_1),u,x})$, holds.
\end{cor}
\begin{proof}
Part (1) is a combination of Lemma \ref{lemAFLLie} and Theorem \ref{thmAFLgenmain}. Part (2) follows since the AFL has been proven for $n \leq 3$.
\end{proof}
The same arguments imply an analogous result for the group version.
\begin{cor}
Let $g\in U(J_1)_{\mr{rs}}$ be regular semi-simple and artinian with integral characteristic polynomial.\footnote{This ensures that $\mcO_E[g]$ is stable under the adjoint involution. Note that both sides of (AFL$_{E/E_0,(V,J_1),u,g}$) vanish if the characteristic polynomial of $g$ is not integral. So this is not a serious restriction.}
Let $A_0/E_0$ be a field extension such that $A := A_0\tensor_{E_0} E$ is again a field, together with an embedding $\mcO_A \hookrightarrow \mcO_E[g]$ that is equivariant for the Galois conjugation on $A$ and the adjoint involution on $\mcO_E[g]$. We assume that $J_1^A(u,u) \in \mcO_{A_0}^\times$, where $J_1^A$ is the lifted hermitian form.

(1) Then $g$ is also an element of $U(J_1^A)_{\mr{rs}}$ and there is an equivalence
$$(\text{AFL}_{E/E_0,(V,J_1),u,g})\ \Leftrightarrow\ (\text{AFL}_{A/A_0,(V,J_1^A),u,g}).$$

(2) In particular, if $\dim_A(V) \leq 3$, then the AFL for $g$, $(\text{AFL}_{E/E_0,(V,J_1),u,g})$, holds.\qed
\end{cor}

Let us now formulate the variant for $v(J_1^A(u,u))\geq 1$. As in Lemma \ref{lemAFLgen}, we raise the dimension by $1$ for this.

\begin{cor}
Assume that $q_{A_0}+1>n$. Let $g$, $A_0$, $A$ and the embedding $\mcO_A\hookrightarrow \mcO_E[g]$ be as in the previous corollary, but let $v(J_1^A(u,u))\geq 1$. Let $V^A := V \oplus A\tilde u$, $u^A:=u + \tilde u$ and extend $J_1^A$ to a hermitian form $J_1^{A,\sharp}$ on $V^A$ by defining $\tilde u \perp V$ and $J_1^{A,\sharp}(u^A,u^A) = 1$.

Let $P\in A[t]$ be the characteristic polynomial of $g$ as $A$-linear endomorphism of $V$ and let $a\in A^1$ be such that $P(a)\not\equiv 0$ modulo $π_A$ where $π_A$ is a uniformizer of $A$. Define $g^A\in U(J_1^{A,\sharp})$ as
$$g^A := \left(\begin{smallmatrix} g & \\ & a\end{smallmatrix}\right) \in \End(V^A).$$

(1) Then $g^A\in U(J_1^{A,\sharp})_{\mr{rs}}$ is regular semi-simple with respect to $V^A = (u^A)^\perp \oplus Au^A$, artinian with respect to $u^A$ and there is an equivalence
$$(\mr{AFL}_{E/E_0,(V,J_1),u,g})\ \Leftrightarrow\ (\mr{AFL}_{A/A_0,(V^A,J_1^{A,\sharp}),u^A,g^A}).$$

(2) In particular, if $\dim_A(V) \leq 2$, then the AFL for $g$, $(\mr{AFL}_{E/E_0,(V,J_1),u,g})$, holds.\qed
\end{cor}
\begin{proof}
Part (1) is a combination of Lemma \ref{lemAFLgen} and Theorem \ref{thmAFLgenmain}. Part (2) follows since the AFL has been proven for $n\leq 3$.
\end{proof}

\subsection{Multiplication by an étale algebra $A/E$}
\begin{thm}\label{thmMainEtale}
Let $(x,j)\in \End_E(V)\times V$ be a regular semi-simple, adjoint-stable and artinian pair. Assume that there exists a product decomposition $\mcO_E[x] = R_0\times R_1$ that is stable under ${}^*$. Let $ V = V_0\times V_1$ be the corresponding decomposition of $V$ and let $(x_0,j_0)$ resp.\! $(x_1,j_1)$ denote the components of $(x,j)$ in $V_0$ resp.\! $V_1$. We assume that $J_1\vert_{V_0}$ is even, which implies that $J_1\vert_{V_1}$ is odd.

Then the two identities $(\mr{FL}(x_0,j_0))$ and $(\mr{AFL}(x_1,j_1))$ imply the identity $(\mr{AFL}(x,j))$.
\end{thm}
\begin{proof}
\emph{Computation of the analytic side of $(\mr{AFL}(x,j))$:}
Any $x$-stable lattice $Λ\subset V$ is a product $Λ = Λ_0\times Λ_1$ where $Λ_i\subset V_i$ is an $x_i$-stable lattice. A special case is $L(x,j) = L(x_0,j_0)\times L(x_1,j_1)$. Furthermore, $τ(x,j) = τ(x_0,j_0)\times τ(x_1,j_1)$ and hence there is a bijection
$$\begin{aligned}
M(x_0,j_0)\times M(x_1,j_1) &\overset{\iso}{→} M(x,j)\\
(Λ_0,Λ_1)&\longmapsto Λ_0\times Λ_1\end{aligned}$$
which induces a bijection
$$\coprod_{k+l = m}M(x_0,j_0)_k\times M(x_1,j_1)_l \iso M(x,j)_m.$$
This implies the relation
\begin{equation}\label{eqProdOrbInt}
O(x,j;s) = O(x_0,j_0;s)\cdot O(x_1,j_1;s).
\end{equation}
Taking the derivative and using the vanishing part of the FL, Lemma \ref{lemFLvanishing}, we get
$$\partial O(x,j) = O(x_0,j_0) \cdot \partial O(x_1,j_1).$$

\emph{Computation of the geometric side of $(\mr{AFL}(x,j))$:}
Let $\mcO_E\times \mcO_E\subset R_0\times R_1$ be the $\mcO_E$-algebra generated by the non-trivial idempotent.
Using Proposition \ref{propgeometale}, we get
$$\mcZ(\mcO_E\times \mcO_E) \iso \left(\coprod_{\{Λ_0\subset V_0\ \mid\ Λ_0^* = Λ_0\}} \Spf \mcO_{\breve E}\right)\times_{\Spf \mcO_{\breve E}} \mcN_{E_0,(1,n_1-1)}$$
where $n_1 = \dim_E(V_1)$. By the remarks \ref{rmkComp1} and \ref{rmkComp2}, this description is compatible with the formation of $\mcZ(x)$ and $\mcZ(j)$ and we get
$$\mcZ(x)\cap \mcZ(j) = \left(\coprod_{\{Λ_0\ \mid\ Λ_0^* = Λ_0, x_0Λ_0\subset Λ_0, j_0\in Λ_0\}} \Spf \mcO_{\breve E}\right)\times_{\Spf \mcO_{\breve E}} \big(\mcZ(x_1)\cap \mcZ(j_1)\big).$$
This implies
$$\mr{Int}(x,j) = I(x_0,j_0)\cdot \mr{Int}(x_1,j_1)$$
which finishes the proof of the theorem.
\end{proof}

\begin{rmk}
Let us assume that the characteristic polynomial of $x$ is integral. Then an inclusion $\mcO_E\times \mcO_E\hookrightarrow \mcO_E[x]$ exists if and only if this polynomial has two different prime factors modulo $p$. Implicitly, this was already used in \cite[Section 8]{RTZ}.
\end{rmk}

We conclude this paper with three corollaries. For this, we take up the notation from Sections 6 and 7.

\begin{cor}
Let $x\in \mfu(J_1)^0_{\mr{rs}}$ be regular semi-simple and artinian, of the form
$$x = \left(\begin{smallmatrix} x^\flat & j \\ -j^* & \end{smallmatrix}\right).$$
Assume that there exists an embedding $\mcO_E\times \mcO_E\hookrightarrow \mcO_E[x^\flat]$ that is equivariant for the factor-wise Galois conjugation and the adjoint involution of $J_1^\flat$. Let $W = W_0\times W_1$ be the corresponding decomposition of $W$ and assume that $J_1\vert_{W_0}$ is even. Let $x_0^\flat,x_1^\flat, j_0$ and $j_1$ be the components of $x^\flat$ and $j$. For $i = 0,1$, form the vector space $V_i := W_i \oplus Eu_i$ where $u_i$ is some additional vector. We extend the form $J_1^\flat\vert_{W_i}$ to a form $J_{1,i}$ on $V_i$ by defining $(u_i,u_i) = 1$ and $u_i\perp W_i$.

Then the element
$$x_i = \left(\begin{smallmatrix} x_i^\flat & j_i \\ -j_i^* & \end{smallmatrix}\right)$$
lies in $\mfu(J_{1,i})^0_{\mr{rs}}$ and the identity $(\mr{AFL}_{E/E_0,(V_1,J_{1,1}),u_1,x_1})$ implies the identity $(\mr{AFL}_{E/E_0,(V,J_1),u,x})$.
\end{cor}
\begin{proof}
The Fundamental Lemma $(\mr{FL}(x^\flat_0,j_0))$ holds by Lemma \ref{lemFLgenLie}. Then $(\mr{AFL}(x^\flat_1,j_1))$ implies $(\mr{AFL}(x^\flat,j))$ by Theorem \ref{thmMainEtale}. Lemma \ref{lemAFLLie} yields the translation into the Lie algebra version of the AFL.
\end{proof}

\begin{cor}
Let $g\in U(J_1)_{\mr{rs}}$ be regular semi-simple and artinian.
Assume that there exists an embedding $\mcO_E\times \mcO_E\hookrightarrow \mcO_E[g]$ that is equivariant for the factor-wise Galois conjugation and the adjoint involution of $J_1$. Let $V = V_0\times V_1$ be the corresponding decomposition of $V$ and assume that $J_1\vert_{V_0}$ is even. Let $g_0$ and $g_1$ be the components of $g$ and let $u_0$ and $u_1$ be the components of $u$. Assume that the identity $(\mr{FL}(g_0,j_0))$ holds.

(1) If $J_1(u_1,u_1) \in \mcO_{E_0}^\times$, then the identity $(\mr{AFL}_{E/E_0,(V_1,J_1\vert_{V_1}),u_1,g_1})$, implies the AFL for $g$, $(\mr{AFL}_{E/E_0,(V,J_1),u,g})$. In particular, $(\mr{AFL}_{E/E_0,(V,J_1),u,g})$ holds if $\dim V_1\leq 3$ or if $g_1$ is minuscule in the sense of \cite{RTZ}.

(2) Assume that $q_{E_0}+1>n$ and that $v(J_1(u_1,u_1)) \geq 1$. We define $V_1':=V_1\oplus E\tilde u_1$ and $u_1':=u_1 + \tilde u_1$. We extend $J_1\vert_{V_1}$ to a hermitian form $J_1'$ on $V_1'$ by defining $\tilde u_1\perp V_1$ and $(u_1',u_1') = 1$. We choose an element $a\in E^1$ such that $P(a)\not\equiv 0$ modulo $p$, where $P$ denotes the characteristic polynomial of $g_1$ on $V_1$. We set
$$g_1':=\left(\begin{smallmatrix} g_1 & \\ & a\end{smallmatrix}\right)\in U(J_1')$$
where the block matrix decomposition is with respect to $V_1' = V \oplus E\tilde u_1$.

Then $(\mr{AFL}_{E/E_0,(V_1',J_1'),u_1',g_1'})$ implies $(\mr{AFL}_{E/E_0,(V,J_1),u,g})$. In particular, $(\mr{AFL}_{E/E_0,(V,J_1),u,g})$ holds if $\dim V_1\leq 2$.
\end{cor}
\begin{proof}
Part (1) follows from Theorem \ref{thmMainEtale} and Lemma \ref{lemAFLvj1} and the fact that the AFL has been proven in the minuscule case and in the case $n\leq 3$.

Part (2) follows from Theorem \ref{thmMainEtale} and Lemma \ref{lemAFLgen}.
\end{proof}

The case of a general finite étale $A_0/E_0$ follows with an inductive argument. We decompose $A_0$ into fields,
$$A_0 := \prod_{i\in I} A_{0,i}$$
and set $A := A_0\tensor_{E_0} E$ as well as $A_i := A_{0,i}\tensor_{E_0} E$, all with Galois conjugation $σ := \mr{id}\tensor σ$. Let $\mcO_A$ be the ring of integral elements in $A$.

Any embedding $A\hookrightarrow \End(V)$ that is equivariant for the Galois conjugation of $A$ and the adjoint involution of $J_1$ on $V$ induces an orthogonal decomposition $V = \prod_{i\in I} V_i$. Just as in Definition \ref{defParityIndex}, we call an index $i$ even if there exists a self-dual $\mcO_{A_i}$-lattice in $V_i$. Otherwise, we call $i$ odd. Note that since $V$ itself is odd, there is an odd number of odd indices. Also note that if $i$ is odd, then $A_i$ is necessarily a field.

\begin{cor}
Let $g\in U(J_1)_{\mr{rs}}$ be regular semi-simple and artinian. Assume that there exists an embedding $\mcO_A\hookrightarrow \mcO_E[g]$ that is equivariant for the Galois conjugation and the adjoint involution of $J_1$. Let $V = \prod_{i\in I}V_i$ be the corresponding decomposition of $V$ and let $(g_i)_{i\in I}$ and $(u_i)_{i\in I}$ be the components of $g$ and $u$.

(1) If there is more than one odd index, then both sides of $(\mr{AFL}_{E/E_0,(V,J_1),u,g})$ vanish.

(2) Otherwise, let $i_0 \in I$ be the unique odd index and let us assume that $(\mr{FL}(g_i,u_i))$ holds for $i \neq i_0$. Let us take up the notation from Theorem \ref{thmAFLgenmain} for the factor $V_{i_0}$. Then
$$(\mr{AFL}(g_{i_0},u_{i_0}))_{A_{i_0}}\ \ \Rightarrow\ \ (\mr{AFL}_{E/E_0,(V,J_1),u,g}).$$

(3) Under the assumption $J_1^{A_{i_0}}(u_{i_0},u_{i_0}) \in \mcO_{A_{0,i_0}}^\times$, we get
$$(\mr{AFL}_{A_{i_0}/A_{0,i_0},(V_{i_0},J_1\vert_{V_{i_0}}),u_{i_0},g_{i_0}})\ \ \Rightarrow\ \ (\mr{AFL}_{E/E_0,(V,J_1),u,g}).$$
\end{cor}
\begin{proof}
First note that $(\mr{AFL}_{E/E_0,(V,J_1),u,g})$ is equivalent to $(\mr{AFL}(g,u))$ by Lemma \ref{lemAFLvj1} and we work with this simpler version.
We first prove Part (1). For the geometric side, note that $\mcZ(g)\subset \mcZ(\mcO_A)$. So if there is more than one odd index, then $\mcZ(g) \subset \mcN_{E_0,(1,n-1)}$ is empty by Lemma \ref{lemoddindex}.

On the analytic side, we use the idempotents $\prod_{i\in I}\mcO_E \subset \mcO_A\subset \mcO_E[g]$ to get a product decomposition just as in formula \eqref{eqProdOrbInt},
$$O(g,u;s) = \prod_{i\in I} O(g_i,u_i;s).$$
Taking the derivative and using the vanishing part of the Fundamental Lemma, Lemma \ref{lemFLvanishing}, we get that $\partial O(g,u) = 0$ if there is more than one odd index.

Part (2) follows from Theorem \ref{thmMainEtale} by an induction argument and from Theorem \ref{thmAFLgenmain}.

Part (3) is then an application of Lemma \ref{lemAFLvj1}.
\end{proof}

\part{Appendix on strict formal $\mcO$-modules}
\section{Introduction}
Let $\mcO$ be the ring of integers in a $p$-adic local field with uniformizer $π$. Let $R$ be a $π$-adic $\mcO$-algebra. In \cite{Ahs}, Ahsendorf constructs an equivalence of categories
\begin{equation}\label{eqEquiv}
\left\{\text{strict formal }\mcO\text{-modules over }R\right\}\iso \left\{\text{nilpotent }\mcO\text{-displays over }R\right\}.
\end{equation}
We refer to \cite{ACZ} for more information. By Lau \cite{Lau}, there is a good notion of duality on the right hand side. This defines good notions of duality and polarization on the left hand side.
By definition, there is also an equivalence
\begin{equation}\label{eqEquiv2}
\{\text{strict formal $\mcO$-modules over $R$}\}\iso \{\text{nilpotent displays over $R$ with strict $\mcO$-action}\}.
\end{equation}
This equivalence has the advantage that one can forget the $\mcO$-action on both sides to read off the underlying $p$-divisible group and its display. This is not possible in \eqref{eqEquiv}. The aim of this appendix is to identify the correct notion of duality on the right hand side of \eqref{eqEquiv2}.

More precisely, our results are the following. For each finite and totally ramified extension $\mcO\subset \mcO'$ of rings of integers in $p$-adic local fields, we define the \emph{Lubin-Tate $\mcO'$-frame} $\mcL_{\mcO'/\mcO}(R)$ and prove the equivalence
$$\{\text{strict formal $\mcO'$-modules over $R$}\}\iso \{\text{nilpotent $\mcL_{\mcO'/\mcO}(R)$-windows}\}.$$
Depending on the existence of certain units, this equivalence is compatible with duality, see Lemma \ref{lemDual}. To prove this compatibility, we reinterpret the construction of Ahsendorf in \cite[Definition 2.24]{ACZ} as a base change along a morphism of frames
$$\mcL_{\mcO'/\mcO}(R)→\mcL_{\mcO'/\mcO'}(R).$$

\section{Strict $\mcO$-modules and duality}
\label{sect:strmod}
\subsection{Windows and Duality}
We work with the definitions of $\mcO$-frames and $\mcO$-windows from \cite[Section 3]{ACZ}, but keep the terminology of Lau \cite{Lau} concerning strict and not necessarily strict morphisms of frames. We now recall the definition of the dual $\mcO$-window.

Let $\mcA = (S,I,R,σ,\dot{σ})$ be an $\mcO$-frame and let $\mcP=(P,Q,F,\dot{F})$ be an $\mcA$-window. Choose a normal decomposition $P = L\oplus T,\ Q = L\oplus IT$ and consider the linearization
\begin{equation}\label{eqlin}
\mathbf{F}:=(\dot{F}\oplus F)^\sharp:S\tensor_{σ,S} (L\oplus T)→P.
\end{equation}
Let $P^\vee := \Hom_S(P,S)$ and $Q^\vee := \{ϕ\in P^\vee\mid ϕ(Q)\subset I\}$. We define the $\mcA$-window
$$\mcP^\vee=(P^\vee, Q^\vee, F^\vee, \dot{F}^\vee)$$
through the operator $(\mathbf{F}^\vee)^{-1}$ and the normal decomposition $P^\vee = L^\vee \oplus T^\vee, Q^\vee = IL^\vee \oplus T^\vee$, see \cite[Lemma 3.6]{ACZ}.

\begin{defn}\label{defDual}
The $\mcA$-window $\mcP^\vee$ is the \emph{dual $\mcA$-window} of $\mcP$.
\end{defn}

It is clear that dualizing is an anti-equivalence of the category of $\mcA$-windows and that there is a canonical identification $(\mcP^\vee)^\vee \iso \mcP$ coming from the canonical identification $(P^\vee)^\vee \iso P$.

Let
$$α:\mcA→\mcA':=(S',I',R',σ',\dot{σ}')$$
be a $u$-morphism of frames for some unit $u\in S'$, i.e.\! $uα\circ \dot{σ} = \dot{σ}'\circ α$. If $u=1$, i.e.\! if $α$ is a strict morphism, and if $\mcP = (P,Q,F,\dot F)$ is an $\mcA$-window, then
\begin{equation}\label{eqDualStrict}
α_*(\mcP^\vee) \iso (α_*\mcP^\vee),
\end{equation}
up to the identification $P^\vee\tensor_S S' \iso (P\tensor_S S')^\vee$. To treat the case of general $u$, first recall that the base change along the $u$-morphism
$$(S,I,R,σ,\dot{σ})→(S,I,R,σ,u\dot{σ})$$
is given by 
$$(P,Q,F,\dot{F})\mapsto (P,Q,uF,\dot{F}).$$

\begin{lem}\label{lemDual}
Let $α:\mcA=(S,I,R,σ,\dot{σ})→(S',I',R',σ',\dot{σ}')$ be a $u$-isomorphism and let $ε\in S^\times$ be a unit such that $σ(ε)ε^{-1} = u$. Let $\mcP = (P,Q,F,\dot{F})$ be an $\mcA$-window. Then multiplication by $ε$ defines an isomorphism
$$α_*(\mcP^\vee) \iso (α_* \mcP)^\vee.$$
\end{lem}
\begin{proof}
Choose a normal decomposition $P = L\oplus T,\ Q = L\oplus IT$ and consider the linearization $\mathbf{F}$ as in \eqref{eqlin}. Then the window $α_*(\mcP^\vee)$ (resp.\! $(α_*\mcP)^\vee$) corresponds to the normal decomposition $P^\vee = L^\vee \oplus T^\vee,\ Q^\vee = IL^\vee \oplus T^\vee$ and the operator
$$\left( \begin{smallmatrix} 1 & \\ & u^{-1}\end{smallmatrix} \right) α(\mathbf{F}^\vee)^{-1}\ \ \left(\text{resp.\!}\ \left( \begin{smallmatrix} u & \\ & 1\end{smallmatrix}\right) α(\mathbf{F}^\vee)^{-1}\right).$$
It is clear that multiplication by $ε$ defines an isomorphism.
\end{proof}

\begin{defn}\label{defBiHom}
Let $\mcP_i = (P_i,Q_i,F_i,\dot F_i)$, for $i = 1,2,3$, be three $\mcA$-windows. We define
$$\mr{BiHom}(\mcP_1\times \mcP_2, \mcP_3)$$ to be the set of $S$-bilinear forms $(\ ,\ ):P_1\times P_2→P_3$ such that $(Q_1,Q_2)\subset Q_3$ and such that
$$(\dot F_1 q_1,\dot F_2 q_2) = \dot F_3(q_1,q_2),\ \ \ q_1\in Q_1, q_2\in Q_2.$$
\end{defn}

Note that $\mcA$ (or rather just the quadruple $(S,I,σ,\dot{σ})$) is an $\mcA$-window over itself.

\begin{lem}
Let $\mcP = (P,Q,F,\dot F)$ be an $\mcA$-window. Then the canonical pairing $\langle\ ,\ \rangle:P\times P^\vee→S$ defines an element in $\mr{BiHom}(\mcP\times \mcP^\vee, \mcA)$.
\end{lem}
\begin{proof}
This can be checked after choosing a normal decomposition $P = L\oplus T$, $Q = L\oplus IT$. Let $\mathbf{F}$ be the linearization of $\dot F \oplus F$ as in \eqref{eqlin}. The relation $\langle Q, Q^\vee\rangle \subset I$ holds by definition of $Q^\vee$. Now for example if $q\in L$ and $ξq^\vee\in IL^\vee$, then
\begin{equation}
\begin{aligned}
\langle \dot F (q), \dot F^\vee(ξq^\vee)\rangle & = \langle \mathbf{F}(q) , \dot{σ}(ξ)(\mathbf{F}^{-1})^\vee(q^\vee)\rangle\\
& = \dot {σ}(ξ)σ(\langle q, q^\vee\rangle)\\
& = \dot {σ}(\langle q, ξq^\vee\rangle).
\end{aligned}
\end{equation}
The other cases for $q$ and $q^\vee$ are checked analogously.
\end{proof}

\begin{prop}\label{propDualForm}
Let $\mcP_i = (P_i,Q_i,F_i,\dot F_i)$, for $i = 1,2$, be two $\mcA$-windows. Then pullback of the canonical pairing defines an isomorphism
$$\Hom(\mcP_1,\mcP_2^\vee) \iso \mr{BiHom}(\mcP_1\times \mcP_2,\mcA).$$
This isomorphism is functorial in both $\mcP_1$ and $\mcP_2$ and compatible with base change along morphisms of $\mcO$-frames $\mcA→\mcA'$.
\end{prop}
\begin{proof}
The canonical map is injective since it is induced from the analogous isomorphism on underlying $S$-modules,
$$\Hom_S(P_1,P_2^\vee) \overset{\iso}{→} \mr{BiHom}_S(P_1\times P_2,S).$$
To prove surjectivity, we consider a homomorphism $f:P_1→P^\vee_2$ such that the induced bilinear form $(\ ,\ ):P_1\times P_2→S$ lies in $\mr{BiHom}(\mcP_1\times \mcP_2,\mcA)$. We claim that $f$ is a homomorphism of $\mcA$-windows.

The relation $f(Q_1)\subset Q_2^\vee$ follows immediately from the relation $(Q_1,Q_2)\subset I$. We still have to show $f(\dot F_1 q_1) = \dot F_2^\vee(f(q_1))$ for all $q_1\in Q_1$. For this we compute for all $q_2\in Q_2$,
\begin{equation}
\begin{aligned}
\langle f(\dot F_1 q_1), \dot F_2 q_2\rangle & = (\dot F_1 q_1, \dot F_2 q_2) \\
  & = \dot{σ}(q_1,q_2)\\ 
  & = \dot{σ} \langle f(q_1), q_2\rangle\\
  & = \langle \dot F^\vee_2 f(q_1), \dot F_2 f(q_2)\rangle.
\end{aligned}
\end{equation}
Now $\dot F_2:Q_2→P_2$ is a $σ$-linear epimorphism and $\langle\ ,\ \rangle$ is $S$-bilinear. This implies
$$\langle f(\dot F_1 q_1), p_2\rangle = \langle \dot F_2^\vee (f(q_1)), p_2\rangle,\ \ p_2\in P_2,$$
which proves $\dot F_2^\vee f(q_1) = f( \dot F_1 q_1)$.
\end{proof}

In terms of the pairings, the isomorphism from Lemma \ref{lemDual} corresponds to scaling the form $α_*(\ ,\ )$ by $ε^{-1}$. 

\begin{defn}\label{defPol}
Let $\mcP=(P,Q,F,\dot F)$ be an $\mcA$-window. A \emph{principal polarization} is an isomorphism $λ:\mcP→\mcP^\vee$ such that $λ^\vee:\mcP = (\mcP^\vee)^\vee→\mcP^\vee$ equals $-λ$. Equivalently, a polarization is an alternating perfect pairing $λ(\ ,\ )\in \mr{BiHom}(\mcP\times \mcP,\mcA).$
\end{defn}

\subsection{Strict $\mcO$-modules and duality}
\label{subsect:strOKMod}
Let us fix a uniformizer $π\in \mcO$. We refer to \cite[Section 1.2]{FF} for the definition and properties of the relative Witt vectors. 

\begin{defn}\label{defWitt}
For any $\mcO$-algebra $R$, we define the \emph{Witt $\mcO$-frame}\footnote{There is no need to write the Verschiebung as a superscript in the appendix.}
$$\mcW_{\mcO}(R) = (W_{\mcO}(R), I_{\mcO}(R), {}^F,V)$$
over $R$ as follows. The ring $W_{\mcO}(R)$ is the ring of relative $\mcO$-Witt vectors of $R$ with respect to $π$. The ideal $I_{\mcO}(R)$ is the augmentation ideal
$$I_{\mcO}(R) := \ker(W_{\mcO}(R)→R)$$
and ${}^F$ (resp.\! $V$) denotes the Frobenius (resp.\! the Verschiebung with respect to $π$). Windows over $\mcW_{\mcO}(R)$ are also called \emph{$\mcO$-displays over $R$}, see \cite{ACZ}.
\end{defn}

\begin{defn}
Let $R$ be an $\mcO$-algebra. A \emph{strict $\mcO$-module} over $S = \Spec R$ is a pair $(X,ι)$ where $X/S$ is a $p$-divisible group and $ι:\mcO→\End(X)$ an action such that $\mcO$ acts on $\Lie(X)$ via the structure morphism $\mcO→\mcO_S$. A strict $\mcO$-module is called \emph{formal}, if the underlying $p$-divisible group is formal.
\end{defn}

Recall that, by Zink \cite{Zink} and \cite{Lauequiv} (in the absolute case $\mcO = \mbZ_p$) and the extension by Ahsendorf \cite{ACZ} (in the general case), there is an equivalence of categories
$$\{\text{strict formal $\mcO$-modules$/S$}\} \iso \{\text{nilpotent $\mcO$-displays$/S$}\}$$
whenever $π$ is nilpotent in $R$.

\begin{defn}\label{defDualModule}
Let $X = (X,ι)$ be a strict formal $\mcO$-module over $S$ with associated $\mcO$-display $\mcP$.\\
i) $X$ is called \emph{biformal} if the dual $\mcO$-display $\mcP^\vee$ is also nilpotent.
\footnote{Being biformal is equivalent to the slopes $0$ and $1$ not occurring in the slope filtration of (the relative $\mcO$-isocrystal of) $X$ at every geometric point of $S$.}\\
ii) The dual of a biformal strict $\mcO$-module $X$ is the strict $\mcO$-module associated to the dual of its $\mcO$-display $\mcP^\vee$.\\
iii) A \emph{polarization} (resp.\! \emph{principal polarization}) of the biformal strict $\mcO$-module $X$ is an isogeny (resp.\! an isomorphism) $λ:X→X^\vee$ such that $λ^\vee = -λ$.
\end{defn}

\begin{rmk}
The restriction to biformal strict $\mcO$-modules is necessary since we only work with $\mcO$-displays instead of Dieudonné $\mcO$-displays. See \cite[Section 4]{ACZ} for the definition of the dual group in the general case.
\end{rmk}

\begin{rmk}\label{rmkDualModule}
Note that the definition of the Verschiebung $V$ on $W_{\mcO}(R)$ and hence the definition of the dual $\mcO$-display (resp.\! the dual strict $\mcO$-module) depends on the choice of the uniformizer $π$.
\end{rmk}

Recall the following results from \cite[Section 3]{ACZ}. To any strict formal $\mcO$-module, there is associated a crystal $\mbD_X$ on the category of $\mcO$-pd-thickenings. We denote by $\mbD_X(S')$ its value at an $\mcO$-pd-thickening $S→S'$. As in the case of $p$-divisible groups, there is a Hodge filtration $\mcF\subset \mbD_X(S)$ and deformations of $X$ along $\mcO$-pd-thickenings are in bijection with liftings of the Hodge filtration.

Now assume that $X$ is biformal. It follows from the definitions that there is a perfect pairing
$$\mbD_X(S')\times \mbD_{X^\vee}(S')→\mcO_{S'}.$$
Furthermore, the Hodge filtration $\mcF\subset \mbD_X(S)$ is the orthogonal complement of the Hodge filtration $\mcF^\vee\subset \mbD_{X^\vee}(S)$ of the dual $\mcO$-module. In particular if $λ:X→X^\vee$ is a principal polarization, then the induced bilinear form on $\mbD_X$ is alternating and the Hodge filtration $\mcF\subset \mbD_X(S)$ is a Lagrangian subspace. Deformations of $(X,λ)$ along an $\mcO$-pd-thickening are then in bijection with liftings of the Hodge filtration as a Lagrangian subspace.

\section{The totally ramified case}\label{sectLT}
\subsection{Lubin-Tate frames}
Let $\mcO'/\mcO$ be a finite, integrally closed and totally ramified extension of degree $e$ and choose a uniformizer $π'\in \mcO'$. For any $\mcO'$-algebra $R$, we consider the ring $\mcO'\tensor_{\mcO}W_{\mcO}(R)$. We denote the $\mcO'$-linear extension of the Frobenius by $σ:=\mr{id}_{\mcO'}\tensor {}^F$. We also define
$$J_{\mcO'}(R):=\ker(\mcO'\tensor_{\mcO}W_{\mcO}(R)→R).$$
Our aim now is to define a $σ$-linear epimorphism $\dot{σ}:J_{\mcO'}(R)→\mcO'\tensor_{\mcO} W_{\mcO}(R)$ that makes
$$(\mcO'\tensor_{\mcO}W_{\mcO}(R),J_{\mcO}(R),R,σ,\dot{σ})$$
into an $\mcO'$-frame such that strict formal $\mcO'$-modules over $R$ are equivalent to windows over that frame. 

\begin{defn}
Let $R$ be a $π$-adic $\mcO'$-algebra. A \emph{Lubin-Tate $\mcO$-display over $R$ (for the extension $\mcO'$)} is an $\mcO$-display $(P,Q,F,\dot F)$ over $R$ equipped with a strict $\mcO'$-action such that $P$ is free of rank $1$ over $\mcO'\tensor_{\mcO} W_{\mcO}(R).$
\end{defn}

The strictness implies that $Q = J_{\mcO'}(R)P$. We will usually choose a generator of $P$ and hence consider $\mcO$-displays of the form
$$(\mcO'\tensor_{\mcO} W_{\mcO}(R), J_{\mcO'}(R), F, \dot F).$$
Here, $\mcO'$ acts naturally on $\mcO'\tensor_{\mcO} W_{\mcO}(R)$ and both $F$ and $\dot F$ are $σ$-linear.

\begin{rmk}\label{rmkLTx}
(1) The definition could be extended to $P$ being only locally free of rank $1$ over $\mcO'\tensor_{\mcO} W_{\mcO}(R)$. But we will not need this.

(2) Let $u\in \mcO'\tensor_{\mcO} W_{\mcO}(R)$ be a unit and let $(P,Q,F,\dot F)$ be a Lubin-Tate $\mcO$-display over $R$. Then also $(P,Q,uF,u\dot F)$ is a Lubin-Tate $\mcO$-display.

(3) Let $\dot F:J_{\mcO'}(R)→\mcO'\tensor_{\mcO} W_{\mcO}(R)$ be any $σ$-linear epimorphism. Then there is at most one way to define a $σ$-linear endomorphism $F$ of $\mcO'\tensor_{\mcO} W_{\mcO}(R)$ which satisfies the identity
\begin{equation}\label{condLT}
\dot F(ξx) = V^{-1}(ξ)F(x),\ \ ξ\in \mcO'\tensor_{\mcO} I_{\mcO}(R),\ x\in \mcO'\tensor_{\mcO} W_{\mcO}(R),
\end{equation}
where $V$ denotes the $\mcO'$-linear extension of the $π$-Verschiebung to $\mcO'\tensor_{\mcO} I_{\mcO}(R)$. It is given by 
$$F(x) = \dot F(V(1)x)$$
and it is now a condition that the so-defined $F$ satisfies the relation \eqref{condLT} for all $ξ$. It is enough to check this for $x = 1$ in which case the condition becomes
\begin{equation}\label{eqLT}
\dot F(ξ) = V^{-1}(ξ)F(1) = V^{-1}(ξ)\dot F(V(1)),\ \ ξ\in \mcO'\tensor_{\mcO} I_{\mcO}(R).
\end{equation}
\end{rmk}

\begin{prop}\label{propLTDisp}
Let $R$ be any $π$-adic $\mcO'$-algebra.

(1) For any Lubin-Tate $\mcO$-display
$$(\mcO'\tensor_{\mcO} W_{\mcO}(R),J_{\mcO'}(R), F, \dot{F}),$$
the element $κ:=\dot{F}(π'\tensor 1 - 1\tensor[π'])$ is a unit.

(2) For every unit $κ\in \mcO'\tensor_{\mcO} W_{\mcO}(R)$, there exists a unique Lubin-Tate $\mcO$-display
$$(\mcO'\tensor_{\mcO} W_{\mcO}(R),J_{\mcO'}(R), F, \dot{F})$$
such that $\dot{F}(π'\tensor 1 - 1\tensor[π']) = κ$.
\end{prop}

Zink \cite[Proposition 26]{Zink} proves part (2) in the case $\mcO = \mbZ_p$ and $π$-torsion free $R$. The proof carries over to the case of general $\mcO$. Applying this result with $R = \mcO'$ and using base change, we get the existence of Lubin-Tate $\mcO$-displays for all $π$-adic $\mcO'$-algebras $R$. Applying (2) of Remark \ref{rmkLTx}, we get the existence for all units $κ$. So we are left with proving (1) and the uniqueness assertion from (2).

\emph{Proof of Proposition \ref{propLTDisp}, part (1).}
Let us show that $κ$ is a unit. For this recall the following lemma from \cite{Zink}. It follows from the fact that $W_{\mcO}(R)$ is $I_{\mcO}(R)$-adically complete.

\begin{lem}\label{lemUnit}
Let $R$ be a $π$-adic $\mcO'$-algebra. Then an element $u\in \mcO'\tensor_{\mcO}W_{\mcO}(R)$ is a unit if and only if its image in
$$(\mcO'/π')\tensor_{\mcO'} (R/π')$$
is.\qed
\end{lem}

In particular, we can check that $κ$ is a unit at geometric points $R/π'→k$. But then
$$\mcO'\tensor_{\mcO}W_{\mcO}(k) \iso W_{\mcO'}(k)$$
is a complete DVR with uniformizer $π'$ and residue field $k$. The element $π'\tensor 1 - 1\tensor [π']$ maps to a generator of $J_{\mcO'}(k) = π'W_{\mcO'}(k)$. In particular, it is sent to a unit by (the base change to $k$) of $\dot F$. This finishes the proof of part (1).
\qed

\begin{lem}\label{lemMagic}
There exists an element $θ\in \mcO'\tensor_{\mcO}W_{\mcO}(\mcO')$ with the following two properties.
\begin{itemize}
\item[(i)] $θJ_{\mcO'}(\mcO')\subset \mcO'\tensor_{\mcO} I_{\mcO}(\mcO')$.
\item[(ii)] The image of $θ$ under $\mcO'\tensor_{\mcO} W_{\mcO}(\mcO')→\mcO'\tensor_{\mcO} W_{\mcO}(\mcO'/π') \iso \mcO'$ has valuation $e-1$. 
\end{itemize}
\end{lem}
\begin{proof}
First note that for any $π$-adic $\mcO'$-algebra $R$, the ring $\mcO'\tensor_{\mcO}W_{\mcO}(R)$ has the $W_{\mcO}(R)$-basis
$$1\tensor 1, (π')^i\tensor 1 - 1\tensor[π']^i,\ \ i = 1,\ldots,e-1.$$
In particular,
$$J_{\mcO'}(R) = \mcO'\tensor_{\mcO}I_{\mcO}(R) + (π'\tensor 1 - 1\tensor[π'])\mcO'\tensor_{\mcO} W_{\mcO}(R).$$
Thus the first condition is equivalent to $θ(π'\tensor 1 - 1\tensor [π'])\in \mcO'\tensor_{\mcO} I_{\mcO}(\mcO').$ If $\overbar{θ}$ denotes the image of $θ$ in $\mcO'\tensor_{\mcO} \mcO'$, then this is equivalent to
$$\ob{θ}(π'\tensor 1 - 1\tensor π') = 0.$$

Informally, we define $\ob{θ}$ as the fraction
$$\frac{Nπ'\tensor 1 - 1 \tensor Nπ'}{π'\tensor 1 - 1\tensor π'}\in \mcO'\tensor_{\mcO}\mcO'$$
where $Nπ'$ denotes the norm of $π'$ with respect to the ring extension $\mcO'/\mcO$. This does not make sense as stated since the numerator vanishes and the denominator is a zero divisor. The precise definition is as follows. Let
$$(π')^e+a_{e-1}(π')^{e-1}+\ldots+a_1π'+(-1)^eNπ' = 0$$
be the Eisenstein equation of $π'$. Then we set
$$(-1)^e\overbar{θ}:=-\frac{(π')^e\tensor 1 - 1\tensor (π')^e}{π'\tensor 1 - 1\tensor π'}-\sum_{i = 1}^{e-1} a_i \frac{(π')^i\tensor 1 - 1 \tensor (π')^i}{π'\tensor 1 - 1\tensor π'}$$
where each summand is understood as a geometric series. Let $θ\in \mcO'\tensor_{\mcO}W_{\mcO}(\mcO')$ be any lift of $\overbar{θ}$. Then $θ$ satisfies (i) by construction and we are left with verifying (ii).

Consider the quotient
$$β:\mcO'\tensor_{\mcO} W_{\mcO}(\mcO')→\mcO'\tensor_{\mcO} W_{\mcO}(\mcO'/π')→\mcO'\tensor_{\mcO} (W_{\mcO}(\mcO'/π')/π).$$
Then $θ$ satisfies (ii) if and only if $β(θ)\neq 0$ and $π'β(θ) = 0$. Now note that $β(\mcO'\tensor_{\mcO} I_{\mcO}(\mcO')) = 0$ and hence $β$ factors through $\mcO'\tensor_{\mcO} \mcO'$. It is easy to see that the image of $\ob{θ}$ in $\mcO'\tensor_{\mcO} (W_{\mcO}(\mcO'/π')/π)$ satisfies these two properties.
\end{proof}

\begin{lem}\label{lemthetaunit}
Let $θ$ be an element as in Lemma \ref{lemMagic} and let $V$ denote the $\mcO'$-linear extension of the Verschiebung to $\mcO'\tensor_{\mcO} W_{\mcO}(R)$. Then
$$V^{-1}(θ(π'\tensor 1 - 1\tensor [π']))$$
is a unit in $\mcO'\tensor_{\mcO} W_{\mcO}(\mcO')$.
\end{lem}
\begin{proof}
This can be checked in $\mcO'\tensor_{\mcO} W_{\mcO}(\mcO'/π')\iso \mcO'$. But here, the Verschiebung $V$ is multiplication by $π$. Using property (ii) from Lemma \ref{lemMagic}, we get the result.
\end{proof}

\begin{proof}[Proof of Proposition \ref{propLTDisp}, part (2).]
We now prove the uniqueness of a Lubin-Tate $\mcO$-display structure
$$(\mcO'\tensor_{\mcO} W_{\mcO}(R), J_{\mcO'}(R), F, \dot F)$$
with $\dot F(π'\tensor 1 - 1\tensor [π']) = κ$.

Note first that $κ$ determines $\dot F$ on $(π'\tensor 1 - 1\tensor [π'])\mcO'\tensor_{\mcO} W_{\mcO}(R)$ by $σ$-linearity. Since
$$J_{\mcO'}(R) = \mcO'\tensor_{\mcO} I_{\mcO}(R) + (π'\tensor 1 - 1\tensor [π'])\mcO'\tensor_{\mcO}W_{\mcO}(R),$$
we are left with showing that $κ$ determines $\dot F$ on $\mcO'\tensor_{\mcO} I_{\mcO}(R)$. By the relation \eqref{eqLT}, it is enough to show that $F(1)$ is determined by $κ$.

Let $θ$ be as in Lemma \ref{lemMagic} and set $a := θ(π'\tensor 1 - 1\tensor [π'])\in \mcO'\tensor_{\mcO} I_{\mcO}(R)$. Then
$$\dot F (a) = \dot F (θ(π'\tensor 1 - 1\tensor [π'])) = σ(θ) κ$$
but also
$$\dot F(a) = V^{-1}(a) F(1).$$
By Lemma \ref{lemthetaunit}, $V^{-1}(a)$ is a unit and hence $F(1)$ is determined by $κ$. 
\end{proof}

\begin{defn}\label{defLT}
Let $R$ be a $π$-adic $\mcO'$-algebra. A \emph{Lubin-Tate $\mcO'$-frame over $R$} is an $\mcO'$-frame of the form
$$(\mcO'\tensor_{\mcO} W_{\mcO}(R), J_{\mcO}(R), R, σ,\dot{σ})$$
where $\dot {σ}$ is a $σ$-linear epimorphism satisfying the relation analogous to \eqref{eqLT},
$$\dot{σ}(ξ) = V^{-1}(ξ) \dot{σ}(V(1)).$$
In other words, $\dot{σ}$ is coming from a Lubin-Tate $\mcO$-display. For a unit $κ\in \mcO'\tensor_{\mcO} W_{\mcO}(R)$, 
we denote by
$$\mcL_{\mcO'/\mcO,κ}(R)$$
the Lubin-Tate $\mcO'$-frame such that $\dot{σ}(π'\tensor 1 - 1\tensor [π']) = κ$. By Proposition \ref{propLTDisp}, such a $\dot{σ}$ exists and is unique.
\end{defn}

\begin{rmk}
By \cite[Lemma 2.2]{Lau}, there exists a unique element $s\in \mcO'\tensor_{\mcO} W_{\mcO}(R)$ such that $σ(ξ) = s\dot{σ}(ξ)$ for all $ξ\in J_{\mcO'}(R)$. For the $\mcO'$-frame $\mcL_{\mcO'/\mcO,κ}(R)$, this element is $s = κ^{-1} σ(π'\tensor 1 - 1 \tensor [π'])$.
\end{rmk}

\begin{ex}\label{exLT}
(1)
We consider the case $\mcO' = \mcO$. For any $π$-adic $\mcO$-algebra $R$, the Witt $\mcO$-frame $\mcW_{\mcO}(R)$ is an example of a Lubin-Tate $\mcO$-frame. It agrees with $\mcL_{\mcO/\mcO,ε}(R)$, where $ε\in W_{\mcO}(R)$ is the unit 
$$ε = V^{-1}(π-[π]).$$

(2) 
We return to the case of an arbitrary totally ramified extension $\mcO'/\mcO$. Let $θ$ be an element as in Lemma \ref{lemMagic}. We define $\dot{σ}:J_{\mcO'}(R)→\mcO'\tensor_{\mcO} W_{\mcO}(R)$ as
$$\dot{σ}(x) = V^{-1}(θx).$$
Then for $ξ\in \mcO'\tensor_{\mcO} I_{\mcO}(R)$,
$$\dot{σ}(ξ) = V^{-1}(θξ) = σ(θ)V^{-1}(ξ) = V^{-1}(ξ)\dot{σ}\left (V(1)\right )$$
because of the identity $θV(1) = V(σ(θ)).$ Thus $\dot{σ}$ defines the Lubin-Tate $\mcO'$-frame $\mcL_{\mcO'/\mcO,κ}(R)$ where $κ$ is the unit $V^{-1}(θ(π'\tensor 1 - 1\tensor [π']))$ from Lemma \ref{lemthetaunit}.
\end{ex}

We now consider a tower of extensions $\mcO''/\mcO'/\mcO$, all totally ramified. We fix uniformizers $π'',π'$ and $π$ in the respective rings. Recall from \cite{FF} that for any $\mcO'$-algebra $R$, there is a natural map of $\mcO$-algebras
$$α:W_{\mcO}(R)→W_{\mcO'}(R).$$
This map is Frobenius equivariant and satisfies $α\circ V_{π} = \frac{π}{π'} V_{π'} \circ α$ where $V_π$ and $V_{π'}$ denote the respective Verschiebung maps.

\begin{prop}\label{propLTFunct}
Let $R$ be a $π$-adic $\mcO''$-algebra and let $κ\in \mcO''\tensor_{\mcO}W_{\mcO}(R)$ be a unit. Then the natural map of $\mcO''$-algebras
$$α:\mcO''\tensor_{\mcO}W_{\mcO}(R)→\mcO''\tensor_{\mcO'}W_{\mcO'}(R)$$
induces a strict morphism of $\mcO''$-frames
$$\mcL_{\mcO''/\mcO,κ}(R)→\mcL_{\mcO''/\mcO',α(κ)}(R).$$
In other words, $α$ commutes with the $\dot{σ}$-operators.
\end{prop}
\begin{proof}
Let us consider the Lubin-Tate $\mcO$-display $(\mcO''\tensor_{\mcO} W_{\mcO}(R), J_{\mcO}(R), F, \dot{σ})$ underlying the $\mcO''$-frame $\mcL_{\mcO''/\mcO,κ}(R)$. By \cite[Proposition 2.23]{ACZ}, there exists a Lubin-Tate $\mcO'$-display $(\mcO''\tensor_{\mcO'} W_{\mcO'}(R),J_{\mcO'}(R),F',\dot{σ}')$ over $R$ such that $α\circ \dot{σ} = \dot{σ}'\circ α$. The corresponding Lubin-Tate $\mcO''$-frame then equals $\mcL_{\mcO''/\mcO',α(κ)}(R)$ which proves the proposition.
\end{proof}

\subsection{Windows over Lubin-Tate frames}
\begin{prop}\label{propLTWindows}
Let $\mcO'/\mcO$ be a totally ramified extension of rings of integers in $p$-adic local fields. Let $R$ be a $π$-adic $\mcO'$-algebra and $κ\in \mcO'\tensor_{\mcO}W_{\mcO}(R)$ a unit.
Then there is an equivalence of categories
$$\{\text{strict formal }\mcO'\text{-modules over }R\}\iso \{\text{nilpotent}\ \mcL_{\mcO'/\mcO,κ}(R)\text{-windows}\}.$$
This equivalence is compatible with base change in $R$ and with base change along the morphisms of $\mcO'$-frames
$$\mcL_{\mcO'/\mcO,κ}(R)→\mcL_{\mcO'/\tilde{\mcO},\tilde{κ}}(R)$$
for intermediate extensions $\mcO\subset \tilde{\mcO}\subset \mcO'$.
\end{prop}
\begin{proof}
Let $X/R$ be a formal $\mcO$-module equipped with a strict $\mcO'$-action $ι:\mcO'→\End(X)$. Let $\mcP := (P,Q,F,\dot F)$ be its $\mcO$-display. Then $P$ is naturally an $\mcO'\tensor_{\mcO} W_{\mcO}(R)$-module and $J_{\mcO'}(R)P\subset Q$. Furthermore, the map $\dot F$ is a $σ$-linear epimorphism $Q→P$. Then there is at most one way to define a $σ$-linear operator $F':P→P$
which makes $(P,Q,F',\dot F)$ into an $\mcL_{\mcO'/\mcO,κ}(R)$-window, namely
$$F'(x) := κ^{-1}\dot F((π'\tensor 1 - 1\tensor [π']) x).$$
We need to verify that this $F'$ satisfies
\begin{equation}\label{ABC}
\dot{F}(ξx) = \dot{σ}(ξ) F'(x),\ \ ξ\in J_{\mcO'}(R),\ x\in P.
\end{equation}
It is enough to verify this for one single $κ$ since all other choices multiply both sides of the equation by a unit. So we choose
$$κ = V^{-1}(θ(π'\tensor 1 - 1\tensor[π']))$$
where $θ$ is an element as in Lemma \ref{lemMagic}. In other words, we work with the Lubin-Tate $\mcO'$-frame from Example \ref{exLT} (2).

Both sides in equation \eqref{ABC} are $σ$-linear, so it is enough to verify the relation for $ξ = (π'\tensor 1 - 1\tensor[π'])$ or $ξ\in I_{\mcO}(R)$. (These elements generate $J_{\mcO'}(R)$ as ideal.) The case $ξ = (π'\tensor 1 - 1\tensor [π'])$ is the definition of $F'$. In the case $ξ \in I_{\mcO}(R)$, we compute
$$\begin{aligned}
\dot{σ}(ξ)F'(x) & = V^{-1}(θξ)F'(x)\\
& = σ(θ)V^{-1}(ξ)F'(x)\\
& = σ(θ)V^{-1}(ξ)κ^{-1} \dot{F} ((π'\tensor 1 - 1\tensor [π'])x)\\
& = V^{-1}(ξ)κ^{-1} \dot{F} (θ(π'\tensor 1 - 1\tensor [π'])x)\\
& = V^{-1}(ξ) F(x) = \dot{F}(ξx).
\end{aligned}$$
In the last step, we used that $\mcP$ is an $\mcO$-display.

Thus we get a functor
$$\{\text{strict formal }\mcO'\text{-modules over }R\}→ \{\text{nilpotent }\mcL_{\mcO'/\mcO,κ}(R)\text{-windows}\}$$
which commutes with base change in $R$ and in the Lubin-Tate $\mcO'$-frame.

To prove that this functor is an equivalence, we construct its inverse. Again it suffices to do this in the special case of $κ = V^{-1}(θ(π'\tensor 1 - 1\tensor [π']))$. Given any $\mcL_{\mcO'/\mcO,κ}(R)$-window $(P,Q,F',\dot{F})$, we define a $σ$-linear operator $F:P→P$ by the formula
$$F(x) = F'(θx)$$
We only need to check that this defines an $\mcO$-display, i.e.\! that $\dot{F}(ξx) = V^{-1}(ξ)F(x)$ for all $ξ\in I_{\mcO}(R)$. But
$$\dot{F}(ξx) = V^{-1}(θξ)F'(x) = σ(θ)V^{-1}(ξ)F'(x) = V^{-1}(ξ)F(x).$$
The compatibility with base change along the morphisms of frames $α:\mcL_{\mcO'/\mcO,κ}(R)→\mcL_{\mcO'/\tilde{\mcO},\tilde {κ}}(R)$ is clear. Namely let $α_*\mcP = (P',Q',F',\dot F')$ be the base change of $\mcP$ and let $\mcP'' = (P'',Q'',F'',\dot F')$ be the $\mcL_{\mcO'/\tilde \mcO, \tilde {κ}}(R)$-window constructed from the $\tilde {\mcO}$-display of $(X,ι)$. Then
$$P' = (\mcO'\tensor_{\tilde{\mcO}} W_{\tilde{\mcO}}(R)) \tensor P = P''$$
by \cite[Definition 2.24]{ACZ} which relates the $\mcO$-display of $X$ and its $\tilde{\mcO}$-display. Furthermore, the submodules $Q'$ and $Q''$ agree under this identification. Now both $\dot F'$ and $\dot F''$ are determined by the condition that they agree with $\dot F$ on the image of $Q$. Since $F'$ and $F''$ are determined by $\dot F'$ and $\dot F''$, the windows $\mcP'$ and $\mcP''$ agree.
\end{proof}

\begin{cor}
The morphisms of $\mcO''$-frames from Proposition \ref{propLTFunct}
$$\mcL_{\mcO''/\mcO,κ}(R)→\mcL_{\mcO''/\mcO',κ'}(R)$$
are all crystalline, i.e.\! they induce equivalences on their categories of windows.
\end{cor}
\begin{proof}
This is just a reformulation of the fact that the equivalence in the previous proposition commutes with the base change along such morphisms of $\mcO''$-frames.
\end{proof}

\section{The unramified case}

For completeness, we also include the case of an unramified extension $\mcO'/\mcO$. Let $f$ be the degree of the extension. Again we fix a uniformizer $π\in \mcO$. For a $π$-adic $\mcO'$-algebra $R$, there exists a unique morphism
$$\mcO'→W_{\mcO}(R)$$
that lifts the given morphism $\mcO'→R$. In particular, there is a direct product decomposition
$$\mcO'\tensor_{\mcO}W_{\mcO}(R)\iso \prod_{\mbZ/f} W_{\mcO}(R).$$

\begin{defn}
For a $π$-adic $\mcO'$-algebra $R$, we define the $\mcO'$-frame
$$\mcA_{\mcO'/\mcO}(R):=(W_{\mcO}(R),I_{\mcO}(R),{}^{F^f},{}^{F^{f-1}}V^{-1}).$$
Windows over $\mcA_{\mcO'/\mcO}(R)$ are also called $f$-$\mcO$-displays, see \cite{ACZ}.
\end{defn}

In his thesis \cite{Ahs}, Ahsendorf constructs a functor
$$γ:\{\text{strict formal }\mcO'\text{-modules over }R\}→\{\text{nilpotent $f$-$\mcO$-displays over $R$}\}.$$
Furthermore, the natural morphism
$$W_{\mcO}(R)→W_{\mcO'}(R)$$
induces a strict morphism of $\mcO'$-frames
$$\mcA_{\mcO'/\mcO}(R)→\mcW_{\mcO'}(R)$$
and thus gives rise to a functor
$$δ:\{\text{$f$-$\mcO$-displays over $R$}\}→\{\mcO'\text{-displays over }R\}$$
that is compatible with duality by Lemma \ref{lemDual}.
Ahsendorf proves that the composition of these functors is an equivalence of categories. We slightly strengthen this result as follows.

\begin{prop}\label{propfOwindow}
Let $R$ be a noetherian $π$-adic $\mcO'$-algebra. Then the above functors $γ$ and $δ$ are both equivalences of categories (when restricted to the full subcategories of nilpotent windows).
\end{prop}
\begin{proof}
By Ahsendorf, the composition $δ\circ γ$ is an equivalence of categories, at least when restricted to the full subcategories of nilpotent windows. It is hence enough to prove that either of these functors is an equivalence. It would even be enough to just prove the faithfulness of $δ$. But for later use, we construct a quasi-inverse for $γ$. For this, we first recall the construction of this functor.

Let $\mcP=(P,Q,F,\dot{F})$ be the $\mcO$-display of a strict formal $\mcO'$-module over $R$. Then the natural map $\mcO'→W_{\mcO}(R)$ induces a $\mbZ/f$-grading
$$P = \bigoplus_{i\in \mbZ/f} P_i$$
such that both $F$ and $\dot{F}$ are homogeneous of degree $1$.
The strictness implies that $Q = Q_0\oplus P_1\oplus \ldots\oplus P_{f-1}.$ In particular, the restriction $\dot{F}_i := \dot{F}\vert_{Q_i}$ is an ${}^F$-linear isomorphism $P_i→P_{i+1}$ for $i = 1,\ldots, f-1$. The $f$-$\mcO$-display is now given by
$$(P_0,Q_0,\dot{F}^{f-1}\circ F\vert_{P_0}, \dot{F}^{f}\vert_{Q_0}).$$

Let us phrase this construction in terms of an $\mcO'$-stable normal decomposition $P = L\oplus T$. The $\mcO'$-stability is equivalent to the fact that both $L$ and $T$ are compatible with the grading, hence have the form
$$\begin{aligned}
L &= L_0\oplus P_1\oplus\ldots \oplus P_{f-1},\\
T &= T_0\oplus 0\oplus \ldots \oplus 0.\end{aligned}$$
Let $\Phi:= \dot{F}\vert_L \oplus F\vert_T = \oplus_{i\in \mbZ/f}\Phi_i$ be the ${}^F$-linear automorphism of $P$ associated to the normal decomposition. We use $\Phi_i$ to identify $P_{i+1}$ with $P_i^{(F)}$. Hence the display $\mcP$ together with its $\mcO'$-action can be describes as follows. The modules are of the form
$$\begin{aligned}
P &= P_0\oplus P_0^{(F)}\oplus P_0^{(F^2)}\oplus \ldots \oplus P_0^{(F^{f-1})},\\
Q &= Q_0\oplus P_0^{(F)}\oplus P_0^{(F^2)}\oplus \ldots \oplus P_0^{(F^{f-1})}
\end{aligned}$$
and the display structure is given by the normal decomposition
$$\begin{aligned}
L &= L_0\oplus P_0^{(F)}\oplus P_0^{(F^2)}\oplus \ldots \oplus P_0^{(F^{f-1})},\\
T &= T_0\oplus 0\oplus \ldots \oplus 0\end{aligned}$$
and the ${}^F$-linear operator
$$\Phi = \left(\begin{matrix} & & & & ϕ\\
1 & & & & \\
& 1 & & & \\
& & \ddots & & \\
& & & 1 & \\
\end{matrix}\right)$$
where $ϕ= \Phi_{f-1}.$

It is now clear how to invert this construction. Given an $\mcO'$-display $\mcP' = (P',Q',F',\dot F')$, we set $P_0 := P'$ and $Q_0:= Q'$. Then we define $P$ and $Q$ by the above formulas. If $(P' = L'\oplus T', ϕ)$ is a normal decomposition of $\mcP'$, then we set $L_0 := L'$, $T_0 := T'$ and define a normal decomposition and the operator $\Phi$ by the above formulas.
\end{proof}

\end{document}